\documentclass[12pt, a4 paper]{amsart}

\calclayout

\usepackage[english]{babel}
\usepackage[utf8]{inputenc}
\usepackage{amsmath}
\usepackage{graphicx}
\usepackage[colorinlistoftodos]{todonotes}
\usepackage{amssymb,amscd,latexsym,epsfig,xy}
\usepackage[mathscr]{euscript}
\usepackage{amsthm}
\usepackage{tikz}
\usepackage{MnSymbol}
\usetikzlibrary{matrix,arrows,decorations.pathmorphing}
\usepackage{hyperref}
\usepackage{enumerate}
\usepackage{float}
\usepackage{tikz-cd}
\usepackage{relsize}
\usepackage{mathtools}
\usepackage{dsfont}

\newtheorem{defn}{Definition}[section]
\newtheorem{thm}[defn]{Theorem}
\newtheorem{lem}[defn]{Lemma}
\newtheorem{cor}[defn]{Corollary}

\newtheorem{prop}[defn]{Proposition}
\theoremstyle{remark}
\newtheorem{remark}[defn]{Physics remark}

\newcommand\A{\mathbb A}

\newcommand\C{\mathbb C}

\newcommand\PP{\mathbb P}
\newcommand\R{\mathbb R}
\newcommand\bS{\mathbb S}
\newcommand\Q{\mathbb Q}

\newcommand\Z{\mathbb Z}
\newcommand{\kk}{\mathds{k}}
\newcommand\cA{\mathcal A}
\newcommand\cB{\mathcal B}
\newcommand\cC{\mathcal C}

\newcommand\cM{\mathcal M}
\newcommand\cN{\mathcal N}
\newcommand\cO{\mathcal O}

\newcommand\cT{\mathcal T}
\newcommand\cX{\mathcal X}

\newcommand\fd{\mathfrak d}
\newcommand\fD{\mathfrak D}

\newcommand\Spec{\operatorname{Spec}\,}

\newcommand\Hom{\operatorname{Hom}}

\newcommand\id{\operatorname{id}\,}

\newcommand\Tr{\operatorname{Tr}}

\newcommand\cl{\operatorname{cl}}

\newcommand\Int{\operatorname{Int}}

\newcommand\virt{\mathrm{virt}}
\newcommand\can{\mathrm{can}}

\newcommand\Sk{\operatorname{Sk}}
\newcommand\Ch{\operatorname{Ch}}
\newcommand\tr{\operatorname{tr}}
\newcommand\Spin{\operatorname{Spin}}

\usepackage{pict2e}

\makeatletter

\newcommand*\@KP@Large@frame[2]{%
    \setlength\unitlength{\fontdimen 22 #1\tw@}%
    \vrule \@width\z@ \@height 4\unitlength \@depth\tw@\unitlength
    \begin{picture}(6,2)(-3,-1)%
        \def\@KP@Radius     {3}%
        \def\@KP@Hole@radius{.5}
        \def\@KP@Diameter   {6}%
        #2%
    \end{picture}%
}
\newcommand*\@KP@Small@frame[2]{%
    \setlength\unitlength{\fontdimen 22 #1\tw@}%
    \vrule \@width\z@ \@height \thr@@\unitlength \@depth\@ne\unitlength
    \begin{picture}(4,2)(-2,-1)%
        \def\@KP@Radius     {2}%
        \def\@KP@Hole@radius{.5}
        \def\@KP@Diameter   {4}%
        #2%
    \end{picture}%
}

\newcommand*\@KP@Radius     {}
\newcommand*\@KP@Hole@radius{}
\newcommand*\@KP@Diameter   {}
\newcommand*\@KP@Shape@A{%
    \put(0,0){\circle{\@KP@Diameter}}%
}
\newcommand*\@KP@Shape@B{%
    \Line(-\@KP@Radius,\@KP@Radius )(\@KP@Radius,-\@KP@Radius)%
    \Line(-\@KP@Radius,-\@KP@Radius)(-\@KP@Hole@radius,-\@KP@Hole@radius)%
    \Line(\@KP@Radius ,\@KP@Radius )(\@KP@Hole@radius ,\@KP@Hole@radius )%
}
\newcommand*\@KP@Shape@C{%
    \cbezier(-\@KP@Radius,\@KP@Radius )(0,0)(0,0)(\@KP@Radius,\@KP@Radius )%
    \cbezier(-\@KP@Radius,-\@KP@Radius)(0,0)(0,0)(\@KP@Radius,-\@KP@Radius)%
}
\newcommand*\@KP@Shape@D{%
    \cbezier(-\@KP@Radius,-\@KP@Radius)(0,0)(0,0)(-\@KP@Radius,\@KP@Radius)%
    \cbezier(\@KP@Radius ,-\@KP@Radius)(0,0)(0,0)(\@KP@Radius ,\@KP@Radius)%
}

\newcommand*\@KP@Atomic@mathpalette[1]{%
    \mathinner{
        \mathchoice{%
            \linethickness{.6\p@}
            \@KP@Large@frame \textfont {#1}%
        }{%
            \linethickness{.4\p@}
            \@KP@Small@frame \textfont {#1}%
        }{%
            \linethickness{.3\p@}
            \@KP@Small@frame \scriptfont {#1}%
        }{%
            \linethickness{.2\p@}
            \@KP@Small@frame \scriptscriptfont {#1}%
        }%
    }%
}

\newcommand*\KPA{\@KP@Atomic@mathpalette \@KP@Shape@A}
\newcommand*\KPB{\@KP@Atomic@mathpalette \@KP@Shape@B}
\newcommand*\KPC{\@KP@Atomic@mathpalette \@KP@Shape@C}
\newcommand*\KPD{\@KP@Atomic@mathpalette \@KP@Shape@D}

\makeatother

\title[Strong positivity for skein algebras]{Strong positivity for the skein algebras of the $4$-punctured sphere
and of the $1$-punctured torus}

\author{Pierrick Bousseau}

\date{}

\begin{document}

\begin{abstract}
The Kauffman bracket skein algebra is a quantization of the algebra of regular functions on 
the $SL_2$ character variety of a topological surface.
We realize the skein algebra of the $4$-punctured sphere as the output of a mirror symmetry construction
based on higher genus Gromov-Witten theory and
applied to a complex cubic surface. 
Using this result, we prove the positivity of the structure constants 
of the bracelets basis for the skein algebras of the $4$-punctured sphere
and of the $1$-punctured torus. 
This connection between topology of the $4$-punctured sphere 
and enumerative geometry of curves in cubic surfaces 
is a mathematical manifestation of the existence of 
dual descriptions in string/M-theory for the $\cN=2$ $N_f=4$ $SU(2)$ gauge theory.
\end{abstract}

\maketitle

\setcounter{tocdepth}{2}

\tableofcontents

\thispagestyle{empty}

\section{Introduction}

In this paper, we address questions 
in low-dimensional topology using algebraic and geometric methods inspired by mirror symmetry. 
More precisely, we prove results on the topology 
of simple closed curves on the $4$-punctured sphere and the $1$-punctured torus by studying 
the a priori unrelated problem of counting holomorphic maps 
from Riemann surfaces to complex cubic surfaces. 
We present our results on positive bases 
for the skein algebras of the $4$-punctured sphere and $1$-punctured torus in Section
\ref{section_intro_results}. We give a survey of the 
proof, based on enumerative algebraic geometry, in Section \ref{section_intro_proof}. 
Motivations from theoretical physics are briefly discussed in 
Section \ref{section_physics}.

\subsection{Results on positive bases for $\Sk_A(\bS_{0,4})$ and $\Sk_A(\bS_{1,1})$}
\label{section_intro_results}

\subsubsection{Skein modules and algebras}

Recall that a  knot in a manifold is 
a connected compact embedded $1$-dimensional submanifold, 
and that
a link is the disjoint union of finitely many (possibly zero) knots. 
A framing of a link is a choice of 
nowhere vanishing section of its normal bundle.

The Kauffman bracket skein module 
of an oriented $3$-manifold $\mathbb{M}$ is the $\Z[A^{\pm}]$-module
generated by isotopy classes of framed links in 
$\mathbb{M}$ satisfying the skein relations 
\begin{equation}
\KPB= A \,\KPC + A^{-1}\, \KPD
\,\,\,\,\text{and} \,\,\,\,
L \cup \KPA = -(A^{2}+A^{-2}) \, L  \,.
\end{equation}
The diagrams in each relation indicate framed links 
that can be isotoped to identical embeddings except within the neighborhood shown, 
where the framing is vertical, i.e.\ pointing out to the reader. 
The Kauffman bracket skein module was
introduced independently by Przytycki
\cite{MR1194712} and Turaev \cite{MR964255} as an extension to general 
$3$-manifolds of the variant of the Jones polynomial
\cite{MR766964} given by the Kauffman bracket polynomial for framed links in the 
$3$-sphere \cite{MR899057}. In the general context of skein modules 
attached to arbitrary ribbon categories \cite{gunningham2019finiteness}, 
the Kauffman bracket skein module is 
associated to the ribbon category of finite-dimensional representations 
of the $SL_2$ quantum group.

Given an oriented $2$-manifold $\mathbb{S}$, 
one can define a natural algebra structure on the 
Kauffmann bracket skein module of the $3$-manifold 
$\mathbb{M} \coloneqq \mathbb{S} \times (-1,1)$: 
given two framed links $L_1$ and $L_2$ in $\mathbb{S} \times (-1,1)$, 
and viewing the interval $(-1,1)$ as a vertical direction, the product $L_1 L_2$
is defined by placing $L_1$ on top of $L_2$. 
We denote by $\Sk_A(\mathbb{S})$
the resulting associative $\Z[A^{\pm}]$-algebra with unit. 
The skein algebra
$\Sk_A(\mathbb{S})$ is in general non-commutative.

We consider the case where $\bS$ is the complement
$\bS_{g,\ell}$
 of a finite number $\ell$ of points in a compact oriented 2-manifold of genus $g$. 
 A multicurve on $\bS_{g,\ell}$ is the union of finitely many disjoint compact connected embedded $1$-dimensional 
submanifolds of $\bS_{g,\ell}$ such that none of them bounds a disc in $\bS_{g,\ell}$. 
Identifying 
$\bS_{g,\ell}$ with $\bS_{g,\ell} \times \{0\} \subset \bS_{g,\ell} \times (-1,1)$, 
a multicurve on $\bS_{g,\ell}$ endowed with the vertical framing naturally defined 
a framed link in 
$\bS_{g,\ell} \times (-1,1)$. By a result of 
Przytycki \cite[Theorem IX.7.1]{przytycki06},
isotopy classes of multicurves form a basis of 
$\Sk_A(\mathbb{S}_{g,\ell})$ as $\Z[A^{\pm}]$-module.

\subsubsection{Positivity of the bracelets basis of $\Sk_A(\bS_{0,4})$ and $\Sk_A(\bS_{1,1})$}
\label{section_positivity_bracelets}
Dylan Thurston introduced in \cite{MR3263305} 
a different basis $\mathbf{B}_T$ of $\Sk_A(\mathbb{S}_{g,\ell})$, 
called the 
bracelets basis and defined as follows.
Let $T_n(x)$ be the Chebyshev polynomials defined by 
\begin{equation}\label{eq:chebyshev}
T_0(x)=1\,,T_1(x)=x\,,
T_2(x)=x^2-2\,,\,\text{and for every}\,\, n \geq 2, \,\,
T_{n+1}(x)=xT_n(x)-T_{n-1}(x)\,.
\end{equation}
Writing $x=\lambda+\lambda^{-1}$, we have 
$T_n(x)=\lambda^n+\lambda^{-n}$ for every 
$n \geq 1$.
Given an isotopy class $\gamma$ of multicurve on $\bS_{g,\ell}$, 
one can uniquely write $\gamma$ in $\Sk_A(\bS_{g,\ell})$ as 
$\gamma=\gamma_1^{n_1} \cdots \gamma_r^{n_r}$
where $\gamma_1, \cdots , \gamma_r$ are all distinct isotopy classes 
of connected multicurves
and $n_j \in \Z_{>0}$, and we define \begin{equation}
\mathbf{T}(\gamma) 
\coloneqq T_{n_1}(\gamma_1) \cdots T_{n_r}(\gamma_r)\,.
\end{equation} 
 As the leading term of $T_n(x)$ is $x^n$,
the set
$\mathbf{B}_T$ 
of all $\mathbf{T}(\gamma)$, for $\gamma$ isotopy class of multicurve, 
is a $\Z[A^{\pm}]$-linear basis of 
$\Sk_A(\bS_{g,\ell})$. If $\gamma$ is a connected multicurve, 
$\gamma^n$ is the class of $n$ disjoint isotopic copies of $\gamma$, 
whereas $T_n(\gamma)$ is the class of a connected bracelet 
made of $n$ isotopic copies of $\gamma$ (see \cite[Proposition 4.4]{MR3263305}), 
hence the name of bracelets basis for $\mathbf{B}_T$.

In \cite[Conjecture 4.20]{MR3263305}, Dylan Thurston made 
the remarkable positivity conjecture that the structure constants of
the bracelets basis, which are a priori in 
$\Z[A^{\pm}]$, in fact belong to 
$\Z_{\geq 0}[A^{\pm}]$. He proved in \cite[Theorem 1]{MR3263305} 
that the conjecture holds after setting $A=1$.
In the present paper, we prove 
\cite[Conjecture 4.20]{MR3263305} in the case of the 
$4$-punctured sphere $\bS_{0,4}$, that is, 
$g=0$ and $\ell=4$, and the $1$-punctured torus $\bS_{1,1}$, that is, 
$g=1$ and $\ell=1$.

\begin{thm} \label{thm_main_intro}
The structure constants for the bracelets basis of the skein algebras $\Sk_A(\bS_{0,4})$ and 
$\Sk_A(\bS_{1,1})$ of the $4$-punctured sphere $\bS_{0,4}$ and of the 
$1$-punctured torus $\bS_{1,1}$  belong to $\Z_{\geq 0}[A^{\pm}]$.
In other words, for every $x$ and $y$ in $\mathbf{B}_T$, the product $x y$ in the skein algebra
 is a linear combination with coefficients in $\Z_{\geq 0}[A^{\pm}]$ of elements of 
$\mathbf{B}_T$. 
\end{thm}

\cite[Conjecture 4.20]{MR3263305} was previously known
in the following cases:
\begin{enumerate}
\item For $g=0$ and $\ell \leq 3$, the skein algebra 
is a commutative polynomial algebra, more precisely, 
we have $\Sk_A(\bS_{0,0})=\Z[A^{\pm}]$, 
$\Sk_A(\bS_{0,1})=\Z[A^{\pm}]$,
$\Sk_A(\bS_{0,2})=\Z[A^{\pm}][x]$, and  
$\Sk_A(\bS_{0,3})=\Z[A^{\pm}][x,y,z]$, and so 
\cite[Conjecture 4.20]{MR3263305} follows directly 
from the identity $T_m(x) T_n(x) = T_{m+n}(x)+ T_{|m-n|}(x)$.
\item For $g=1$ and $\ell=0$. For every $p=(a,b) \in \Z^2$,
write $\gamma_p$ for the isotopy class of 
$\gcd(a,b)$ disjoint copies of connected multicurves with
homology class $\frac{1}{\gcd(a,b)}(a,b)
\in \Z^2=H_1(\bS_{1,0},\Z)$.
Frohman and Gelca proved in \cite{MR1675190} 
the identity 
\[ \mathbf{T}(\gamma_{(a,b)})\mathbf{T}(\gamma_{(c,d)})
=A^{ad-bc}\mathbf{T}(\gamma_{(a+c,b+d)})
+A^{-ad+bc}\mathbf{T}(\gamma_{(a-c,b-d)})\,. \]
\cite[Conjecture 4.20]{MR3263305} follows because the bracelets basis of 
$\Sk_A(\bS_{1,0})$ is made of 
monomials in the variables $\mathbf{T}(\gamma_p)$.
\end{enumerate}

The cases $(g,\ell)=(0,4)$ and 
$(g,\ell)=(1,1)$ treated by Theorem 
\ref{thm_main_intro} are the first examples 
of a proof of \cite[Conjecture 4.20]{MR3263305} in a situation where no 
simple closed formula for the structure constants of the bracelets basis seems to exist.

A conceptual approach to the general case of 
\cite[Conjecture 4.20]{MR3263305} 
would be to construct a monoidal categorification of the skein algebras $\Sk(\bS_{g,\ell})$ 
and a categorification of the bracelets basis. 
First steps towards this goal are described by Queffelec and Wedrich in \cite{queffelec2018khovanov}. 
We do not follow this path to prove Theorem \ref{thm_main_intro}.
Rather, one should view Theorem \ref{thm_main_intro} 
as providing further non-trivial evidence that such monoidal categorification should exist.

For $\ell >0$, there is a more refined positivity conjecture,
\cite[Conjecture 4.21]{MR3263305}, involving the so-called bands basis.
We do not adress this conjecture in the present paper.
General constraints on possible positive bases of skein algebras are discussed by Lê
\cite{MR3801463} and
Lê, Thurston, and Yu
\cite{le2019lower}.
 
\subsubsection{A stronger positivity result for $\Sk_A(\bS_{0,4})$}
\label{section_stronger_04}
We will in fact prove a positivity result
for $\Sk_A(\bS_{0,4})$ stronger than Theorem 
\ref{thm_main_intro} and conjectured by 
Bakshi, Mukherjee, Przytycki, Silvero and Wang 
in \cite[Conjecture 4.10 (1)]{bakshi2018multiplying}.
For $1\leq j\leq 4$, let $p_j$ be the punctures of 
$\bS_{0,4}$, and $a_j$ the isotopy class of connected peripheral curves around $p_j$, 
that is, bounding a $1$-punctured disc with puncture $p_j$. The peripheral curves 
$a_j$ are in the center of the skein algebra $\Sk_A(\bS_{0,4})$, and so $\Sk_A(\bS_{0,4})$ is naturally a 
$\Z[A^{\pm}][a_1,a_2,a_3,a_4]$-module.

We fix a decomposition of $\bS_{0,4}$ into 
two pairs of pants, glued along a connected multicurve
$\delta$ of $\bS_{0,4}$
separating the four punctures into the pairs $p_1$, $p_2$ and $p_3$, $p_4$. 
Isotopy classes of multicurves on 
$\bS_{0,4}$ without peripheral components can 
then be classified by their Dehn-Thurston coordinates with respect to $\delta$ \cite{MR568308, MR1144770}. 
For every $p=(m,n) \in \Z \times \Z_{\geq 0}$ such that $m \geq 0$ if $n=0$, 
there exists a unique isotopy class $\gamma_p$ of multicurves without peripheral components, 
such that, the minimal number of intersection points 
of a multicurve of class $\gamma_p$ with $\delta$ is $2n$, 
and such that the twisting number of $\gamma_p$ around $\delta$ is $m$. 
As a special case of a theorem of Dehn,
the map 
$p \mapsto \gamma_p$ defines a bijection between 
\begin{equation} \label{eq:BZ}
B(\Z) \coloneqq  \{(m,n) \in \Z \times \Z_{\geq 0}\,\,\,|\, m \geq 0 \,\,\,
\text{if}\,\,\, n=0\} 
\end{equation}
and the set of isotopy classes of multicurves on 
$\bS_{0,4}$ without peripheral components,
see \cite[Theorem 1.2.1]{MR1144770}. 
For example, $\gamma_{(0,0)}$ is the isotopy class of the empty multicurve, 
$\gamma_{(1,0)}$ is the isotopy class of $\delta$, and a multicurve of class 
$\gamma_p$ with $p=(m,n)$ has $\gcd(m,n)$ connected components.
Equivalently, if $p=(m,n)$ with $m$ and $n$ coprime, and
if we realize 
$\bS_{0,4}$ as the quotient
of the four-punctured torus $(\R^2 \backslash(\frac{1}{2}\Z \oplus \frac{1}{2} \Z))/\Z^2$
 by the involution $x \mapsto -x$ , then 
$\gamma_p$ is the class of the image
in $\bS_{0,4}$ of a straight line of slope $n/m$ in $\R^2 \backslash (\frac{1}{2}\Z \oplus \frac{1}{2} \Z)$ (e.g.\ see \cite[Proposition 2.6]{MR2850125}).
As isotopy classes of multicurves form a basis of the skein algebra as 
$\Z[A^{\pm}]$-module, the set
$\{ \gamma_p \}_{p \in B(\Z)}$ is a basis of $\Sk_A(\bS_{0,4})$ as 
$\Z[A^{\pm}][a_1,a_2,a_3,a_4]$-module.

For every $p_1,p_2,p \in B(\Z)$, 
we define structure constants $C_{p_1,p_2}^{\bS_{0,4},p} \in \Z[A^{\pm}][a_1,a_2,a_3,a_4]$ by 
\begin{equation} \label{eq:str_cst_0_4}
\mathbf{T}(\gamma_{p_1}) \mathbf{T}(\gamma_{p_2})
=\sum_{p \in B(\Z)} C_{p_1,p_2}^{\bS_{0,4}, p} \mathbf{T}(\gamma_p) \,.
\end{equation}
Following \cite{bakshi2018multiplying}, we introduce the notation
\begin{equation}\label{eq:R}
R_{1,0}\coloneqq a_1a_2+a_3a_4 \,,\,\,\,
 R_{0,1} \coloneqq a_1a_3+a_2a_4 \,,\,\,\,
 R_{1,1} \coloneqq a_1a_4+a_2a_3\,,
\end{equation}
\begin{equation} \label{eq:y}
 y \coloneqq a_1a_2a_3a_4
+a_1^2+a_2^2+a_3^2+a_4^2+(A^2-A^{-2})^2 \,.
\end{equation}
The following Theorem \ref{thm_main_intro_2} 
is our main result and proves Conjecture 4.10(1) of
\cite{bakshi2018multiplying}.

\begin{thm} \label{thm_main_intro_2}
For every $p_1,p_2,p \in B(\Z)$, we have 
\begin{equation} C^{\bS_{0,4},p}_{p_1,p_2} \in \Z_{\geq 0}[A^{\pm}]
[R_{1,0}, R_{0,1}, R_{1,1},y]\,.
\end{equation}
\end{thm}

As we will see at the end of Section \ref{section_scattering_04}, 
it is elementary to check that Theorem \ref{thm_main_intro_2} implies Theorem 
\ref{thm_main_intro} for $\Sk_A(\bS_{0,4})$.

\subsubsection{A stronger positivity result for $\Sk_A(\bS_{1,1})$}

Let $\eta$ be the isotopy class of connected peripheral curves around the puncture of $\bS_{1,1}$.
 As $\eta$ is in the center of $\Sk_A(\bS_{1,1})$, 
 the skein algebra $\Sk_A(\bS_{1,1})$ is naturally a $\Z[A^{\pm}][\eta]$-module. 
 Isotopy classes multicurves
on $\bS_{1,1}$ without peripheral components are classified by their homology classes, 
which are well-defined up to sign. 
Fixing a basis of homology, we get a bijection $p \mapsto \gamma_p$ between $B(\Z)$ 
and the set of isotopy classes of multicurves on $\bS_{1,1}$ without peripheral components. 
For example,  multicurve of class $\gamma_p$ with $p=(m,n)$ has $\gcd(m,n)$ components.
As isotopy classes of multicurves form a basis of the skein algebra as $\Z[A^{\pm}]$-module,
the set
$\{ \gamma_p \}_{p \in B(\Z)}$ is a basis of $\Sk_A(\bS_{1,1})$ as 
$\Z[A^{\pm}][\eta]$-module.
For every $p_1,p_2,p \in B(\Z)$, we define structure constants $C_{p_1,p_2}^{\bS_{1,1},p} \in \Z[A^{\pm}][\eta]$ by 
\begin{equation} \label{eq:str_cst_1_1} 
\mathbf{T}(\gamma_{p_1}) \mathbf{T}(\gamma_{p_2})
=\sum_{p \in B(\Z)} C_{p_1,p_2}^{\bS_{1,1}, p} \mathbf{T}(\gamma_p) \,.
\end{equation}
We write 
\begin{equation} 
z \coloneqq A^2+A^{-2}+\eta\,.
\end{equation}
Note that $z$ is the deformation parameter from $\Sk_A(\bS_{1,0})$ to 
$\Sk_A(\bS_{1,1})$: 
indeed, closing the puncture means setting $\eta=-A^2-A^{-2}$, that is $z=0$.

\begin{thm} \label{thm_main_intro_4}
For every $p_1,p_2,p \in B(\Z)$, we have 
\begin{equation} C^{\bS_{1,1},p}_{p_1,p_2} \in \Z_{\geq 0}[A^{\pm}][z]\,.
\end{equation}
\end{thm}

As we will see at the end of Section \ref{section_scattering_11}, 
it is elementary to check that Theorem \ref{thm_main_intro_4} implies Theorem 
\ref{thm_main_intro} for $\Sk_A(\bS_{0,4})$.

\subsubsection{Strong positivity for the quantum cluster algebras $\cX^q_{PGL_2,\bS_{0,4}}$ and $\cX^q_{PGL_2,\bS_{1,1}}$}
\label{section_cluster_intro}
We can appply our positivity result on the skein algebras $\Sk_A(\bS_{0,4})$
and $\Sk_A(\bS_{1,1})$, Theorem \ref{thm_main_intro_2},
to prove a similar positivity result 
for the quantum cluster algebras $\cX^q_{SL_2,\bS_{0,4}}$
and $\cX^q_{SL_2,\bS_{1,1}}$.

For every punctured surface $\bS_{g,\ell}$ with $\ell >0$, 
Fock and Goncharov introduced in \cite{MR2233852}
the cluster varieties $\cA_{SL_2,\bS_{g,\ell}}$ 
and $\cX_{PGL_2,\bS_{g,\ell}}$: $\cA_{SL_2,\bS_{g,\ell}}$ 
is a moduli space of decorated $SL_2$-local system on $\bS_{g,\ell}$, 
and $\cX_{PGL_2,\bS_{g,\ell}}$ is a moduli space of 
framed $PGL_2$-local systems on $\bS_{g,\ell}$ 
and both admit a cluster structure. Fock and Goncharov constructed a ``duality map"
\begin{equation}
\mathbb{I} \colon \cA_{SL_2,\bS_{g,\ell}}(\Z^t) \longrightarrow \cO(\cX_{PGL_2,\bS_{g,\ell}}) 
\end{equation}
from the set  $\cA_{SL_2,\bS_{g,\ell}}(\Z^t)$ 
of integral tropical points of $\cA_{SL_2,\bS_{g,\ell}}$ 
to the algebra $\cO(\cX_{PGL_2,\bS_{g,\ell}})$ 
of regular functions on $\cX_{PGL_2,\bS_{g,\ell}}$. 
They proved that $\{\mathbb{I}(l)\}_{l \in \cA_{SL_2,\bS_{g,\ell}}(\Z^t)}$ 
is a basis of $\cO(\cX_{PGL_2,\bS_{g,\ell}})$
(\cite[Theorem 12.3]{MR2233852}) 
with positive structure constants (\cite[Theorem 12.2]{MR2233852}). 

The cluster variety $\cX_{PGL_2,\bS_{g,\ell}}$ admits a natural Poisson structure, 
which can be canonically quantized using the cluster structure 
to produce a quantum cluster algebra $\cX^q_{PGL_2,\bS_{g,\ell}}$ \cite{MR2567745}.
Fock and Goncharov conjectured in \cite[Conjecture 12.4]{MR2233852} 
the existence of a quantization 
\begin{equation} \hat{\mathbb{I}} \colon \cA_{SL_2,\bS_{g,\ell}}(\Z^t) \longrightarrow \cX^q_{PGL_2,\bS_{g,\ell}}
 \end{equation}
of $\mathbb{I}$ with structure constants in $\Z_{\geq 0}[q^{\pm \frac{1}{2}}]$, 
where $q$ is the quantum parameter.
Note that to be consistent with the rest of the paper, 
we denote by $q^{\frac{1}{2}}$ the parameter denoted by $q$ in \cite{MR2233852}
and \cite{MR3581328}.
The skein algebra $\Sk_A(\bS_{g,\ell})$ and the quantum cluster variety $\cX^q_{PGL_2,\bS_{g,\ell}}$ are closely related, and in fact \cite[Conjecture 12.4]{MR2233852} was a strong motivation \cite[Conjecture 4.20]{MR3263305}. 
A precise relation between $\Sk_A(\bS_{g,\ell})$ and $\cX^q_{PGL_2,\bS_{g,\ell}}$ was established by Bonahon and Wong \cite{MR2851072} and then used by Allegretti and Kim \cite{MR3581328} to construct a quantum duality map $\hat{\mathbb{I}}$ with the expected properties, excepted the positivity of the structure constants left as a conjecture. 
A different construction of $\hat{\mathbb{I}}$ based on spectral networks was given by Gabella \cite{MR3613514}
 and shown to be equivalent to the one of Allegretti and Kim by 
Kim and Son
\cite{kim2018rm}.
We first remark that the positivity of the stucture constants 
of the bracelets basis of the skein algebra 
$\Sk_A(\bS_{g,\ell})$ 
implies the positivity of the structure constants defined by $\hat{\mathbb{I}}$.

\begin{thm} \label{thm_main_intro_3}
Assume that the structure constants of 
the bracelets basis of the skein algebra $\Sk_A(\bS_{g,\ell})$ belong to 
$\Z_{\geq 0}[A^{\pm}]$.
Then the structure constants $c(l,l',l'') \in \Z[q^{\pm \frac{1}{2}}]$
for $\cX^q_{PGL_2,\bS_{g,\ell}}$,
defined by the quantum duality map $\hat{\mathbb{I}}$ of
\cite{MR3581328} via 
\begin{equation} \hat{\mathbb{I}}(l)\hat{\mathbb{I}}(l')=
\sum_{l'' \in \cA_{SL_2,\bS_{0,4}}(\Z^t)}
c(l,l',l'') \hat{\mathbb{I}}(l'')\,,
\end{equation}
belong to $\Z_{\geq 0}[q^{\pm \frac{1}{2}}]$. 
\end{thm}

The proof of Theorem \ref{thm_main_intro_3}
is given in Section \ref{section_cluster}.
Combining Theorem \ref{thm_main_intro_3} with Theorems 
\ref{thm_main_intro_2} and \ref{thm_main_intro_4}, we obtain the 
following corollary.

\begin{cor} \label{cor_intro}
The structure constants
defined by the quantum duality map $\hat{\mathbb{I}}$ of
\cite{MR3581328} for $\cX^q_{PGL_2,\bS_{0,4}}$ and
$\cX^q_{PGL_2,\bS_{1,1}}$
belong to $\Z_{\geq 0}[q^{\pm \frac{1}{2}}]$.
\end{cor} 

\subsection{Structure of the proof: quantum scattering diagrams and curve counting}
\label{section_intro_proof}

We will prove Theorems \ref{thm_main_intro_2}
and \ref{thm_main_intro_4}
by giving an algorithm which computes the structure constants for the bracelets basis of 
$\Sk_A(\bS_{0,4})$ 
and $\Sk_A(\bS_{1,1})$, and makes manifest their positivity properties. 
This algorithm is based on the notion of quantum broken lines 
defined by a quantum scattering diagram.

\subsubsection{Quantum scattering diagrams, quantum broken lines and quantum theta functions}
Scattering diagrams and broken lines are algebraic and 
combinatorial objects playing a key role in the Gross-Siebert 
approach to mirror symmetry. Scattering diagrams were introduced by Gross and Siebert
\cite{MR2846484}, following early insights of Kontsevich and Soibelman
\cite{MR2181810}. Broken lines were introduced by Gross
\cite{MR2600995}, studied by Carl, Pumperla and Siebert \cite{cps}, 
and discussed in a quite general context by Gross, Hacking, and Siebert \cite{gross2016theta}. 
Given an integral affine manifold with singularities $B$, 
a scattering diagram $\fD^{\cl}$ is a collection of codimension $1$ integral affine subspaces of $B$ 
called walls and which are decorated by power series. A broken line is a continuous picewise integral affine 
line in $B$ which bends when crossing 
walls of $\fD^{\cl}$. 
When the scattering diagram $\fD^{\cl}$ is so called consistent, 
one can construct a commutative associative algebra $\cA_{\fD^{\cl}}$, coming with a basis
$\{\vartheta_p^{\cl}\}_{p \in B(\Z)}$ of so-called theta functions indexed by integral points $B(\Z)$ of $B$, 
and whose structure constants are determined explicitly in terms of the broken lines.

Scattering diagrams and broken lines have $q$-deformed versions, that we
refer to as quantum scattering diagrams and quantum broken lines. 
Quantum scattering diagrams were considered by Kontsevich and Soibelman
\cite{MR2181810, MR2596639, kontsevich2013wall},
and Filippini and Stoppa
\cite{MR3383167}. Quantum broken lines were studied by 
Mandel \cite{mandel2015scattering} and the author
\cite{MR4048291}. Given a consistent quantum scattering diagram $\fD$, 
one can construct an associative
non-necessarily commutative algebra $\cA_\fD$, coming with a basis
$\{\vartheta_p\}_{p \in B(\Z)}$ of so-called quantum theta functions, 
and whose structure constants are determined explicitly in terms of the quantum broken lines.

Scattering diagrams and broken lines have been used by Gross, Hacking, Keel and Kontsevich \cite{GHKK18} to construct canonical bases with positive structure constants for cluster algebras. Their quantum versions have been used more recently by Davison and Mandel
\cite{DaM19} to construct canonical bases with structures constants in $\Z_{\geq 0}[q^{\frac{1}{2}}]$
for quantum cluster algebras.
It is expected that the canonical basis of \cite{GHKK18} coincides with the canonical basis constructed by Fock and Goncharov \cite{MR2233852}, and that the canonical basis of \cite{DaM19} 
for  $\cX^q_{PGL_2,\bS_{g,\ell}}$ agrees with the canonical basis constructed by Allegretti and Kim \cite{MR3581328}. Proving these conjectural expectations would lead to a general proof of the quantum positivity conjecture \cite[Conjecture 12.4]{MR2233852}. In the present paper, we use quantum scattering diagrams which are slightly different from the ones in \cite{DaM19}, but  related to them by 
``moving worms". The positivity properties of our quantum scattering diagrams will follow from their explicit descriptions, and we will not have to use the general quantum positivity result of \cite{DaM19}.

In order to prove Theorem \ref{thm_main_intro_2}
and \ref{thm_main_intro_4}, we will first define explicit quantum scattering diagrams $\fD_{0,4}$
and $\fD_{1,1}$ over the integral affine manifold with singularities $B=\R^2/\langle -\id \rangle$ and prove that they are consistent. 
We will then show that
the algebras $\cA_{\fD_{0,4}}$
and $\cA_{\fD_{1,1}}$ are respectively isomorphic to the skein algebras 
$\Sk_A(\bS_{0,4})$ and $\Sk_A(\bS_{1,1})$, and that the bases of quantum theta functions 
agree with the bracelets bases. 
The positivity of the structure constants 
will follow from the description in terms of quantum broken lines 
and from the explicit definitions of $\fD_{0,4}$
and $\fD_{1,1}$.
As the results for the $1$-punctured torus $\bS_{1,1}$ will follow from those 
for $\bS_{0,4}$ by specialization and change of variables, 
we focus on the case of the $4$-punctured sphere $\bS_{0,4}$. 
There are two results to show: the consistency of $\fD_{0,4}$ 
(Theorem \ref{thm:consistent_0_4}), and the identification of $\cA_{\fD_{0,4}}$ with 
$\Sk_A(\bS_{0,4})$ matching the basis of quantum theta functions 
with the bracelets basis (Theorem \ref{thm_ring_isom}).

\subsubsection{Consistent quantum scattering diagrams from curve counting}
We will prove the consistency of $\fD_{0,4}$ by 
showing that $\fD_{0,4}$ 
arises from the enumerative geometry 
of holomorphic curves in complex cubic surfaces.
It is a general expectation from mirror symmetry 
that one should obtain consistent scattering diagrams by counting genus $0$
holomorphic curves in log Calabi-Yau varieties, 
see the work of Gross, Pandharipande and Siebert
\cite{MR2667135} and Gross, Hacking
and Keel \cite{MR3415066} in dimension $2$, and Gross and Siebert
\cite{MR3821173}, Keel and Yu
\cite{keel2019frobenius}, and Argüz and Gross
\cite{arguz2020higher} in higher dimensions.
Given a maximal log Calabi-Yau variety
$(Y,D)$,
that is, the pair of a smooth projective variety 
$Y$ over $\C$ and of an anticanonical normal crossing divisor $D$
with a $0$-dimensional stratum,
one can construct a consistent canonical scattering diagram
$\fD_\can^{\cl}$ by counting holomorphic maps from genus $0$ 
holomorphic curves to $Y$ whose images intersect $D$ at a single point
\cite{MR3415066, MR3821173}.
More precisely, these counts of holomorphic curves 
are defined using logarithmic Gromov-Witten theory \cite{MR3011419, MR3257836}.
The corresponding algebra $\cA_{\fD_\can^{\cl}}$ is then 
the algebra of functions on the family of varieties mirror to $(Y,D)$. Heuristically, the integral affine manifold with singularities $B$ containing $\fD_\can^{\cl}$
should be the basis of a special Lagrangian torus fibration on
the complement of $D$ in $Y$ \cite{MR1429831,MR2386535}. 

For $(Y,D)$ a maximal log Calabi-Yau surface, 
we explained in \cite{MR4048291} how to construct 
a consistent canonical quantum scattering diagram $\fD_\can$ 
in terms of log Gromov-Witten counts 
of holomorphic maps from higher genus holomorphic curves to $Y$ 
whose images intersect $D$ at a single point. 
The corresponding non-commutative algebra 
$\cA_{\fD_\can}$ is a deformation quantization 
of the mirror family of holomorphic symplectic surfaces constructed in \cite{MR3415066}.
The main idea of the present paper is to apply the framework of \cite{MR4048291} 
for $Y$ a smooth cubic surface and $D$ a triangle of lines on $Y$.
Before giving more details, we need to review the 
general relation between skein algebras and character varieties.

\subsubsection{Skein algebras and character varieties}
\label{section_skein_character}
Let $\Ch_{SL_2}(\bS_{g,\ell})$ be the $SL_2$-character variety of the 
$\ell$-punctured genus $g$ surface 
$\bS_{g,\ell}$. This is an affine variety of finite type over 
$\Z$ obtained as affine GIT quotient by the $SL_2$ 
conjugation action of the affine variety of representations of the fundamental group
$\pi_1(\bS_{g,\ell})$ into $SL_2$. 
The character variety $\Ch_{SL_2}$ admits a natural Poisson structure.

Setting 
$A=-e^{\frac{i \hbar}{4}}$, the skein algebra 
$\Sk_A(\bS_{g,\ell})$ defines a deformation quantization 
of the Poisson variety $\Ch_{SL_2}(\bS_{g,\ell})$.
If $\gamma$ is a multicurve on $\bS_{g,\ell}$ 
with connected components $\gamma_1, \cdots, \gamma_r$, 
then, the map sending a representation $\rho \colon \pi_1(\bS_{g,\ell})
\rightarrow SL_2$  to $\prod_{j=1}^r (-\tr (\rho(\gamma_j)))$ 
defines a regular
function $f_\gamma$
on
$\Ch_{SL_2}(\bS_{g,\ell})$.
The map $\gamma \mapsto f_\gamma$ defines a ring isomorphism between the
specialization 
$\Sk_{-1}(\bS_{g,\ell})$ of the skein algebra at $A=-1$ 
and the ring of regular functions of $\Ch_{SL_2}(\bS_{g,\ell})$.
If $\gamma$ is a connected multicurve on 
$\bS_{g,\ell}$, then the building blocks 
$T_n(\gamma)$ of the bracelets basis are
quantizations of the functions $\rho \mapsto -\tr(\rho(\gamma)^n)$ on 
$\Ch_{SL_2}(\bS_{g,\ell})$.

The general idea of a connection between skein algebras 
and quantization goes back to Turaev \cite{MR1142906}.
Bullock
\cite{MR1600138} and Przytycki and Sikora with a different 
proof \cite{MR1710996} showed that $\gamma \mapsto f_\gamma$ 
defines a ring isomorphism between the quotient of $\Sk_{-1}(\bS_{g,\ell})$
 by its nilradical and the ring of regular functions of $\Ch_{SL_2}(\bS_{g,\ell})$. 
 The fact that the nilradical of $\Sk_{-1}(\bS_{g,\ell})$ is trivial 
 was shown by Charles and Marché for $\ell=0$
\cite[Theorem 1.2]{MR2914854},
and by Przytycki and Sikora
\cite{MR3885180} in general. 

\subsubsection{Curve counting in cubic surfaces}
It is classically known that the $SL_2$-character variety 
$\Ch_{SL_2}(\bS_{0,4})$
of the $4$-punctured sphere $\bS_{0,4}$ can be described 
explicitly as a $4$-parameters family of affine cubic surfaces:
original 19\textsuperscript{th} century sources are \cite{MR1508833, fricke1896theorie}, \cite[II, Eq.(9), p298]{MR0183872} and more recent references 
include \cite{MR558891, MR688949, MR1685040, MR2497777, goldman2010affine}.
Recently, Gross, Hacking, Keel and Siebert \cite{gross2019cubic}
proved that this family of cubic surfaces 
is the result of the general mirror construction
of \cite{MR3415066}
for maximal log Calabi-Yau surfaces applied 
to a pair $(Y,D)$, where $Y$ is a smooth projective cubic surface in $\PP^3$ 
and $D$ is a triangle of lines on $Y$. 
In other words, they showed that the algebra obtained 
from the consistent canonical scattering diagram defined by counting genus $0$
holomorphic curves in $(Y,D)$ is exactly the algebra of regular functions on $\Ch_{SL_2}(\bS_{0,4})$.

Thus, we now have two ways to produce a deformation quantization of 
$\Ch_{SL_2}(\bS_{0,4})$ and it is natural to compare them: 
either consider the skein algebra $\Sk_A(\bS_{0,4})$, 
or consider the algebra $\cA_{\fD_\can}$ 
obtained from the consistent canonical quantum scattering diagram 
$\fD_\can$ defined in \cite{MR4048291} by counting higher genus holomorphic curves in 
$(Y,D)$.

First of all, we will compute explicitly the quantum scattering diagram $\fD_\can$. 
It involves computing higher genus log Gromov-Witten invariants of $(Y,D)$. 
The corresponding calculation in genus $0$ was done in 
\cite{gross2019cubic}: exploiting a
large $PSL_2(\Z)$ group of birational automorphisms of $(Y,D)$, 
Gross, Hacking, Keel and Siebert showed that the genus $0$ 
calculation can be reduced to genus $0$ multiple covers of $8$ lines and $2$ conics in 
$(Y,D)$. Following the same strategy, 
we will prove that the higher genus calculation reduces to higher genus multiple covers of the same $8$ lines and $2$ conics. 
The contribution of multiple covers of lines is fairly standard 
but the contribution of multiple covers of the conics is more intricate 
and we will use our previous work
\cite{bousseau2018quantum_tropical} on higher genus log Gromov-Witten 
invariants of log Calabi-Yau surfaces to evaluate it. 
At the end of the day, we can phrase the result 
as stating that $\fD_\can$ is equal (after an appropriate specialization of variables) 
to the explicit quantum scattering diagram $\fD_{0,4}$. 
As $\fD_\can$ is consistent by \cite{MR4048291}, 
this proves the consistency of $\fD_{0,4}$.

\subsubsection{Comparison of $\cA_{\fD_{0,4}}$ and $\Sk_A(\bS_{0,4})$} 

Once we know that the quantum scattering diagram $\fD_{0,4}$ is consistent, 
we have the corresponding algebra $\cA_{\fD_{0,4}}$, 
with its basis of quantum theta functions
$\{ \vartheta_p \}_{p \in B(\Z)}$ 
and structure constants expressed in terms of quantum broken lines. 
It remains to construct an isomorphism of algebras
$\varphi \colon \cA_{\fD_{0,4}}
\rightarrow \Sk_A(\bS_{0,4})$ matching the bracelets basis $\{\mathbf{\gamma}_p \}_{p \in B(\Z)}$ and the basis of quantum theta functions $\{ \vartheta_p\}_{p \in B(\Z)}$, 
i.e.\ such that $\varphi(\vartheta_p)=
\mathbf{T}(\gamma_p)$ for every $p \in B(\Z)$.

By explicit computations with quantum broken lines in $\fD_{0,4}$, 
we will obtain an explicit presentation of $\cA_{\fD_{0,4}}$ 
by generators and relations as a family of non-commutative cubic surfaces (Theorem \ref{thm_eq_cubic_nu}). 
On the other hand, it was known since Bullock and Przytycki \cite{MR1625701}
that the description of $\Ch_{SL_2}(\bS_{0,4})$ as a family of cubic surfaces 
deforms into a presentation of the skein algebra 
$\Sk_A(\bS_{0,4})$ as a family of non-commutative cubic surfaces 
(Theorem \ref{thm:sk_cubic}). Comparing these two families 
of non-commutative cubic surfaces, we
will define an isomorphism of algebras
$\varphi \colon \cA_{\fD_{0,4}}
\rightarrow \Sk_A(\bS_{0,4})$. 

Finally, we will have to prove that
$\varphi(\vartheta_p)=
\mathbf{T}(\gamma_p)$ for every $p \in B(\Z)$. We will first prove it
for $p=(k,0)$ by some explicit computation of quantum broken lines. 
In particular, we will see how the recursion relation 
\eqref{eq:chebyshev}
defining the Chebyshev polynomials $T_n(x)$ naturally arises 
from drawing quantum broken lines. To prove the general result, 
we will check explicitly that $\varphi$ intertwines the natural action of $PSL_2(\Z)$ on $\Sk_A(\bS_{0,4})$ via the mapping class group of $\bS_{0,4}$, 
with an action of $PSL_2(\Z)$ of $\cA_{\fD_{0,4}}$
coming from a $PSL_2(\Z)$-symmetry of 
the quantum scattering diagram $\fD_{0,4}$.
This ends our summary of the proof.

We remark that by taking the classical limit 
of the statement that $\varphi(\vartheta_p)=
\mathbf{T}(\gamma_p)$ for every $p \in B(\Z)$, 
we obtain that the classical theta functions $\vartheta_p^{\cl}$ constructed in
\cite{MR3415066, gross2019cubic}
agree with the trace functions 
$\rho \mapsto -\tr(\rho(\gamma_{p_{prim}})^k)$ on the character variety 
$\Ch_{SL_2}(\bS_{0,4})$, where $p=kp_{prim}$ with $k \in \Z_{\geq 1}$ and $p_{prim}
\in B(\Z)$ primitive (Corollary \ref{cor_traces}).

\subsubsection{More on non-commutative cubic surfaces}

We briefly comment about works related to an essential ingredient of the proof 
of our main result: the presentation of $\Sk_
A(\bS_{0,4})$ as a family of non-commutative cubic surfaces. 
This non-commutative cubic equation has appeared in quite a number of contexts. 
The present paper provides one more: the non-commutative cubic surface appears 
for us as a quantum mirror in the sense of \cite{MR4048291} 
and as the result of calculations in higher genus log Gromov-Witten theory.

The quantization of the family of affine cubic surfaces $\Ch_{SL_2}(\bS_{0,4})$ 
from the point of view of quantum Teichmüller theory has been studied 
by Chekhov and Mazzocco \cite[Eq. (3.20)-(3.24)]{MR2733812},  
and by Hiatt \cite{MR2661526}.
Quantization from the cluster point of view has been 
discussed by Hikami \cite[Eq. (7.2)-(7.3)]{MR3901894}. 
The general relation between skein algebras and the quantum Teichmüller/cluster points of view 
follows from the existence of the quantum trace map of Bonahon and Wong
\cite{MR2851072} (see also \cite{MR3827810}).

The skein algebra $\Sk_A(\bS_{0,4})$ is 
isomorphic to the spherical double affine Hecke algebra (DAHA) of type 
$(C^{\vee}_1, C_1)$ defined in
\cite{MR1715325, MR2085854,  MR2006574}.  
The explicit connection between the spherical DAHA of type $(C^{\vee}_1, C_1)$ 
and the quantization of cubic surfaces was established by Oblomkov
\cite{MR2037756}. Terwilliger \cite[Proposition 16.4]{MR3116183} 
wrote down an explicit presentation of the spherical 
DAHA of type $(C^{\vee}_1, C_1)$ from which the isomorphism with $\Sk_A(\bS_{0,4})$ is clear.
A much earlier appearance of the non-commutative cubic surface 
is the Askey-Wilson algebra $AW(3)$
of Zhedanov \cite{MR1151381}.
A comparison between $AW(3)$ and the  spherical DAHA of type $(C^{\vee}_1, C_1)$ 
was done by Koornwinder \cite{MR2425640}.
More details on the relation between the skein and DAHA points of view can be found in  
\cite[Section 2]{MR3530443}, \cite[Section 2]{MR3765471}, \cite{MR4001842}.

Skein algebras can also be considered in the framework of $SL_2$-factorization homology. 
Explicit presentations of $\Sk_A(\bS_{0,4})$ and 
$\Sk_A(\bS_{1,1})$ as non-commutative cubic surfaces 
are recovered using this point of view by Cooke
\cite{cooke2018kauffman}.

\subsection{Line operators and BPS spectrum of the $\cN=2$ $N_f=4$ $SU(2)$ gauge theory}

\label{section_physics}

In this section, which can be ignored by a 
purely mathematically minded reader, we briefly discuss the 
string/M-theoretic motivation for a connection between the skein algebra 
$\Sk_A(\bS_{0,4})$ and the enumerative geometry of curves in cubic surfaces.

Let $\mathcal{T}$ be a four-dimensional quantum field theory 
with $\cN=2$ supersymmetry. Such theory has in general 
an interesting dynamics connecting its short-distance behaviour (UV) with its long-distance behaviour (IR).
The IR behaviour of $\cT$ is largely determined by its Seiberg-Witten geometry
$\nu \colon \cM \rightarrow B$ \cite{MR1293681, MR1306869}, described as follows.
The special Kähler manifold with singularities $B$ 
is the Coulomb branch of the moduli space of vacua of $\cT$ on 
$\R^{1,3}$. The hyperkähler manifold $\cM$ is the 
Coulomb branch of the moduli space of vacua of $\cT$ on
$\R^{1,2} \times S^1$. The map $\nu$ is a complex integrable system, that is, 
$\nu$ is holomorphic with respect to a specific complex structure 
$I$ on $\cM$, and the fibers of $\nu$ are
holomorphic Lagrangian with respect to the
$I$-holomorphic symplectic form. General fibers of $\nu$ 
endowed with the complex structure $I$ are abelian varieties.

Due to supersymmetry, particular sectors of $\cT$ have remarkable protections against arbitrary quantum corrections and so can be often computed exactly. 
Examples of such protected sectors are the algebra $\cA_\cT$ of $\frac{1}{2}$BPS line operators and the 
spectrum of BPS 1-particle states. 
The algebra $\cA_\cT$ depends only on the UV behaviour of $\cT$. 
By wrapping around $S^1$, a line operator on $\R^{1,3}$ becomes a local operator on $\R^{1,2}$, and so its expectation value can be viewed as a function on $\cM$. 
In fact, 
$\cA_\cT$ is an algebra of functions on $\cM$ which are holomorphic for a complex structure $J$ on $\cM$ with respect to which $\nu$ is a special Lagrangian fibration. 
By contrast, the BPS spectrum depends on a choice of vacuum $u \in B$ 
and changes discontinuously along real codimension-one walls in $B$.

Gaiotto, Moore and Neitzke \cite{MR3250763}
described how to construct a non-commutative deformation 
$\cA_\cT^q$ of $\cA_\cT$ by twisting correlation functions 
by rotations in the plane transverse to the line operators. 
They also explained that, given a choice of vacuum $u \in B$, 
line operators have expansions in terms of IR line operators 
with coefficients given by counts of framed BPS states. 
These expansions depend discontinuously on $u$: they jump when the spectrum 
of framed BPS states jumps by forming bound states with (unframed) BPS states.

The same $\cN=2$ theory can often 
be engineered
in several ways in string/M-theory. Given a punctured Riemann surface $\bS_{g,\ell}$, 
one obtains a $\cN=2$ theory $\cT_{g,\ell}$ by compactifying on $\bS_{g,\ell}$ 
the six-dimensional $\cN=(2,0)$ superconformal field theory of type $A_1$
 living, at low energy and after decoupling of gravity, 
 on two coincident $M5$-branes in $M$-theory \cite{MR3006961}. 
 The corresponding Seiberg-Witten geometry $\nu \colon \cM \rightarrow B$ 
 is the Hitchin fibration on the moduli space $\cM$ of semistable $SL_2(\C)$-Higgs bundles on 
$\bS_{g,\ell}$ (with regular singularities and given residues at the punctures).
By non-abelian Hodge theory, $\cM$ with its complex structure $J$ 
is isomorphic to the 
 $SL_2(\C)$-character variety of $\bS_{g,\ell}$ 
(with given conjugacy classes around the punctures). 
The algebra $\cA_\cT$ of line operators is identified with the algebra of regular functions on the 
$SL_2(\C)$-character variety, and the non-commutative algebra $\cA_\cT^q$ is 
identified with the skein algebra 
$\Sk_A(\bS_{g,\ell})$, which is physically realized 
as the algebra of loop operators in quantum Liouville theory on 
$\bS_{g,\ell}$ \cite{MR2672742}. 
Explicit discussions of the families of non-commutative cubic surfaces 
describing $\Sk_A(\bS_{0,4})$ and 
$\Sk_A(\bS_{1,1})$ can be found in
\cite[Eq. (3.32)-(3.33)]{MR3080552},
\cite[Eq. (5.29)]{MR3250763}
\cite[Eq. (6.3)-(6.4)]{MR3418499},
\cite[Eq. (3.55)-(3.57)]{MR3435496}.

The theory $\cT_{0,4}$ has a Lagrangian description: 
it is the $\cN=2$ $SU(2)$ gauge theory with $N_f=4$ matter hypermultiplets
in the fundamental representation.
It is one of the earliest example of 
$\cN=2$ theory for which the low-energy effective action and 
the BPS spectrum have been determined by Seiberg and Witten
\cite{MR1306869}. In particular, 
$\cT_{0,4}$ admits a $PSL_2(\Z)$ $S$-duality group and 
a $\Spin(8)$ flavour symmetry group,  which are mixed by 
the triality action of $PSL_2(\Z)$ via its quotient $PSL_2(\Z/2\Z) \simeq S_3$. 
The Coulomb branch $B$ of 
$\cT_{0,4}$ is of complex dimension one. In the complex structure $I$, the map
$\nu \colon \cM \rightarrow B$ is an elliptic fibration. 
In the complex structure $J$, the space $\cM$ is
a $SL_2(\C)$-character variety for $\bS_{0,4}$, 
and so an affine cubic surface obtained as complement of a triangle $D$ of lines in a smooth projective cubic surface $Y$.

The key point is that there is a different realization of 
$\cT_{0,4}$ from $M$-theory. Consider $M$-theory on the $11$-dimensional background 
$\R^{1,3} \times \cM \times \R^3$ 
with an $M5$-brane on $\R^{1,3} \times \nu^{-1}(u)$, 
where $u \in B$. Then, the theory living on the $\R^{1,3}$ part of the 
$M5$-brane is $\cT_{0,4}$ in its vacuum $u$
\cite{MR1408388, MR1413905}. Furthermore, 
BPS states are geometrically realized by open $M2$-branes in $\cM$ with boundary on 
$\nu^{-1}(u)$ \cite{MR1662174}. 
Via the Ooguri-Vafa correspondence between counts of open $M2$-branes 
and open topological string
theory \cite{MR1765411}, these counts can be translated into all-genus open Gromov-Witten invariants of $\cM$. In the limit where $u$ is large, 
one can close open curves in 
$\cM$ into closed curves in $Y$ meeting $D$ in a single point, 
and we recover the invariants entering 
the definition of the canonical quantum scattering diagram 
$\fD_{\can}$ of $(Y,D)$. Our explicit description of $\fD_\can$ 
will agree with the expected $PSL_2(\Z)$-symmetric BPS spectrum of $\cT_{0,4}$ 
at large $u$
\cite{MR1306869, MR1471168} and can be viewed as a new derivation of it. 

We can now obtained the desired 
connection.
When the $\cN=2$ $N_f=4$ $SU(2)$ gauge theory is realized as a compactification
on $\bS_{0,4}$ of the $\cN=(2,0)$ $A_1$ theory,
the skein algebra $\Sk_A(\bS_{0,4})$ 
naturally appears as the algebra of line operators. 
On the other hand,
when 
the $\cN=2$ $N_f=4$ $SU(2)$ gauge theory 
is realized on a $M5$-brane wrapped on a torus fiber of $\nu \colon \cM \rightarrow B$, 
the enumerative geometry of holomorphic curves in the cubic surface $(Y,D)$
naturally appears as describing the BPS spectrum.
By Gaiotto, Moore and Neitzke \cite{MR3250763}, 
line operators and BPS spectrum are related via the wall-crossing phenomenon for the IR expansions 
of the line operators in terms of counts of framed BPS states. 
It is exactly what will happen in our proof: 
quantum scattering diagrams encode BPS states, quantum broken lines describe
framed BPS states, and the skein algebra will be reconstructed from quantum broken lines.

\subsection{Plan of the paper}
In Section
\ref{section_quantum_scattering}, 
we introduce the notions of quantum scattering diagram 
and quantum broken line in the restricted setting that will be used in all the paper. 
In Section
\ref{section_algorithms}, we introduce the quantum scattering diagram $\fD_{0,4}$, 
we state Theorem 
\ref{thm:consistent_0_4} on the consistency of $\fD_{0,4}$ and Theorem 
\ref{thm_ring_isom} comparing $\cA_{\fD_{0,4}}$ and $\Sk_A(\bS_{0,4})$, 
and we explain how Theorems \ref{thm_main_intro}-\ref{thm_main_intro_3} follow from
 Theorem 
\ref{thm:consistent_0_4} and Theorem 
\ref{thm_ring_isom}. 
In Section
\ref{section_canonical_scattering}, we define the canonical quantum scattering 
$\fD_\can$ encoding higher genus log Gromov-Witten invariants of the cubic surface $(Y,D)$, and we compute $\fD_\can$ explicitly.
In Section
\ref{section_derivation}, we compute a presentation 
by generators and relations of the algebra 
$\cA_{\fD_\can}$ defined by $\fD_\can$.
Finally, in Section
\ref{section_comparison}, we compare 
$\cA_{\fD_\can}$ with $\Sk_A(\bS_{0,4})$, and 
we end the proofs of Theorems 
\ref{thm:consistent_0_4} and
\ref{thm_ring_isom}.

\subsection*{Acknowledgment}

I acknowledge the support of Dr.\ Max R\"ossler, the Walter Haefner Foundation and the ETH Z\"urich
Foundation.

\section{Quantum scattering diagrams and quantum theta functions}

\label{section_quantum_scattering}

In Section \ref{section_integral_affine},
we introduce the integral affine manifold
with singularity $B$.
In Section \ref{section_quantum_broken_lines}, 
we define the notions of quantum scattering diagram and quantum broken lines on $B$.
In Section \ref{section_quantum_theta}, 
we define the algebra $\cA_\fD$ attached to a consistent quantum scattering diagram $\fD$.

\subsection{The integral affine manifold with singularity $B$}
\label{section_integral_affine}

Let $B$ be the quotient of $\R^2$ by the linear transformation $(x,y) \mapsto (-x,-y)$. 
We denote $0 \in B$ the image of $0 \in \R^2$. 
As $(x,y) \mapsto (-x,-y)$ acts freely on $\R^2 \backslash \{0\}$, 
the standard integral linear structure of $\R^2$ induces an integral linear structure on 
$B_0 \coloneqq B \backslash \{ 0\}$. 
The integral linear structure on 
$B_0$ has the non-trivial order two monodromy
$-\id$ around $0$, and so does not extend to the whole of $B$. 
We view $B$ as an integral linear manifold with singularity, with unique singularity $0$. 
We denote by $B_0(\Z)$ the set of integral points 
of the integral linear manifold $B_0$ and $B(\Z) \coloneqq B_0(\Z) \cup \{0\}$.

Concretely, we identify $B$ with the upper half-plane $\{(x,y) \in \R^2\,|\, y \geq 0\}$
the positive $x$-axis and the negative $x$-axis 
being identified by $(x,0) \mapsto (-x,0)$, 
and we describe $B(\Z)$ as in Equation \eqref{eq:BZ}.
Let $v_1$, $v_2$, $v_3$ be the three integral points of $B$ given by $v_1=(1,0)=(-1,0)$, 
$v_2=(0,1)$, and $v_3=(-1,1)$. We denote by 
$\rho_1$, $\rho_2$, $\rho_3$ the rays $\R_{\geq 0} v_1$, $\R_{\geq 0} v_2$, $\R_{\geq 0} v_3$, see Figure
\ref{figure: tropicalization}. 
We will generally consider the index $j$ of a point $v_j$ or of a ray $\rho_i$
as taking values modulo $3$, 
so that it makes sense to talk about the point $v_{j+1}$ 
or the ray $\rho_{j+1}$. For every $j \in \{1,2,3\}$, 
we denote by $\sigma_{j,j+1}$ the closed two-dimensional cone of $B$ 
generated by the rays $\rho_j$ and $\rho_{j+1}$. 
In particular, every element $v \in \sigma_{j,j+1}$ can be
uniquely written as $v=av_j+bv_{j+1}$ with 
$a,b \in \R_{\geq 0}$.
The three cones $\sigma_{j,j+1}$ 
define an integral polyhedral decomposition $\Sigma$ of $B$.

We write $\Lambda$ the rank two local system on $B_0$
of integral tangent vectors to $B_0$, 
and we fix a trivialization of $\Lambda$ 
on each two-dimensional cone $\sigma_{j,j+1}$.
In particular, for every point $Q \in \sigma_{j,j+1}$ and $p \in B(\Z) \cap 
\sigma_{j,j+1}$, we can view $p$ as an integral tangent vector at the point $Q$.
\begin{figure}[h]
\centering
\setlength{\unitlength}{1cm}
\begin{picture}(6,4)
\put(3,0.5){\circle*{0.1}}
\put(3,0.5){\line(1,0){4}}
\put(3,0.5){\line(-1,0){4}}
\put(3,0.5){\line(0,1){3}}
\put(3,0.5){\line(-1,1){3}}
\put(6,0.2){$\rho_1$}
\put(4,0.5){\circle*{0.1}}
\put(4,0.2){$v_1$}
\put(3,0.1){$0$}
\put(3,1.5){\circle*{0.1}}
\put(3.2,1.5){$v_2$}
\put(2,1.5){\circle*{0.1}}
\put(1.4,1.5){$v_3$}
\put(0,2.8){$\rho_3$}
\put(3.2,3.2){$\rho_2$}
\put(1.5,3){$\sigma_{2,3}$}
\put(5,2.5){$\sigma_{1,2}$}
\put(0,1.5){$\sigma_{3,1}$}
\put(2,0.5){\circle*{0.1}}
\put(2,0.2){$v_1$}
\put(0,0.2){$\rho_1$}
\end{picture}
\caption{$(B,\Sigma)$}
\label{figure: tropicalization}
\end{figure}

\subsection{Quantum scattering diagrams and quantum broken lines}
\label{section_quantum_broken_lines}

In Sections \ref{section_quantum_broken_lines} and 
\ref{section_quantum_theta}, we fix a 
$\Z[A^{\pm}][t^{D_1},t^{D_2}, t^{D_3}]$-algebra $R$ 
of coefficients, and an half-integer 
$\mu \in \frac{1}{2}\Z$. 
We will use the skew-symmetric bilinear form 
$\langle-,-\rangle \coloneqq \mu \det(-,-)$ on $\Lambda$.

\begin{defn}\label{defn_quantum_ray}
A \emph{quantum ray} $\rho$ with coefficients in $R$ is a pair $(p_{\rho}, f_\rho)$ where:
\begin{enumerate}
\item $p_{\rho} \in B_0(\Z)$ primitive.
\item $f_\rho$ is an element of $R\,[\![z^{-p_\rho}]\!]$
such that $f_\rho=1 \mod z^{-p_\rho}$.
\end{enumerate}
\end{defn}

\begin{defn}\label{defn_quantum_scattering_gen} 
A \emph{quantum scattering diagram} over $R$ is a collection
$\fD=\{\rho=(p_{\rho}, f_{\rho})\}$
of quantum rays with coefficients in $R$
such that $\rho_1=\rho_2$ if 
$\R_{\geq 0} p_{\rho_1}= \R_{\geq 0} p_{\rho_2}$.
\end{defn}

\begin{defn} \label{defn_broken_line}
Let $\fD$ be a quantum scattering diagram
over $R$.
A \emph{quantum broken line} $\gamma$ 
for $\fD$
with charge  $p \in B_0(\Z)$ and endpoint $Q \in B_0$ 
is a proper continuous 
piecewise integral affine map
\[ \gamma \colon
(-\infty, 0]
\rightarrow B_0 \]
with only finitely 
many domains of linearity,
together with, for each 
$L \subset (-\infty, 0]$
a maximal
connected domain of linearity
of $\gamma$,
a choice of monomial
$m_L=c_L z^{p_L}$
with
$c_L \in R$ non-zero
and
$p_L$ a section of the local system $\gamma^{-1}(\Lambda)|_L$ on $L$, such that
the following statements hold. 
\begin{enumerate}
\item For each maximal connected domain of linearity $L$, we have 
$-p_L(t)=\gamma'(t)$ for every $t \in L$.
\item $\gamma(0)=Q
\in B_0$.
\item  
For the unique unbounded domain of linearity $L$, 
$\gamma|_L$ goes off to infinity parallel to $\R_{\geq 0} p$ and
$m_L=z^{p}$
for $t \rightarrow -\infty$. 
\item Let $t \in (-\infty,0)$
be a point at which $\gamma$ is not linear, passing from the domain of
linearity $L$ to 
the domain of linearity $L'$.
Then, there exists a quantum ray 
$\rho=(p_{\rho}, f_{\rho})$ of $\fD$ such that 
$\gamma(t) \in \R_{\geq 0} p_{\rho}$.
Write $m_L=c_L z^{p_L}$, 
$m_{L'}=c_{L'}z^{p_{L'}}$,
$N \coloneqq |\det(p_{\rho},p_L)|$,  and 
\[ f_{\rho}=\sum_{k \geq 0} c_k z^{-k p_{\rho}} \,.\] 
If $\rho \neq \rho_j$ for every $1 \leq j \leq 3$, then we set $\alpha \coloneqq 1$. 
If $\rho=\rho_j$, $\gamma$ goes from $\sigma_{j-1,l}$ to $\sigma_{j,j+1}$, 
and $p_L=av_{j-1}+bv_j$, then we set $\alpha \coloneqq t^{aD_j}$. 
If $\gamma$ goes from $\sigma_{j,j+1}$ and $\sigma_{j-1,j}$, and 
$p_L=av_j+bv_{j+1}$, then we set $\alpha \coloneqq t^{bD_j}$.
Then there
exists a monomial $d_\ell z^{-\ell p_\rho}$
in the series 
\begin{equation} \label{eq:bending_formula}
\sum_{\ell \geq 0} d_\ell z^{-\ell p_\rho} \coloneqq 
 \prod_{j=0}^{N-1} 
\left( \sum_{k \geq 0} c_k A^{4 \mu k(j-
\frac{N-1}{2})} z^{-kp_\rho} \right)\,, 
\end{equation}
such that 
\begin{equation} \label{eq:broken_line_contribution}
 c_{L'}=\alpha d_\ell c_L \,\,\, 
\text{and} \,\,\,
p_{L'}=p_L- \ell p_{\rho}  \,.
\end{equation}
In other words, when the quantum broken line $\gamma$ bends from 
$L$ to $L'$, the attached monomial changes 
according to Equation \eqref{eq:broken_line_contribution}.
\end{enumerate}
\end{defn}

Note that in some cases, we will consider a 
$\Z[A^{\pm}][t^{D_1},t^{D_2},t^{D_3}]$-module $R$ 
where $t^{D_j}$ acts as the identity on $R$ for all $1 \leq j \leq 3$. 
In such case, we can forget 
the discussion of the factor $\alpha$ in Definition \ref{defn_broken_line}.

We recall symmetrized (invariant under $A \mapsto A^{-1})$ versions of standard $q$-objects.
For every nonnegative integer $n$, define the $A$-integer 
\begin{equation} [n]_A \coloneqq \frac{A^{2 n}-A^{-2 n}}{A^{2}-A^{-2}}= A^{-2(n-1)}\sum_{j=0}^{n-1} A^{4 j} \in \Z_{\geq 0}[A^{\pm}] \,,\end{equation}
and the $A$-factorial
\begin{equation} [n]_A! \coloneqq \prod_{j=1}^n [j]_A \in \Z_{\geq 0}[A^{\pm}]\,.\end{equation}
For every nonnegative integers $k$ and $n$, 
define the $A$-binomial coefficient (e.g. see Section 1.7 of
\cite{MR2868112})
\begin{equation}\label{eq:binomial}
\binom{n}{k}_A 
\coloneqq 
\frac{[n]!_A}{[k]!_A [n-k]!_A}
\in \Z_{\geq 0}[A^{\pm}] \,.
\end{equation}

\begin{lem}
Let $f_\rho=\sum_{k \geq 0} c_k z^{-kp_\fd}$ such that $f_\rho=1 \mod z^{-p_\fd}$. 
Writing 
\[f_\rho=\prod_{k \geq 1}(1+a_k z^{-kp_\fd})\,,\] 
we have 
\begin{equation}
\prod_{j=0}^{N-1} 
\left( \sum_{k \geq 0} c_k A^{4 \mu k(j-
\frac{N-1}{2})} z^{-kp_\rho} \right)
=\prod_{k \geq 1}\left(
\sum_{j=0}^N \binom{N}{j}_{A^{\mu k}} 
a_k^j z^{-kj p_\rho}
\right) \,.
\end{equation}
\end{lem}

\begin{proof}
The result follows from the $q$-binomial theorem (see e.g.\ Equation 
(1.87) of \cite{MR2868112}).
\end{proof}

\begin{defn}
Using the notation of Definition \ref{defn_broken_line}, 
the positive integer $\sum_{k \geq 0} n_k k$ is the 
\emph{amount of bending} of the quantum broken line $\gamma$ 
between the domain of linearity $L$ and the domain of
linearity $L'$.
\end{defn}

\begin{defn}  \label{defn_monomial_broken_line}
Let $\fD$ be a quantum scattering diagram
over $R$ and 
$\gamma$ a broken line for $\fD$. The \emph{final monomial} of $\gamma$ is the monomial 
$m_L$ attached to the domain of linearity $L$
of $\gamma$ containing $0$. We write the final
monomial of $\gamma$ as $c(\gamma) z^{s(\gamma)}$, 
where $c(\gamma) \in R$ and 
$s(\gamma) \in \Lambda_{\gamma(0)}$.
\end{defn}

Following \cite{gross2019cubic}, we now introduce a function 
$F \colon B \rightarrow \R$, which will be used to constrain the possible broken lines.
In order to minimize the number of minus signs, 
we take for our $F$ the function $-F$ in the notation of
\cite{gross2019cubic}. 
Let $F \colon B \rightarrow \R$ be the continuous
function on $B$, which is linear on each cone
of $\Sigma$ and such that $F(v_j)=1$ for every 
$1 \leq j \leq 3$. Explicitly, we have 
 $F((x,y))=x+y$ for $(x,y) \in \sigma_{1,2}$,
 $F((x,y))=y$ for $(x,y) \in \sigma_{2,3}$,
 and $F((x,y))=x+2y$ for $(x,y) \in \sigma_{3,1}$. 
 Note that, for every 
 $p \in B(\Z)$, $F(p)$ is a nonnegative integer.

\begin{prop} \label{prop_bound}
Let $\fD$ be a quantum scattering diagram, $p_1, p_2 \in B_0(\Z)$,
$p \in B(\Z)$, and $Q$ a point in the interior of a 
two-dimensional cone of $\Sigma$
containing $p$. Let $(\gamma_1,\gamma_2)$ be
a pair of quantum broken lines for
$\fD$ with charges $p_1$,$p_2$ and common endpoint $Q$, 
such that writing $c(\gamma_1)z^{s(\gamma_1)}$ and 
$c(\gamma_2)z^{s(\gamma_2)}$ the final monomials, we have 
$s(\gamma_1)+s(\gamma_2)=p$. 
Then, the following holds.
\begin{enumerate}
\item $F(p)\leq F(p_1)+F(p_2)$.
\item If either $\gamma_1$ or 
$\gamma_2$ crosses one of the rays
$\rho_j$ or bends at a wall, then 
\[F(p) \leq F(p_1)-F(p_2)-1\,.\]
\item The sum over all the walls $\rho$ at which either 
$\gamma_1$ or $\gamma_2$ bends, of 
the product of $F(p_{\rho})$ by the amount of bending, 
is bounded above by $F(p_1)+F(p_2)-F(p)$. 
\end{enumerate}
\end{prop}

\begin{proof}

If $\gamma$ is a broken line, we can consider 
the piecewise constant function 
$dF(\gamma'(.))\colon t \mapsto dF(\gamma'(t))$ 
defined on the interior of the domains of 
linearity of $\gamma$.

Let $\gamma_1$ and $\gamma_2$ be broken lines like 
in the statement of Proposition \ref{prop_bound}. 
As $Q \notin \bigcup_{j=1}^3 \rho_j$,
$F$ is linear in a neighborhood of $Q$, and so for $t<<0$, we have 
\[ F(p_1)+F(p_2)-F(p)
=-dF(\gamma_1'(t))-dF(\gamma_2'(t))+dF(-\gamma_1'(0)-\gamma_2'(0))\]
\[
=(dF(\gamma_1'(0))-dF(\gamma_1'(t)))+
(dF(\gamma_2'(0)-dF(\gamma_2'(t))) \,.\]

Therefore,
using the notations of 
Definition
\ref{defn_broken_line}, it is enough to show that each time $\gamma$ crosses a ray 
$\rho_j$, $dF(\gamma'(.))$ increases at least by $1$, 
and that each time 
$\gamma$
  bends along a quantum ray
$\rho$ of $\fD$, $dF(\gamma'(.))$
increases at least by $F(p_{\rho}) \sum_{k\geq 0}n_k k$.

We consider first the case where $\gamma$ crosses a ray $\rho_j$ 
without bending
at some $t \in (-\infty,0]$. Let
$t_-<t$ and $t_+>t$ be very close to $t$.
 Assume for example that $j=2$ and that $\gamma$ 
 goes from $\sigma_{12}$ to $\sigma_{13}$.
Then, $\gamma'(t_-)=(a,b)$ with $a<0$, and so 
$dF(\gamma'(t_-))=a+b$ and $dF(\gamma'(t_+)=b$, that is $dF(\gamma'(.))$ 
increases by $-a \in \Z_{\geq 1}$. If  $\gamma$ goes from $\sigma_{13}$ to $\sigma_{12}$, then 
$\gamma'(t_-)=(a,b)$ with $a>0$, and so 
$dF(\gamma'(t_-))=b$ and $dF(\gamma'(t_+)=a+b$, and so  $dF(\gamma'(.))$ 
increases by $a \in \Z_{\geq 1}$. Cases with $j=1$ and $j=3$ follow similarly.

We then consider the case where 
$\gamma$ bends along a quantum ray $\rho$ of $\fD$.
Let $t_-$ be in the domain of linearity $L$ just before the bending 
and $t_+$ the domain of linearity
$L'$ just after the bending. 
If $p_\fd \neq v_j$ for every $j$, then $F$ is linear on a 
neighborhood of the bending and so
\[ dF(\gamma'(t_+))-dF(\gamma'(t_-))
=dF(p_{L})-dF(p_{L'})=F(p_{\rho}) \sum_{k \geq 1} n_k k\,.\]
If $p_\fd=v_j$ for some $j$, then, following the analysis done 
in the case $\gamma$ crosses $\rho_j$ without bending, we have the even stronger bound
\[ dF(\gamma'(t_+))-dF(\gamma'(t_-))
=dF(p_{L})-dF(p_{L'})>F(p_{\rho}) \sum_{k \geq 1} n_k k\,.\]

\end{proof}

\begin{defn}
Let $\fD$ be a quantum scattering diagram,
$p_1, p_2 \in B_0(\Z)$, and $p \in B(\Z)$. 
We define
$\fD_{p_1,p_2}^p \coloneqq \{ \rho=(p_{\rho},f_{\rho}) \in \fD \,\,|\,\, |F(p_{\rho})| \leq F(p)-F(p_1)-F(p_2)\} $.
\end{defn}

\begin{lem} \label{lem_finite}
Let $\fD$ be a quantum scattering diagram,
$p_1, p_2 \in B_0(\Z)$, and 
$p \in B(\Z)$. Then $\fD_{p_1,p_2}^p$ is finite.
\end{lem}

\begin{proof}
The function $F \colon B \rightarrow \R$ is proper.
\end{proof}

\begin{prop} \label{prop_finite}
Let $\fD$ be a quantum scattering diagram, $p_1, p_2 \in B_0(\Z)$, 
$p\in B(\Z)$, and $Q$ a point in the interior 
of a two-dimensional cone of $\Sigma$ containing $p$. 
Then there are finitely many 
pairs $(\gamma_1,\gamma_2)$ of quantum broken lines for
$\fD$ with charges $p_1$,$p_2$ and common endpoint $Q$, 
such that  writing $c(\gamma_1)z^{s(\gamma_1)}$ and 
$c(\gamma_2)z^{s(\gamma_2)}$ the final monomials, we have 
$s(\gamma_1)+s(\gamma_2)=p$. 
Furthermore, if $\rho$ is a bending quantum ray for either 
$\gamma_1$ or $\gamma_2$, then $\rho \in \fD_{p_1,p_2}^p$.
\end{prop}

\begin{proof}
By Proposition \ref{prop_bound} and Lemma \ref{lem_finite}, 
there are finitely many possible bending quantum rays for $\gamma_1$ and $\gamma_2$, 
and the amount of each bending is uniformly bounded.
\end{proof}

\subsection{Quantum theta functions}
\label{section_quantum_theta}

\begin{defn}
Let $\fD$ be a quantum scattering diagram over $R$, $p_1, p_2 \in B_0(\Z)$,
$p \in B(\Z)$, and $Q \in B_0$ a point in a connected component of 
\[ B_0 \backslash \left(\bigcup_{j=1}^3 \rho_j\cup \bigcup_{\rho \in \fD_{p_1,p_2}^p} \R_{\geq 0} p_{\rho} \right)\] 
whose closure contains $\R_{\geq 0}p$, and such that 
the half-line $\R_{\geq 0}Q$ has irrational slope.
We define the \emph{structure constants}
\begin{equation} \label{eq:structure_constant}
C_{p_1,p_2}^{\fD, p}(Q) \coloneqq \sum_{(\gamma_1,\gamma_2)}
c(\gamma_1)c(\gamma_2) A^{2 \langle 
s(\gamma_1), s(\gamma_2) \rangle} \in R\,,
\end{equation}
where the sum is over pairs $(\gamma_1,\gamma_2)$ of quantum broken lines for
$\fD$ with charges $p_1$,$p_2$ and common endpoint $Q$, 
such that  writing $c(\gamma_1)z^{s(\gamma_1)}$ and 
$c(\gamma_2)z^{s(\gamma_2)}$ the final monomials, we have 
$s(\gamma_1)+s(\gamma_2)=p$. 
\end{defn}

We extend the definition of $C_{p_1,p_2}^{\fD, p}(Q)$ 
to all $p_1,p_2,p \in B(\Z)$ by setting 
\begin{equation}
C_{0,p_2}^{\fD, p}(Q)\coloneqq \delta_{p_2,p}
\,\,\,\,
\text{and} 
\,\,\,\,
C_{p_1,0}^{\fD, p}(Q)
\coloneqq \delta_{p_1,p} \,.
\end{equation}

By Proposition \ref{prop_finite}, 
the sum in Equation \eqref{eq:structure_constant} is indeed finite.

\begin{defn} \label{def:algebra}
A quantum scattering diagram 
$\fD$ over $R$ is \emph{consistent} if the following conditions hold:
\begin{enumerate} 
\item For every $p_1, p_2,p \in B(\Z)$,
the structure constant $C_{p_1,p_2}^{\fD, p}(Q)$ 
does not depend on the choice of the point $Q$. 
In such case, we write 
$C_{p_1,p_2}^{\fD, p}$ for $C_{p_1,p_2}^{\fD, p}(Q)$.
\item The product on the free $R$-module 
\[\cA_{\fD} \coloneqq \bigoplus_{p \in B(\Z)} R \, \vartheta_p \]
defined by 
\begin{equation}\label{eq:product_theta}
 \vartheta_{p_1}\vartheta_{p_2}=\sum_{p\in B(\Z)} 
C_{p_1,p_2}^{\fD,p} \vartheta_p 
\end{equation}
is associative.
\end{enumerate}
\end{defn}

In other words, given a consistent quantum scattering diagram
$\fD$ over $R$, one can construct an associative $R$-algebra 
$\cA_{\fD}$, coming with a $R$-linear basis $\{\vartheta_p\}_{p \in B(\Z)}$, 
called basis of \emph{quantum theta functions}, 
and whose structure constants can be computed 
in terms of quantum broken lines by Equation \eqref{eq:structure_constant}.

Note that the sum in Equation \eqref{eq:product_theta} 
is indeed finite: 
if $C_{p_1,p_2}^{\fD,p} \neq 0$, then 
$F(p) \leq F(p_1)+F(p_2)$ by Proposition
\ref{prop_finite}
and there are finitely many such $p \in B(\Z)$ by properness of $F$.

\begin{lem} \label{lem_filtration}
Let $p_1,p_2 \in B(\Z)$. If there is no two-dimensional 
cone of $\Sigma$ containing 
both $p_1$ and $p_2$, then 
\begin{equation}\label{eq:no_cone}
\vartheta_{p_1}\vartheta_{p_2}=\sum_{\substack{p\in B(\Z)\\ F(p) \leq F(p_1)+F(p_2)-1} }
C_{p_1,p_2}^{\fD,p} \vartheta_p \,.
\end{equation}
If there is a two-dimensional cone of 
$\Sigma$ containing both $p_1$ and $p_2$, then 
\begin{equation} \label{eq:cone}
\vartheta_{p_1}\vartheta_{p_2}=
A^{2 \langle p_1, p_2 \rangle}
\vartheta_{p_1+p_2}+
\sum_{\substack{p\in B(\Z)\\ F(p) \leq F(p_1)+F(p_2)-1} }
C_{p_1,p_2}^{\fD,p} \vartheta_p \,.
\end{equation}
\end{lem}

\begin{proof}
By Proposition \ref{prop_bound}, all the
non-zero terms in the sum \eqref{eq:product_theta} 
have $F(p) \geq F(p_1)+F(p_2)$, and the only possibility for 
$F(p)=F(p_1)+F(p_2)$ is that all the broken lines contributing to 
$C_{p_1,p_2}^{\fD,p}$ do not cross $\bigcup_{j=1}^3 \rho_j$ 
and do not bend. 
If there is no two-dimensional cone of $\Sigma$ containing 
both $p_1$ and $p_2$, then a broken line contributing to $C_{p_1,p_2}^{\fD,p}$
necessarily crosses $\bigcup_{j=1}^3  \rho_j$, so 
$F(p) \leq F(p_1)+F(p_2)-1$, and we obtain 
Equation 
\eqref{eq:no_cone}. 
If there is a two-dimensional cone of 
$\Sigma$ containing both $p_1$ and $p_2$,
then the only possibly non-zero term with
$F(p)=F(p_1)+F(p_2)$ is obtained for $p=p_1+p_2$, for which the
 broken lines are straight, and we obtain
 Equation 
\eqref{eq:cone}.   
\end{proof}

For every $n \in \Z_{\geq 0}$, define 
\begin{equation} \cA_{\fD_\can}^n \coloneqq 
\bigoplus_{\substack{p \in B(\Z) \\
F(p) \leq n }} R \,\vartheta_p \,.
\end{equation}
By Lemma
\ref{lem_filtration}, the increasing filtration 
$(\cA_{\fD_\can}^n)_{n \in \Z_{\geq 0}}$ defines a structure of filtered algebra on 
$\cA_{\fD_\can}$.
For every $p \in B(\Z)$, we define 
$m[p] \in \cA_{\fD_\can}$ as the following monomials in $\vartheta_{v_1}$,
$\vartheta_{v_2}$, $\vartheta_{v_3}$:
if $p=av_j+bv_{j+1}$ with 
$a \geq 0$ and $b \geq 0$, then 
$m[p]\coloneqq \vartheta_{v_j}^a 
\vartheta_{v_{j+1}}^b$.

\begin{lem} \label{lem_m_linear_basis}
The monomials $m[p]$ for $p \in B(\Z)$ form a 
$R$-linear basis of $\cA_{\fD_\can}$. 
In particular, the quantum theta functions 
$\vartheta_{v_1}$, $\vartheta_{v_2}$ and 
$\vartheta_{v_3}$ generate 
$\cA_{\fD_\can}$ as 
$R$-algebra.
\end{lem}

\begin{proof}
By Lemma
\ref{lem_filtration}, we have $m[p] \in 
\cA_{\fD_\can}^{F(p)}$, and the images of 
$m[p]$ and $\vartheta_p$ in the quotient 
$\cA_{\fD_\can}^{F(p)}/\cA_{\fD_\can}^{F(p)-1}$ 
only differ by a power of $A$. Therefore, the fact that 
$\{\vartheta_p\}_{p\in B(\Z)}$ is a $R$-linear 
basis of $\cA_{\fD_\can}$ implies that 
$\{m[p]\}_{p \in B(\Z)}$
is also a $R$-linear basis of $\cA_{\fD_\can}$.
\end{proof}

\section{Algorithms from the quantum scattering diagrams $\fD_{0,4}$ and $\fD_{1,1}$}
\label{section_algorithms}

In Section
\ref{section_scattering_04}, we first introduce the quantum scattering diagram $\fD_{0,4}$, and then
we state Theorem 
\ref{thm:consistent_0_4} on the consistency of $\fD_{0,4}$ and Theorem 
\ref{thm_ring_isom} comparing $\cA_{\fD_{0,4}}$ and $\Sk_A(\bS_{0,4})$. 
We also explain how to deduce Theorem \ref{thm_main_intro}
for $\Sk_A(\bS_{0,4})$
and Theorem \ref{thm_main_intro_2} from
 Theorem 
\ref{thm:consistent_0_4} and Theorem 
\ref{thm_ring_isom}. In Section \ref{section_scattering_11}, we introduce the quantum scattering diagram $\fD_{1,1}$ and we explain how Theorem \ref{thm_main_intro}
for $\Sk_A(\bS_{1,1})$
and Theorem \ref{thm_main_intro_4} follow from
 Theorem 
\ref{thm:consistent_0_4} and Theorem 
\ref{thm_ring_isom}.
In Section \ref{section_closed_torus}, we use our description of 
$\Sk_{A}(\bS_{1,1})$ to recover the results of
Frohman and Gelca \cite{MR1675190} on the skein algebra $\Sk_A(\bS_{1,0})$ 
of the closed torus $\bS_{1,0}$. Finally, in Section 
\ref{section_cluster}, we prove Theorem \ref{thm_main_intro_3} relating 
positivity for the bracelets basis of the skein algebras and
positivity for the quantum cluster $\cX$-varieties.

\subsection{The quantum scattering diagram $\fD_{0,4}$}
\label{section_scattering_04}
We take 
$R
=\Z[A^{\pm}][R_{1,0}, R_{0,1}, R_{1,1}, y]$, and 
$\mu=1$.
We view $R$ as a $\Z[A^{\pm}][t^{D_1},t^{D_2},t^{D_3}]$-module 
$R$ where $t^{D_j}$ acts as the identity on $R$ for all $1 \leq j \leq 3$. 
This means that we can ignore the discussion of the factors $\alpha$ 
in the Definition \ref{defn_broken_line} of quantum broken lines.

\begin{defn}
We define
\begin{align} \label{eq:def_F}
 F(r,s,y,x) 
\coloneqq  1&+ \frac{rx(1+x^2)}{(1-A^{-4}x^2)(1-A^4x^2)}
+\frac{yx^2}{(1-A^{-4}x^2)(1-A^4x^2)}\\ \nonumber
&
+ \frac{sx^3(1+sx+x^2)}{(1-A^{-4}x^2)
(1-x^2)^2(1-A^4x^2)}\,.
\end{align}
\end{defn}

\begin{lem} \label{lem_positivity_F}
Expanding $F(r,s,y,x)$ as a power series in $x$, we have
\[ F(r,s,y,x) \in \Z_{\geq 0}[A^{\pm}][r,s,y][\![x]\!]
\,.\]
\end{lem}

\begin{proof}
Immediate from Equation 
\eqref{eq:def_F} defining $F(r,s,y,x)$ and from 
the power series expansion 
\[ \frac{1}{1-u}=\sum_{k \geq 0} u^k \,.\]
\end{proof}

The first few terms of $F$ as a power series in $x$ are 
\begin{equation}
 F(r,s,y,x)=1+rx+yx^2
 +(s+r(A^{-4}+1+A^4))x^3
 +(s^2+A^{-4}+A^4)x^4+\dots 
 \end{equation}

\begin{defn} \label{defn_rays_sk}
For every $(m,n) \in B_0(\Z)$ with $m$ and $n$ coprime, we define a quantum ray 
$\rho_{m,n}=(p_{\rho_{m,n}},
f_{\rho_{m,n}})$ with coefficients 
in $\Z[A^{\pm}][R_{1,0}, R_{(0,1}, R_{1,1},y]$
by 
$p_{\rho_{m,n}}=(m,n)$,
and
\begin{enumerate}
\item if $(m,n)=(1,0) \mod 2$, 
$ f_{\rho_{m,n}}
\coloneqq F(R_{1,0}, R_{0,1}R_{1,1},y,z^{-(m,n)})$,
\item if $(m,n)=(0,1) \mod 2$, 
$ f_{\rho_{m,n}}
\coloneqq F(R_{0,1}, R_{1,0}R_{1,1},y,z^{-(m,n)})$,
\item if $(m,n)=(1,1) \mod 2$, 
$ f_{\rho_{m,n}}
\coloneqq F(R_{1,1}, R_{1,0}R_{0,1},y,z^{-(m,n)})$.
\end{enumerate}
\end{defn}

\begin{lem} \label{lem_positive_functions}
For every $(m,n) \in B_0(\Z)$ with $m$ and $n$
coprime, we have 
\[ f_{\rho_{m,n}} \in \Z_{\geq 0}[A^{\pm}][R_{1,0}, R_{0,1}, R_{1,1},y][\![z^{-(m,n)} ]\!]  \,.\]
\end{lem}

\begin{proof}
Immediate from Lemma \ref{lem_positivity_F}
and Definition \ref{defn_rays_sk}.
\end{proof}

\begin{defn} \label{defn_scat_04}
We define a quantum scattering diagram 
$\fD_{0,4}$ over $\Z[A^{\pm}][R_{1,0}, R_{0,1}, R_{1,1}, y]$
by
\[ \fD_{0,4} \coloneqq \{\rho_{m,n}\,|\,(m,n) \in B_0(\Z)\,,\,\gcd(m,n)=1\} \,.\]
\end{defn}

\begin{remark}
In the physics language used in Section \ref{section_physics}, we claim that
the quantum scattering diagram $\fD_{0,4}$ 
encodes the BPS spectrum of the $\cN=2$ theory $\cT_{0,4}$, that is, of the 
$\cN=2$ $N_f=4$ $SU(2)$ gauge theory,
at large values of $u$ on the Coulomb branch. The fact that $f_{\rho_{m,n}}$ only depends on $(m,n) \mod 2$ via permutations of 
$\{R_{1,0},R_{0,1},R_{1,1}\}$ reflects the $PSL_2(\Z)$
$S$-duality symmetry, mixed with the $\Spin(8)$ flavour symmetry 
by the triality action of 
$PSL_2(\Z/2\Z) \simeq S_3$ \cite{MR1306869}.
However, it is not so clear from
Equation \eqref{eq:def_F} that the precise form of $f_{\rho_{m,n}}$
agrees with the expected BPS spectrum of $\cT_{0,4}$. This will become manifest after some rewriting:
see Remark \ref{rem}.
\end{remark}

The following Theorem \ref{thm:consistent_0_4} and Theorem \ref{thm_ring_isom} are 
our main technical results and their proof will take the remainder of the paper.
The proof of Theorem \ref{thm:consistent_0_4} ends in Section 
\ref{sect_can_nu}, whereas the proof of Theorem \ref{thm_ring_isom}
is concluded in Section \ref{section_end_proof}.

\begin{thm} \label{thm:consistent_0_4}
The quantum scattering diagram 
$\fD_{0,4}$ is consistent.
\end{thm}

By Theorem \ref{thm:consistent_0_4}, 
it makes sense to consider the
$\Z[A^{\pm}][R_{1,0}, R_{0,1}, R_{1,1}, y]$-algebra $\cA_{\fD_{0,4}}$
given by Definition
\ref{def:algebra}, 
with its basis 
$\{ \vartheta_p \}_{p \in B(\Z)}$ of quantum theta functions, 
and structure constants
$C_{p_1,p_2}^{\fD_{0,4},p}$.

\begin{thm} \label{thm_ring_isom}
There is a unique morphism 
\[ \varphi \colon \cA_{\fD_{0,4}} 
\longrightarrow \Sk_A(\bS_{0,4}) \]
of $\Z[A^{\pm}][R_{1,0},R_{0,1}, R_{1,1},y]$-algebras
such that \[ \varphi(\vartheta_p)=\mathbf{T}(\gamma_p)\] 
for every $p \in B(\Z)$. Moreover, 
after extension of scalars for  $\cA_{\fD_{0,4}} $ from 
$\Z[A^{\pm}][R_{1,0},R_{0,1}, R_{1,1},y]$ to $\Z[A^{\pm}][a_1,a_2,a_3,a_4]$, $\varphi$ becomes an isomorphism of $\Z[A^{\pm}][a_1,a_2,a_3,a_4]$-algebras.

In particular, structure constants of the skein algebra $\Sk_A(\bS_{0,4})$ 
defined by Equation \eqref{eq:str_cst_0_4} coincide 
with the structure constants of the scattering diagram $\fD_{0,4}$ defined by 
Equation \eqref{eq:structure_constant}:
for every $p_1,p_2,p \in B(\Z)$, we have
\begin{equation} 
C_{p_1,p_2}^{\bS_{0,4},p}=C_{p_1,p_2}^{\fD_{0,4},p} \,.
\end{equation}
\end{thm}

\begin{cor} \label{cor_traces}
The classical theta functions $\vartheta_p^{\cl}$ constructed in
\cite{MR3415066, gross2019cubic}
coincide with the trace functions 
$\rho \mapsto -\tr(\rho(\gamma_{p_{prim}})^k)$ on the character variety 
$\Ch_{SL_2}(\bS_{0,4})$, where $p=kp_{prim}$ with $k \in \Z_{\geq 1}$ and $p_{prim}
\in B(\Z)$ primitive. 
\end{cor}

\begin{proof}
It is an immediate corollary of Theorem \ref{thm_ring_isom}. In the classical limit 
$A=-1$, our quantum theta functions $\vartheta_p$ reduce to the classical theta functions 
$\vartheta_p^{\cl}$ of \cite{MR3415066, gross2019cubic}, and the element 
$\mathbf{T}(\gamma_p)$ of $\Sk_A(\bS_{0,4})$ reduces to the function 
$\rho \mapsto -\tr(\rho(\gamma_{p_{prim}})^k)$ on $\Ch_{SL_2}(\bS_{0,4})$
by the general relation between skein algebras and character varieties 
reviewed in Section \ref{section_skein_character}.
\end{proof}

Theorem \ref{thm_ring_isom} implies Theorem \ref{thm_main_intro_2}. 
Indeed, by Lemma \ref{lem_positive_functions}, 
the functions attached to the quantum rays of $\fD_{0,4}$ have coefficients in 
$\Z_{\geq 0}[A^{\pm}][R_{1,0}, R_{0,1}, R_{1,1},y]$, 
and so it follows from the definition
of the structure constants $C_{p_1,p_2}^{\fD_{0,4},p} $
in Equation \eqref{eq:structure_constant} in terms of broken lines, and
from the formulas
\eqref{eq:bending_formula}-\eqref{eq:broken_line_contribution} recursively 
computing the contribution of a broken line, that 
\begin{equation}
C_{p_1,p_2}^{\fD_{0,4},p}
\in \Z_{\geq 0}[A^{\pm}][R_{1,0}, R_{0,1}, R_{1,1},y]
\end{equation}
for every $p_1,p_2,p \in B(\Z)$.
By Theorem \ref{thm_ring_isom}, we have 
$C_{p_1,p_2}^{\bS_{0,4},p}=C_{p_1,p_2}^{\fD_{0,4},p}$, and so 
\begin{equation}
C_{p_1,p_2}^{\bS_{0,4},p}
\in \Z_{\geq 0}[A^{\pm}][R_{1,0}, R_{0,1}, R_{1,1},y]\,,
\end{equation}
that is, Theorem \ref{thm_main_intro_2} holds.

Finally, we explain how Theorem \ref{thm_main_intro_2} 
implies Theorem \ref{thm_main_intro} for $\Sk_A(\bS_{0,4})$.
A general element of the bracelets basis 
$\mathbf{B}_T$ of $\Sk_A(\bS_{0,4})$ is of the form 
\begin{equation} T_{n_1}(a_1)T_{n_2}(a_2)T_{n_3}(a_3)T_{n_4}(a_4)
\mathbf{T}(\gamma_p)
\end{equation}
for some $n_j \in \Z_{>0}$ for $1 \leq j \leq 4$ 
and some $p \in B(\Z)$. As the $T_{n_j}(a_j)$ are in the center of $\Sk_A(\bS_{0,4})$ and 
\begin{equation} \label{eq:cheby}T_{n_j}(a_j)T_{n_j'}(a_j)=T_{n_j+n_j'}(a_j)
+T_{|n_j-n_j'|}(a_j)\,,
\end{equation} it is enough to show that the structure constants 
$C_{p_1,p_2}^{\fD_{0,4},p}$ are polynomials 
in the variables $T_n(a_j)$ with coefficients in $\Z_{\geq 0}[A^{\pm}]$. 
By Theorem \ref{thm_main_intro_2}, we have
\[C_{p_1,p_2}^{\bS_{0,4},p}
\in \Z_{\geq 0}[A^{\pm}][R_{1,0}, R_{0,1}, R_{1,1},y]\,.\]
Using Equation \eqref{eq:cheby} again, it is enough to show that 
$R_{1,0}$, $R_{0,1}$, $R_{1,1}$ and $y$
are polynomials in the variables $T_n(a_j)$ with coefficients in $\Z_{\geq 0}[A^{\pm}]$.
Using that $T_1(x)=x$ and $T_2(x)=x^2-2$, we have from Equations 
\eqref{eq:R}-\eqref{eq:y} 
\begin{equation} R_{1,0}=T_1(a_1)T_1(a_2)+T_1(a_3)T_1(a_4)\,,
\end{equation}
\begin{equation}  R_{0,1}=T_1(a_1)T_1(a_3)+T_1(a_2)T_1(a_4)\,,
\end{equation}
\begin{equation}  R_{1,1}=T_1(a_1)T_1(a_4)+T_1(a_2)T_1(a_3)\,,
\end{equation}
\begin{equation}  y
=T_1(a_1)T_1(a_2)T_1(a_3)T_1(a_4)
+T_2(a_1)^2+T_2(a_2)^2+T_2(a_3)^2+
T_2(a_4)^2+A^4+6+A^{-4} \,,
\end{equation}
and so the result holds.

\subsection{The quantum scattering diagram $\fD_{1,1}$}
\label{section_scattering_11}

We take $R=\Z[A^{\pm}][z]$
and $\mu=\frac{1}{2}$.
We view $R$ as a $\Z[A^{\pm}][t^{D_1},t^{D_2},t^{D_3}]$-module 
where $t^{D_j}$ acts as the identity on $R$ for all $1 \leq j \leq 3$. 
This means that we can ignore the discussion of the factors $\alpha$ 
in the Definition \ref{defn_broken_line} of quantum broken lines.

\begin{defn}
We define 
\begin{equation} \label{eq:def_G}
G(z,x) \coloneqq 1+\frac{zx^2}{(1-A^{-2}x^2)
(1-A^2 x^2)}
\end{equation} 
\end{defn}

\begin{lem} \label{lem_positivity_G}
Expanding $G(z,x)$ as a power series in $x$, we have
\[ G(z,x) \in \Z_{\geq 0}[A^{\pm}][z][\![x]\!]
\,.\]
\end{lem}

\begin{proof}
Immediate from Equation 
\eqref{eq:def_F} defining $G(z,x)$ and from 
the power series expansion 
\[ \frac{1}{1-u}=\sum_{k \geq 0} u^k \,.\]
\end{proof}

The first few terms of $G$ as a power series in $x$ are 
\begin{equation}
 G(z,x)=1+zx^2
 +(A^{-2}+A^2)x^4+\dots 
\end{equation}
Note that writing $z=A^2+A^{-2}+\eta$ and 
$\eta=\lambda+\lambda^{-1}$, we have 
\begin{equation} \label{eq:bps_n_4}
G(z,x)=\frac{1+\eta x^2+x^4}{(1-A^{-2}x^2)(1-A^2 x^2)}=\frac{(1+\lambda x^2)(1+\lambda^{-1}x^2)}{(1-A^{-2}x^2)(1-A^2 x^2)}\,.
\end{equation}

\begin{defn} \label{defn_rays_sk_1}
For every $(m,n) \in B_0(\Z)$ with $m$ and $n$ coprime, 
we define a quantum ray 
$\tau_{m,n}=(p_{\tau_{m,n}},
f_{\tau_{m,n}})$
with coefficients in 
$\Z[A^{\pm}][z]$
by $ p_{\tau_{m,n}}=(m,n)$,
and
$ f_{\tau_{m,n}}
\coloneqq G(z ,z^{-(m,n)})$.
\end{defn}

\begin{lem} \label{lem_positive_functions_1_1}
For every $(m,n) \in B_0(\Z)$ with $m$ and $n$
coprime, we have 
\[ f_{\tau_{m,n}} \in \Z_{\geq 0}[A^{\pm}][z][\![z^{-(m,n)} ]\!]  \,.\]
\end{lem}

\begin{proof}
Immediate from Lemma \ref{lem_positivity_G}
and Definition \ref{defn_rays_sk_1}.
\end{proof}

\begin{defn} 
We define a quantum scattering diagram 
$\fD_{1,1}$ over $\Z[A^{\pm}][z]$ by
\[ \fD_{1,1} \coloneqq \{\tau_{m,n}\,|\,(m,n) \in B_0(\Z)\,,\,\gcd(m,n)=1\} \,.\]
\end{defn}

\begin{remark}
In the physics language used in Section \ref{section_physics}, 
the quantum scattering diagram $\fD_{1,1}$ 
encodes the BPS spectrum of the $\cN=2$ theory $\cT_{1,1}$ 
at large values of $u$ on the Coulomb branch. 
The theory $\cT_{1,1}$ has a Lagrangian description: 
it is the $\cN=2$ $SU(2)$ gauge theory coupled with 
a matter hypermultiplet in the adjoint representation, 
also known as the $\cN=2^{\star}$ theory. 
The BPS spectrum at large values of $u$ 
reduces to the BPS spectrum on the Coulomb branch of the theory 
with zero mass for the matter hypermultiplet, 
that is of the $\cN=4$ $SU(2)$ gauge theory.
Our definition of $\fD_{1,1}$ agrees 
with the expected BPS spectrum on the Coulomb branch of the $\cN=4$ $SU(2)$ gauge theory: 
for every $(m,n) \in \Z^2$ with $m$ and $n$ coprime, 
we have one vector multiplet of charge $(2m,2n)$, which corresponds to 
to the denominator of Equation 
\eqref{eq:bps_n_4}, 
two hypermultiplets of charge $(2m,2n)$,
which correspond to the numerator of 
Equation \eqref{eq:bps_n_4}
and no other states of charge a multiple of $(m,n)$
\cite{MR1306869}. 
Note that the $\cN=2$ vector multiplet and the two 
$\cN=2$ hypermultiplets combine into one $\cN=4$ vector multiplet. The states of 
charge $(2,0)$ can be seen classically 
(as $W$-bosons and elementary quarks), 
and the general states of charge $(2m,2n)$ are obtained from them by $SL_2(\Z)$ S-duality.
\end{remark}

\begin{lem} \label{lem_04_11}
$\fD_{1,1}$ is obtained from $\fD_{0,4}$
by replacing $A^4$ by $A^2$, setting 
$R_{1,0}=R_{0,1}=R_{1,1}=0$ and $y=z$.
\end{lem}

\begin{proof}
Immediate from comparing the Equations \eqref{eq:def_F} and 
\eqref{eq:def_G} defining 
$F(r,s,y,x)$ and $G(z,x)$.
\end{proof}

\begin{thm} \label{thm:consistent_1_1}
The quantum scattering diagram $\fD_{1,1}$ is consistent.
\end{thm}

\begin{proof}
By Lemma \ref{lem_04_11}, $\fD_{1,1}$ is a specialization 
of $\fD_{0,4}$, and so the consistency of $\fD_{1,1}$
follows from the consistency of 
$\fD_{0,4}$ given by Theorem \ref{thm:consistent_0_4}.
\end{proof}
 
By Theorem \ref{thm:consistent_1_1},
it makes sense to consider the
$\Z[A^{\pm}][z]$-algebra $\cA_{\fD_{1,1}}$
given by Definition
\ref{def:algebra}, 
with its basis 
$\{ \vartheta_p \}_{p \in B(\Z)}$ of quantum theta functions, 
and structure constants
$C_{p_1,p_2}^{\fD_{1,1},p}$.

\begin{thm} \label{thm_ring_isom_1_1}
There is a unique morphism 
\[ \varphi \colon \cA_{\fD_{1,1}} 
\longrightarrow \Sk_A(\bS_{1,1}) \]
of $\Z[A^{\pm}][z]$-algebras
such that \[ \varphi(\vartheta_p)=\mathbf{T}(\gamma_p)\] 
for every $p \in B(\Z)$. Moreover, 
$\varphi$ is an isomorphism of $\Z[A^{\pm}][z]$-algebras.

In particular, structure constants of the skein algebra 
$\Sk_A(\bS_{1,1})$ defined by Equation \eqref{eq:str_cst_1_1} 
coincide with the structure constants of the scattering diagram
 $\fD_{1,1}$ defined by 
Equation \eqref{eq:structure_constant}:
for every $p_1,p_2,p \in B(\Z)$, we have
\begin{equation} 
C_{p_1,p_2}^{\bS_{1,1},p}=C_{p_1,p_2}^{\fD_{1,1},p} \,.
\end{equation}
\end{thm}

\begin{proof}
Theorem \ref{thm_ring_isom_1_1} follows from the similar result, Theorem 
\ref{thm_ring_isom}, for $\bS_{0,4}$. Indeed, by Lemma \ref{lem_04_11}, 
the algebra 
$\cA_{\fD_{1,1}}$ is obtained from $\cA_{\fD_{0,4}}$ by setting 
$R_{1,0}=R_{0,1}=R_{1,1}=0$, $y=z$,
and by matching the quantum theta 
functions 
$\{ \vartheta_p \}_{p \in B(\Z)}$.
On the other hand,
Bullock and Przytycki gave in \cite{MR1625701} explicit 
presentations by generators and relations of $\Sk_A(\bS_{0,4})$ and 
$\Sk_A(\bS_{1,1})$, and observed that 
 $\Sk_A(\bS_{1,1})$ is obtained from
 $\Sk_A(\bS_{0,4})$ by setting 
$R_{1,0}=R_{0,1}=R_{1,1}=0$, $y=z$, 
and by matching the multicurves 
$\{\gamma_p\}_{p \in B(\Z)}$. 
\end{proof}

Theorem \ref{thm_ring_isom_1_1} implies Theorem \ref{thm_main_intro_4}. 
Indeed, by Lemma \ref{lem_positive_functions_1_1}, 
the functions attached to the quantum rays of $\fD_{1,1}$ have coefficients in 
$\Z_{\geq 0}[A^{\pm}][z]$, and so it follows from the definition
of the structure constants $C_{p_1,p_2}^{\fD_{1,1},p} $
in Equation \eqref{eq:structure_constant} 
in terms of broken lines, and from the formulas 
\eqref{eq:bending_formula}-\eqref{eq:broken_line_contribution} recursively computing 
the contribution of a broken line 
 that 
\begin{equation}
C_{p_1,p_2}^{\fD_{1,1},p}
\in \Z_{\geq 0}[A^{\pm}][z]
\end{equation}
for every $p_1,p_2,p \in B(\Z)$.
By Theorem \ref{thm_ring_isom_1_1}, we have 
$C_{p_1,p_2}^{\bS_{1,1},p}=C_{p_1,p_2}^{\fD_{1,1},p}$, and so 
\begin{equation}
C_{p_1,p_2}^{\bS_{1,1},p}
\in \Z_{\geq 0}[A^{\pm}][z]\,,
\end{equation}
that is, Theorem \ref{thm_main_intro_4} holds.

Finally, we explain how Theorem \ref{thm_main_intro_4} 
implies Theorem \ref{thm_main_intro} for $\Sk_A(\bS_{1,1})$.
A general element of the bracelets basis 
$\mathbf{B}_T$ of $\Sk_A(\bS_{1,1})$ is of the form 
\begin{equation} T_{n}(\eta)
\mathbf{T}(\gamma_p)
\end{equation}
for some $n \in \Z_{>0}$ 
and some $p \in B(\Z)$. As the $T_{n}(\eta)$ 
are in the center of $\Sk_A(\bS_{1,1})$, it follows from the identity 
\eqref{eq:cheby}
that it is enough to show that the structure constants 
$C_{p_1,p_2}^{\fD_{1,1},p}$ are polynomials in the variables 
$T_n(\eta)$ with coefficients in $\Z_{\geq 0}[A^{\pm}]$. 
By Theorem \ref{thm_main_intro_4},
we have
$C_{p_1,p_2}^{\bS_{1,1},p}
\in \Z_{\geq 0}[A^{\pm}][z]$.
Using Equation \eqref{eq:cheby} again, it is enough to show that 
$z$ is a polynomial in the variables $T_n(\eta)$ 
with coefficients in $\Z_{\geq 0}[A^{\pm}]$.
As $z=T_1(\eta)+A^2+A^{-2}$, this indeed holds.

\subsection{Recovering the skein algebra of the closed torus}
\label{section_closed_torus}

As reviewed in Section \ref{section_positivity_bracelets},
Frohman and Gelca \cite{MR1675190}
described explicitly the skein algebra $\Sk_A(\bS_{1,0})$ 
of the (closed) torus $\bS_{1,0}$. 
We explain below how this result appears as a limit of 
our description of the skein algebra $\Sk_A(\bS_{1,1})$ 
of the $1$-punctured torus $\bS_{1,1}$.

The closed torus $\bS_{1,0}$ is obtained from the $1$-punctured torus $\bS_{1,1}$ by closing the puncture, that is by
making topologically trivial the peripheral curve $\eta$. 
Thus, the skein algebra $\Sk_A(\bS_{1,0})$ is obtained from $\Sk_A(\bS_{1,1})$
by setting $\eta=-A^2-A^{-2}$, that is $z=0$. 
Theorem \ref{thm_ring_isom_1_1} gives a description of 
$\Sk_A(\bS_{1,1})$ in terms of the scattering diagrams 
$\fD_{1,1}$. Setting $z=0$ in the Definition \ref{defn_rays_sk_1} of 
$\fD_{\can}$, we obtain a trivial scattering diagrams 
whose all quantum rays $\rho=(p_\rho, f_\rho)$ have 
$f_\rho=1$. In particular, no broken line can bend 
in such scattering diagram, and the structure constants become extremely simple.

Let $p_1, p_2 \in B_0(\Z)$. Using the 
$PSL_2(\Z)$ action on $B(\Z)$, we can assume that $p_1$ is horizontal, 
that is, $p_1=(a,0)$.  Then, there are only 
two configurations of broken lines contributing to the product 
$\vartheta_{p_1} \vartheta_{p_2}$:
the one with $\gamma_1$ straight going to infinity parallel 
to $\R_{\geq 0}p_1$ and $\gamma_2$ straight going to infinity parallel to $\R_{\geq 0}p_2$, and the one
with $\gamma_1$ straight going to infinity parallel 
to $-\R_{\geq 0}p_1$ and $\gamma_2$ straight going to infinity parallel to $\R_{\geq 0}p_1$. 
Therefore, applying Equation\eqref{eq:structure_constant}, we obtain
\begin{equation}
\vartheta_{p_1}\vartheta_{p_2}=
A^{\det(p_1,p_2)} \vartheta_{p_1+p_2}
+A^{-\det(p_1,p_2)} \vartheta_{p_1-p_2} \,,
\end{equation}
that is, we recover the product-to-sum formula 
of Frohman and Gelca \cite{MR1675190} 
(see also \cite{MR3346923} for a different proof).

Note that in the limit where the scattering diagram on 
$B$ becomes trivial, and considering the classical limit 
$A=1$, the mirror cubic surface constructed in \cite{gross2019cubic}
becomes  
\[ \vartheta_{v_1}\vartheta_{v_2}\vartheta_{v_3}
=\vartheta_{v_1}^2+\vartheta_{v_2}^2+\vartheta_{v_3}^2-4 \,,\]
which is isomorphic to $(\mathbb{G}_m)^2/(\Z/2\Z)$, 
where $\Z$ acts on the torus $(\mathbb{G}_m)^2$ by 
$(x,y) \mapsto (x^{-1},y^{-1})$
(an isomorphism is given by 
$\vartheta_{v_1}=x+x^{-1}$, 
$\vartheta_{v_2}=y+y^{-1}$,
$\vartheta_{v_3}=xy+x^{-1}y^{-1}$).
This is the classical version of the description 
given by Frohman and Gelca \cite{MR1675190} of $\Sk_A(\bS_{1,0})$ 
as a $\Z/2\Z$-quotient of the quantum torus.

\subsection{Application to quantum cluster algebras}
\label{section_cluster}

In this section, we prove Theorem \ref{thm_main_intro_3}, 
that is, that the positivity of the structure constants 
of the bracelets basis of the skein algebra $\Sk_A(\bS_{g,\ell})$ 
implies the positivity of the structure constants
defined by the quantum duality map $\hat{\mathbb{I}}$ of
\cite{MR3581328}. We use the notations
introduced in Section \ref{section_cluster_intro}.

The quantum duality map $\hat{\mathbb{I}}$ is defined in 
\cite{MR3581328} using the quantum 
trace map of Bonahon and Wong
\cite{MR2851072}. Given an ideal triangulation $T$ of $\bS_{g,\ell}$, 
there is a corresponding quantum trace map, 
which is an injective algebra morphism
\[ \Tr_T \colon \Sk_A(\bS_{g,\ell})
\longrightarrow \hat{\mathcal{Z}}_T \,,\]
where $\hat{\mathcal{Z}}_T$ is the square root Chekhov-Fock algebra.

The set of tropical points $\cA_{SL_2,\bS_{g,\ell}}(\Z^t)$ 
is the set of 
even integral laminations on $\bS_{g,\ell}$ \cite{MR2233852}.
For every $l \in \cA_{SL_2,\bS_{g,\ell}}(\Z^t)$, 
we can write uniquely
$l = \sum_j k_j l_j$
where the $l_j$ are connected multicurves 
with distinct homotopy classes, 
$k_j \in \Z_{\geq 1}$ if $l_j$ is not peripheral, and 
$k_j \in \Z$ if $l_j$ is peripheral.
By Definition 3.11 of 
\cite{MR3581328}, we have 
$\hat{\mathbb{I}}(l)=\prod_j \hat{\mathbb{I}}(k_j l_j)$.
Therefore, to prove Theorem \ref{thm_main_intro_3}, 
it is enough 
to prove the positivity of the structure constants 
appearing in the products of the form
$\hat{\mathbb{I}}(kl)\hat{\mathbb{I}}(k'l')$ 
where $l$ and $l'$ are connected multicurves, 
$k \in \Z_{\geq 0}$ (resp.\ $k' \in \Z_{\geq 1}$)
if $l$ (resp.\ $l'$) is not peripheral, $k \in \Z$ (resp.\ 
$k' \in \Z$) if $l$ (resp.\ $l'$)
is peripheral.

By Lemma 3.25 of \cite{MR3581328}, for  $l$ a 
peripheral connected multicurve, $k  \in \Z$, 
and $l'$ a lamination, we have 
$\hat{\mathbb{I}}(kl) 
\hat{\mathbb{I}}(l')=
\hat{\mathbb{I}}(kl+l')$. 
It follows that it is enough to prove the positivity 
of the structure constants appearing in the products of the form
$\hat{\mathbb{I}}(kl)\hat{\mathbb{I}}(k'l')$ 
where $l$ and $l'$ are 
non-peripheral connected multicurves and $k,k' \in \Z_{\geq 1}$.

So, let us consider $l$ and $l'$ non-peripheral 
connected multicurves and $k,k' \in \Z_{\geq 1}$.
By Definitions 3.4 and 3.8 of
\cite{MR3581328},we have 
$\hat{\mathbb{I}}(kl)=\Tr_T(T_k(l))$
and $\hat{\mathbb{I}}(k'l')=\Tr_T (T_{k'}(l'))$.
Therefore, assuming the positivity 
of the structure constants of the bracelets basis of 
$\Sk_A(\bS_{g,\ell})$, we have 
\begin{equation} \hat{\mathbb{I}}(kl)
\hat{\mathbb{I}}(k'l')
=\Tr_T(T_k(l))\Tr_T (T_{k'}(l'))
=\Tr_T(T_k(l)T_{k'}(l'))
=\sum_\gamma C_{kl,k'l'}^{\gamma} 
\Tr_T(\mathbf{T}(\gamma)) \,,
\end{equation}
where the sum is over finitely many multicurves $\gamma$ and 
$C_{kl,kl'}^{\gamma} \in \Z_{\geq 0}
[A^{\pm}]$.
Write 
$\gamma=\gamma_1^{n_1} \cdots \gamma_r^{n_r}$
with $\gamma_1, \cdots , \gamma_r$ 
all distinct isotopy classes of connected multicurves, 
$\gamma_1,\cdots, \gamma_s$ non-peripheral and 
$\gamma_{s+1},\cdots, \gamma_r$ peripheral, and $n_j \in \Z_{>0}$. 
We have 
\begin{equation} \Tr_T(\mathbf{T}(\gamma))
=\prod_{j=1}^r \Tr_T(T_{n_j}(\gamma_j))
=\prod_{j=1}^s \hat{\mathbb{I}}
(n_j \gamma_j)
\prod_{k=s+1}^r (\Tr_T(\gamma_k))^{n_k} \,.
\end{equation}
For $s+1 \leq k \leq r$, $\gamma_k$ is peripheral, and so
by Lemma 3.24 of \cite{MR3581328}, we have 
\begin{equation}\hat{\mathbb{I}}(\gamma_k)
+\hat{\mathbb{I}}(-\gamma_k)=\Tr_T(\gamma_k)\,.
\end{equation}
Therefore, using again Lemma 3.25 of \cite{MR3581328}, we have 
\begin{equation} 
\begin{aligned}\hat{\mathbb{I}}(kl)
\hat{\mathbb{I}}(k'l')
&=\sum_\gamma C_{kl,k'l'}^{\gamma} 
\prod_{j=1}^s \hat{\mathbb{I}}
(n_j \gamma_j)
\prod_{k=s+1}^r 
(\hat{\mathbb{I}}(\gamma_k)
+\hat{\mathbb{I}}(-\gamma_k))^{n_k}  
\\
&=\sum_\gamma C_{kl,k'l'}^{\gamma} 
\prod_{j=1}^s \hat{\mathbb{I}}
(n_j \gamma_j)
\prod_{k=s+1}^r 
\left(\sum_{a=0}^{n_k}
\binom{n_k}{a}
\hat{\mathbb{I}}((2a-n_k)\gamma_k)\right)
\\
&=\sum_\gamma C_{kl,k'l'}^{\gamma} 
\sum_{a_{s+1}=0}^{n_{s+1}}
\cdots 
\sum_{a_r=0}^{n_r}
\binom{n_{s+1}}{a_{s+1}}
\cdots \binom{n_r}{a_r} 
\hat{\mathbb{I}}
\left(\sum_{j=1}^s n_j \gamma_j + \sum_{k=s+1}^r (2a_k-n_k) \gamma_k \right) \,,
\end{aligned}
\end{equation}
and so, under our assumption that 
$C_{kl,kl'}^\gamma \in \Z_{\geq 0}[A^{\pm}]$, the structure constants of 
$\hat{\mathbb{I}}$ belong to $\Z_{\geq 0}[A^{\pm}]$. On the other hand,
by Theorem 1.2 of  \cite{MR3581328}, these structure constants belong to 
$\Z[A^{\pm 2}]=\Z[q^{\pm \frac{1}{2}}]$. 
Thus, the structure constants of 
$\hat{\mathbb{I}}$ belong to $\Z_{\geq 0}[q^{\pm \frac{1}{2}}]$ 
and this proves Theorem \ref{thm_main_intro_3}.

\section{The canonical quantum scattering diagram}

\label{section_canonical_scattering}

We fix $\kk$ an algebraically closed field of characteristic zero, 
$Y$ a smooth projective cubic surface in $\PP^3_{\kk}$, 
and $D$ the union of three projective lines $D_1$, 
$D_2$, $D_3$ in 
$\PP^3_{\kk}$ contained in $Y$ and 
forming a triangle configuration.
In Section \ref{section_can_scatt}, 
 we review following \cite{MR4048291} 
the construction of the canonical  quantum scattering diagram
$\fD_\can$ associated to $(Y,D)$.
After some preliminaries on curve classes in
$Y$ presented in Section \ref{section_curves}, we give some explicit
description of $\fD_\can$ in Sections 
\ref{section_ray_rho_j_contribution} and 
\ref{section_contribution_general_rays}.

\subsection{The canonical quantum scattering diagram $\fD_{\can}$}
\label{section_can_scatt}

Following \cite[Section 3]{MR4048291}, 
we review the definition of the canonical quantum scattering diagram
attached to $(Y,D)$. The canonical quantum scattering diagram is defined
in terms of the enumerative geometry of curves in $(Y,D)$. 
More precisely, the canonical quantum scattering diagram encodes 
the data of logarithmic Gromov-Witten invariants of $(Y,D)$.

The pair $(B,\Sigma)$ defined in Section 
\ref{section_integral_affine} is the tropicalization of $(Y,D)$. 
It plays for $(Y,D)$ the role of a fan for a toric surface. 
In particular, the three two-dimensional cones 
$\sigma_{j,j+1}$ of $\Sigma$ are in natural 
correspondence with the three points $D_j \cap D_{j+1}$, 
the three one-dimensional rays $\rho_j$ of $\Sigma$
are in natural correspondence with the three divisors $D_j$, 
and the point $0 \in B$ is in natural correspondence 
with the complement $U$ of $D$ in $Y$. 
The integral linear structure on 
$B_0$ encodes the self-intersection numbers of the divisors $D_j$: 
for every 
$j \in \{1,2,3\}$, the fact that $D_j^2=-1$  translates into the
fact that $v_{j-1}+v_j+v_{j+1}=0$. 
We refer to \cite[\S 1]{gross2019cubic} and 
\cite[\S 1.2]{MR3415066}
for further details on the construction of the tropicalization.

Let $NE(Y)$ be the Mori cone of $Y$, 
i.e.\ the cone generated by effective curve classes in the group 
$A_1(Y)$ generated by numerical equivalence classes of curves on $Y$.
The group $A_1(Y)$ is a free abelian group of rank $7$.
The Mori cone $NE(Y)$ is a strictly convex rational polyhedral 
cone in $A_1(Y)$, generated by the classes of the 27 lines on $Y$.
We write $ \Q[\![\hbar]\!][NE(Y)]$ for the corresponding monoid 
ring with coefficients in the power series algebra 
$\Q[\![\hbar]\!]$, and $t^\beta$ 
for the monomial in 
$ \Q[\![\hbar]\!][NE(Y)]$ 
defined by 
$\beta \in NE(Y)$. 
We will apply the formalism of Section
\ref{section_quantum_scattering} with 
$R=\Q[\![\hbar]\!][NE(Y)]$, viewed as a 
$\Z[A^{\pm}][t^{D_1},t^{D_2},t^{D_3}]$-algebra, 
where $A$ acts by multiplication by 
$e^{\frac{i\hbar}{4}}$, and $t^{D_j}$ 
acts by multiplication by the corresponding element in $\Z[NE(Y)]$.
We will often use the notation $q=e^{i\hbar}=A^4$.

Let $\beta \in NE(Y)$ and $v \in B_0(\Z)$. 
We can write $v=av_j + bv_{j+1}$ with 
$a,b \in \Z_{\geq 0}$ for some $j \in \{1,2,3\}$. 
We are considering genus $g$ one-pointed stable maps
$f\colon (C,p) \rightarrow (Y,D)$ with $f^{-1}(D)=\{p\}$,  
such that $g$ has contact order $a$ with $D_j$ at $p$ and 
contact order $b$ with $D_{j+1}$
at $p$. Logarithmic Gromov-Witten theory 
\cite{MR3257836, MR3011419} provides 
a nice compactification $\overline{M}_{g,v}^\beta(Y/D)$ of
the space of such stable maps. The moduli space 
$\overline{M}_{g,v}^\beta(Y/D)$  is a proper Deligne-Mumford stack, 
coming with a virtual fundamental cycle
$[\overline{M}_{g,v}^\beta(Y/D)]^{\virt}$ of degree $g$.

Let $\pi \colon \cC \rightarrow \overline{M}_{g,v}^\beta(Y/D)$ 
be the universal source curve, $\omega_\pi$ 
the relative dualizing sheaf of $\pi$,
and $\pi_{*} \omega_\pi$ the rank $g$
Hodge bundle on $\overline{M}_{g,v}(Y/D,\beta)$. 
The top Chern class $\lambda_g$ of 
the Hodge bundle is a (complex) degree $g$ 
cohomology class on $\overline{M}_{g,v}^\beta(Y/D)$. 
We define log Gromov-Witten invariants $N_{g,v}^{\beta}$ 
of $(Y,D)$ by integration of the cohomology class 
$(-1)^g \lambda_g$ on the virtual fundamental cycle:
\begin{equation} \label{eq:log_gw}
N_{g,v}^{\beta} \coloneqq \int_{[\overline{M}_{g,v}^\beta(Y/D)]^{\virt}} (-1)^g \lambda_g \,.
\end{equation}
We have in general $N_{g,v}^{\beta} \in \Q$. For $g=0$,
we recover the genus $0$ log Gromov-Witten invariants
$N_v^{\beta}$ considered in \cite{MR3415066, gross2019cubic}.

\begin{lem} \label{lem_curve_classes_finite}
Given $v \in B_0(\Z)$, there exists finitely many $\beta \in NE(Y)$ 
such that $N_{g,v}^{\beta} \neq 0$ for some 
$g$.
\end{lem}

\begin{proof}
Write $v=av_j + bv_{j+1}$ with $a,b \in \Z_{\geq 0}$ 
for some $j \in \{1,2,3\}$.
The moduli space $\overline{M}_{g,v}^\beta(Y/D)$ is possibly non-empty, 
and so the invariant $N_{g,v}^{\beta}$ possibly non-zero, 
only if $\beta \cdot D_j=a$ and $\beta \cdot D_{j+1}=b$, 
and so in particular $\beta \cdot D=a+b$. 
As $D$ is an ample divisor on $Y$, 
the set of such curve classes $\beta$ is finite.
\end{proof}

\begin{defn} \label{defn_rays_can}
For every $(m,n) \in B_0(\Z)$ with $m$ and $n$ coprime, 
we define a quantum ray 
$\fd_{m,n}=(p_{\fd_{m,n}},
f_{\fd_{m,n}})$ with coefficients 
in $\Q[\![\hbar]\!][NE(Y)]$
by 
$p_{\fd_{m,n}}=(m,n)$,
and
\begin{equation} \label{eq:rays_can}
 f_{\fd_{m,n}}
\coloneqq 
\exp \left( \sum_{k \geq 1} \sum_{\beta \in NE(Y)} 
\sum_{g \geq 0}
2 \sin \left(\frac{k \hbar}{2} \right)
N_{g,k(m,n)}^{\beta} 
\hbar^{2g-1} t^{\beta} z^{-k(m,n)} 
\right)\,.
\end{equation}
\end{defn}

Note that by Lemma \ref{lem_curve_classes_finite}, we have indeed 
$f_{\fd_{(m,n)}} \in 
\Q[\![\hbar]\!][NE(Y)][\![z^{-(m,n)}]\!]$, as required by Definition 
\ref{defn_quantum_ray}.

\begin{defn}
We define a quantum scattering diagram 
$\fD_{\can}$ over $\Q[\![\hbar]\!][NE(Y)]$
by
\[ \fD_{\can} \coloneqq \{\fd_{m,n}\,|\,(m,n) \in B_0(\Z)\,,\,\gcd(m,n)=1\} \,.\]
We refer to $\fD_{\can}$ as the 
\emph{canonical quantum scattering diagram} defined by $(Y,D)$.
\end{defn}

The following Theorem \ref{thm_can_consistent} 
is the specialization to the case of the cubic surface $(Y,D)$ 
of one of the main results of \cite{MR4048291} on the consistency of 
canonical quantum scattering diagrams attached to log Calabi-Yau surfaces.

\begin{thm}[\cite{MR4048291}]
\label{thm_can_consistent}
The quantum scattering diagram $\fD_{\can}$ is consistent.
\end{thm}

In the following sections, we give an explicit 
description of the canonical quantum 
scattering diagram $\fD_{\can}$. 

\subsection{Curves on the cubic surface}
\label{section_curves}

Recall that lines in $\PP^3_{\kk}$ contained in $Y$
are exactly the $(-1)$-curves on $Y$.
A smooth projective cubic surface 
$Y \subset \PP^3_{\kk}$ contains $27$ lines, 
a classical result going back to Cayley
\cite{cayley1849triple} and Salmon 
\cite{salmon1849triple}
(see e.g. \cite[V.4]{Hartshorne}).
Three of these lines are $D_1$, $D_2$, $D_3$, 
whose union is the triangle $D$. 
It remains $24$ lines on $Y$ not contained in $D$.
By the adjunction formula, each of them intersect $D$
in a single point. One can easily check that for every $j \in \{1,2,3\}$, 
there are $8$ lines not containing in $D$ and intersecting $D_j$, 
that we write 
$L_{jk}$ for $1 \leq k \leq 8$. 

For convenience in describing curves classes of $Y$, 
we also fix for every $j \in \{1,2,3\}$ a pair of disjoint lines
$\{E_{j1},E_{j2}\} \subset \{L_{jm}\}_{1 \leq m \leq 8} $.
Contracting the $6$
$(-1)$-curves 
$E_{11}$, $E_{12}$, $E_{21}$, $E_{22}$, $E_{31}$, $E_{32}$,
gives a presentation of $Y$ 
as a blow-up of $\PP^2_\kk$ in $6$ points. 
We denote by $H \in NE(Y)$ the pullback of 
the class of a line in $\PP^2_\kk$. 
Note that for every $j \in \{1,2,3\}$, we have 
\begin{equation} 
D_j=H-E_{j1}-E_{j2}
\end{equation}

From the explicit 
description of the $27$ lines on 
$\PP^2$ blown-up in $6$ points as 
the $6$ exceptional divisors, the strict transforms of the $15$ 
lines passing through pairs of blown-up points, and the
$6$ strict transforms of the conics passing though 
$5$-tuples of blown-up points, 
we obtain the list of the classes of the $8$ lines $L_{1m}$ 
intersecting $D_1$ and distinct from $D_2$ and $D_3$.
We have $2$ exceptional divisors, 
$4$ strict transforms of a line, and 
$2$ strict transforms of a conic:
\begin{equation} \label{eq:lines_1}
\begin{aligned}
\{ L_{1m}\}_{1 \leq m \leq 8}
= \{ E_{11}\,, E_{12}\,,
H-E_{21}-E_{31}\,,
H-E_{21}-E_{32}\,,
H-E_{22}-E_{31}\,,
H-E_{22}-E_{32}\,, \\
2H-E_{11}-E_{21}-E_{22}-E_{31}-E_{32} \,,
2H-E_{12}-E_{21}-E_{22}-E_{31}-E_{32}
\}\,.
\end{aligned}
\end{equation}
Similarly, we have 
\begin{equation}
\begin{aligned} \label{eq:lines_2}
\{ L_{2m}\}_{1 \leq m \leq 8}
= \{ E_{21}\,, E_{22}\,,
H-E_{11}-E_{31}\,,
H-E_{11}-E_{32}\,,
H-E_{12}-E_{31}\,,
H-E_{12}-E_{32}\,, \\
2H-E_{21}-E_{11}-E_{12}-E_{31}-E_{32} \,,
2H-E_{22}-E_{11}-E_{12}-E_{31}-E_{32}
\}\,,
\end{aligned}
\end{equation}
and 
\begin{equation} \label{eq:lines_3}
\begin{aligned}
\{ L_{3m}\}_{1 \leq m \leq 8}
= \{ E_{31}\,, E_{32}\,,
H-E_{11}-E_{21}\,,
H-E_{11}-E_{22}\,,
H-E_{12}-E_{21}\,,
H-E_{12}-E_{22}\,, \\
2H-E_{31}-E_{11}-E_{12}-E_{21}-E_{22} \,,
2H-E_{32}-E_{11}-E_{12}-E_{21}-E_{22}
\}\,.
\end{aligned}
\end{equation}

For every $j \in \{1,2,3\}$, writing 
$\{1,2,3\}=\{j,k,\ell\}$, there are exactly two conics in $Y$  
tangent to $D_j$ and not intersecting $D_k \cup D_{\ell}$, 
that we write $\cC_{jk}$ for $1 \leq k \leq 2$. 
This comes from the  fact that there are two conics 
passing through $4$ points and tangent to a given line 
in $\PP^2_\kk$. The class in $NE(Y)$ of the two conics 
$\cC_{jk}$ for $1 \leq k \leq 2$ is 
\begin{equation} 
2H-E_{k1}-E_{k2}-E_{\ell 1}-E_{\ell 2}=D_k+D_{\ell} \,.
\end{equation}

\subsection{Contribution of the rays $\fd_j$: calculations in log Gromov-Witten theory}
\label{section_ray_rho_j_contribution}

For every $j \in \{1,2,3\}$, we write $\fd_j$ for 
$\fd_{m,n}$ where $v_j=(m,n)$. 
In the following Proposition
\ref{prop_ray_rho_j_contribution}, 
we compute explicitly the quantum rays  
$\fd_j$ of $\fD_{\can}$.

\begin{prop}\label{prop_ray_rho_j_contribution} 
For every $j \in \{1,2,3\}$,
writing $\{1,2,3\}=\{j,k,\ell\}$ , the quantum ray $\fd_j=(v_j, f_{\fd_j})$ 
in $\fD_\can$
satisfies  
\begin{equation} \label{eq:ray_rho_j_contribution}
f_{\fd_j}
= \frac{\prod_{m=1}^8(1+t^{L_{jm}}z^{-v_j})}{(1-q^{-1} t^{D_k+D_{\ell}}z^{-2v_j})
(1- t^{D_k+D_{\ell}}z^{-2v_j})^2
(1-q t^{D_k+D_{\ell}}z^{-2v_j})} 
\end{equation}
where $q=e^{i\hbar}$.
\end{prop}

The proof of Proposition \ref{prop_ray_rho_j_contribution}
takes the remainder of Section \ref{section_ray_rho_j_contribution}.
We will show that the numerator 
of $f_{\fd_j}$ is the contribution of 
multicovers of the lines 
$L_{jm}$ for 
$1 \leq m \leq 8$, that the denominator of $f_{\fd_j}$ 
is the contribution of the multicovers of 
the conics $\cC_{jk}$ for 
$1 \leq k \leq 2$, and that no other 
curve classes contribute to 
$f_{\fd_j}$.
Given the cyclic $\Z/3\Z$-symmetry permuting $\{1,2,3\}$, 
we can assume 
$j=1$, $k=2$, $\ell=3$.

\begin{lem} \label{lem:gw_lines}
For every $1 \leq j \leq 8$ and $k \in \Z_{\geq 1}$,
we have 
\begin{equation}
\sum_{g \geq 0}
N_{g,kv_1}^{kL_{1j}} \hbar^{2g-1} 
= \frac{(-1)^{k-1}}{k} \frac{1} {2\sin \left( \frac{k \hbar}{2} \right)} \,.
\end{equation}
\end{lem}

\begin{proof}
For every $1 \leq j \leq 8$, the line $L_{1j}$ is a $(-1)$-curve, 
so rigid unique representative of its curve class. 
Hence, every stable log map of class $k L_{1j}$ factors through $L_{1j}$. 
Therefore, we can compute $N_{g,kv_1}^{kL_{1j}}$ 
as some integral over a moduli space of stable log maps to 
$L_{1j} \simeq \PP^1$. More precisely, we have 
\begin{equation}
\label{eq:gw_P1} 
N_{g,kv_1}^{kL_{1j}}=\int_{[M_{g,k}]^{\virt}} 
e(R^1 \pi_* f^*(\cO_{\PP^1} \oplus \cO_{\PP^1}(-1)) \,,
\end{equation}
where $M_{g,k}$ is the moduli space of genus $g$ degree $k$
stable log maps to $\PP^1$ compactifying the 
moduli space of stable maps fully ramified over a given point 
$\infty \in \PP^1$, $\pi \colon \cC \rightarrow M_{g,k}$ 
is the universal curve and $f \colon \cC \rightarrow \PP^1$ 
is the universal stable log map. The insertion $R^1 \pi_*f^* \cO_{\PP^1}$
comes by Serre duality from the insertion 
$(-1)^g \lambda_g$ in Equation \eqref{eq:log_gw}, 
and the insertion $R^1 \pi_*f^* \cO_{\PP^1}(-1)$ 
comes from the comparison of the obstruction theories for stable maps to $Y$ 
with stable maps to $L_j \simeq \PP^1$, using that the normal bundle of
$L_j$ in $Y$ is $\cO_{\PP^1}(-1)$. 
The integral in the right-hand side of Equation \eqref{eq:gw_P1} 
was computed by Bryan and Pandharipande \cite[Theorem 5.1]{MR2115262} 
in relative Gromov-Witten theory of $(\PP^1,\infty)$, 
which is known to be equivalent to 
log Gromov-Witten theory of $(\PP^1,\infty)$ by 
\cite{abramovich2012comparison}.
\end{proof}

\begin{lem} \label{lem:gw_conics}
For every $k \in \Z_{\geq 1}$, we have 
\begin{equation} \sum_{g \geq 0} N_{g,2kv_1}^{k(D_2+D_3)} \hbar^{2g-1}
=\frac{1}{k} \frac{2 \cos \left( \frac{k\hbar}{2}\right)}{2 \sin \left( \frac{k \hbar}{2} \right)}\,.
\end{equation}
\end{lem}

\begin{proof}
The linear system of curves in $Y$ of class 
\[ D_2+D_3=2H-E_{21}-E_{22}-E_{31}-E_{34}\]
is one-dimensional, made of strict transforms of conics in 
$\PP^2$ passing through four given points.
The only curves of class $D_2+D_3$ tangent with 
$D_1$ are the two conics $\cC_{11}$ and $\cC_{12}$.
Hence, every stable log map in the moduli space 
$\overline{M}_{g,2v_1}^{k(D_2+D_3)}(Y/D)$ factors 
through either $\cC_{11}$ or $\cC_{12}$. 
However, the required multicover computation is more complicated 
that the one used in
Lemma \ref{lem:gw_lines} and in fact has done been 
done previously directly in Gromov-Witten theory.
We will follow a slightly roundabout path and 
use a general correspondence theorem between log Gromov-Witten invariants 
of log Calabi-Yau surfaces and quiver Donaldson-Thomas invariants 
proved in  \cite{bousseau2018quantum_tropical}.

As no curve of class $D_2+D_3$ intersect $E_{11}$ or $E_{21}$,
we can contract the two $(-1)$-curves
$E_{11}$ and $E_{21}$ and compute the invariants 
$N_{g,2kv_1}^{k(D_2+D_3)}$ on the resulting surface $Y'$. 
Following Section 8.5 of \cite{bousseau2018quantum_tropical}, 
we can attach to $Y'$ a quiver $Q_{Y'}$:
vertices of $Q_{Y'}$ in one-to-one correspondence 
with the exceptional divisors 
$E_{21}$, $E_{22}$, $E_{31}$, $E_{32}$. 
As $\langle v_2, v_3 \rangle =1$, 
we have an edge from every vertex corresponding to $E_{21}, E_{22}$ 
to every vertex corresponding to $E_{31}$, $E_{32}$.

\begin{figure}
\begin{center}
\begin{tikzpicture}[>=angle 90]
\matrix(a)[matrix of math nodes,
row sep=3em, column sep=5em,
text height=1.5ex, text depth=0.25ex]
{V_{21}& &V_{31}\\
V_{22}& &V_{32}\\};
\path[->](a-1-1) edge node[above]{$f_1$} (a-1-3);
\path[->](a-1-1) edge node[above]{$f_2$} (a-2-3);
\path[->](a-2-1) edge  node[below]{$f_4$} (a-2-3);
\path[->](a-2-1) edge  node[below]{$f_3$} (a-1-3);
\end{tikzpicture}
\end{center}
\caption{The quiver $Q_{Y'}$}
\end{figure}
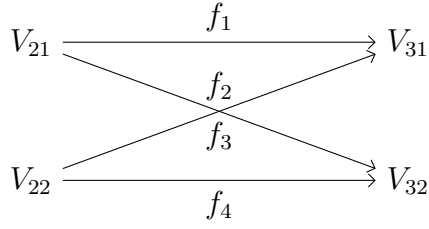

As in Section 8.5 of \cite{bousseau2018quantum_tropical}, let 
$M_k^{ss}$ (resp. $M_k^{st}$) 
be the moduli space of semistable (resp. stable) 
representations of $Q_{Y'}$ 
of dimension vector $(k,k,k,k)$, 
where we consider the 
maximally non-trivial stability condition. Write 
$\iota \colon M_k^{st} \hookrightarrow M_k^{ss}$ 
for the natural inclusion and define 
\begin{equation} \Omega_k^{Q_{Y'}}(q^{\frac{1}{2}})\coloneqq (-q^{\frac{1}{2}})^{-\dim M_k^{st}}
\sum_{j=0}^{\dim M_k^{ss}} \dim H^{2j}(M_k^{ss},\iota_{!*}
\Q_{M_k^{st}} ) q^j\,, 
\end{equation}
where $\iota_{!*}$ is the intermediate extension functor.
Applying Theorem 8.13 of \cite{bousseau2018quantum_tropical}, we obtain
\begin{equation} \label{eq:dw_dt} \sum_{g \geq 0} N_{g,2kv_1}^{k(D_2+D_3)} \hbar^{2g-1}
=- \sum_{k=\ell k'} \frac{1}{\ell}
\frac{\Omega_{k'}^{Q_{Y'}}(q^{\frac{\ell}{2}})}{2 \sin \left( \frac{\ell \hbar}{2} \right)} \,,
\end{equation}
where $q=e^{i\hbar}$.

We have $M_1^{st}=M_1^{sst}=\PP^1$ and so 
\begin{equation}\label{eq:dt_1}
\Omega_1^{Q_{Y'}}(q^{\frac{1}{2}})=-q^{-\frac{1}{2}}-q^{\frac{1}{2}}\,.
\end{equation} On the other hand,
$M_k^{st}$ is empty for $k>1$ and so 
\begin{equation} \label{eq:dt_2}\Omega_k^{Q_{Y'}}(q^{\frac{1}{2}})=0
\end{equation} 
for $k>1$. To check that $M_k^{st}$ is empty for $k>1$, 
one can argue as follows.
Given a representation $(V_{21},V_{22},V_{31},V_{32},f_1,f_2,f_3,f_4)$
of $Q_{Y'}$, one constructs a representation 
$(V_{21} \oplus V_{22}, V_{31} \oplus V_{32}, 
f_1 \oplus f_4, f_2 \oplus f_3)$ of the Kronecker quiver.
One then uses the fact that, 
by the classification of representations of the Kronecker quiver 
(see for example Section 1.8 
of \cite{gabriel1997representations}), 
 every representation of dimension $(n,n)$ 
of the Kronecker quiver contains a subrepresentation of dimension $(1,0)$.

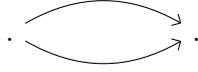
\begin{figure}
\begin{center}
\begin{tikzpicture}[>=angle 90]
\matrix(a)[matrix of math nodes,
row sep=3em, column sep=5em,
text height=1.5ex, text depth=0.25ex]
{.& .\\};
\path[->](a-1-1) edge [bend left] (a-1-2);
\path[->](a-1-1) edge [bend right] (a-1-2);
\end{tikzpicture}
\end{center}
\caption{The Kronecker quiver}
\end{figure}

Lemma \ref{lem:gw_conics} follows by combination of Equations 
\eqref{eq:dw_dt}-\eqref{eq:dt_1}-\eqref{eq:dt_2}.

\end{proof}

\begin{lem} \label{lem_homology_classification_1}
Let $k \in \Z_{\geq 1}$ and $\beta \in NE(Y)$ 
such that there exists $g \in \Z_{\geq 0}$ and a stable log map 
$(f \colon C \rightarrow Y) \in \overline{M}_{g,kv_1}^{\beta}(Y/D)$ 
with $C$ irreducible and $f$ generically injective. 
Then, we have either $\beta=L_{1j}$ for some $1 \leq j \leq 8$ and $k=1$, or 
$\beta=D_2+D_3$ and $k=2$.
\end{lem}

\begin{proof}
We write $\beta=aL-\sum_{j=1}^3 
\sum_{m=1}^2 b_{jm} E_{jk}$ 
with $a, b_{jm} \in \Z$.
As $\overline{M}_{g,kv_1}^{\beta}(Y/D)$ is not empty, we have 
$\beta \cdot D_1 =k$, $\beta \cdot D_2=0$, 
$\beta \cdot D_3=0$, that is $a-b_{11}-b_{12}=k$, $a-b_{21}-b_{22}=0$, 
$a-b_{31}-b_{32}=0$. As $C$ is irreducible and $f$ is generically injective, 
the image $f(C)$ is an integral curve of class $\beta$. 
In particular, the 
arithmetic genus $p_a(f(C))$ of $f(C)$ is nonnegative, 
and so, by the adjunction formula, we have 
\begin{equation}
-2 \leq 2p_a(f(C))-2 = \beta \cdot (\beta +K_Y)
=\beta \cdot (\beta-D_1)=a^2-\sum_{j=1}^3
\sum_{m=1}^2 b_{jm}^2 -k \,.
\end{equation}
The classes $\beta$ satisfying these constraints 
are classified in the first part of 
the proof of \cite[Proposition 2.4]{gross2019cubic} 
by an argument which does not use the assumption $g=0$ 
done in \cite{gross2019cubic}. 
\end{proof}

\begin{lem} \label{lem_tropical_balancing_1}
Let $g \in \Z_{\geq 0}$, $k \in \Z_{\geq 1}$, and $\beta \in NE(Y)$.
Let $f \colon C/W \rightarrow Y$ be a stable log map defining a point of 
$\overline{M}_{g,kv_1}^{\beta}(Y/D)$. Then $f(C) \cap D_2 = \emptyset$
and 
$f(C) \cap D_3 = \emptyset$.
\end{lem}

\begin{proof}
The proof relies on the study of the tropicalization 
of $f \colon C \rightarrow Y$. 
We refer to \cite[Section 2.5]{abramovich2017decomposition} 
for details on tropicalization of stable log maps.

Let
\begin{center}
\begin{tikzcd}
C \arrow{r}{f} \arrow{d}{\pi}
& Y\\
W &  
\end{tikzcd}
\end{center}
be a stable log map defining a point of 
$\overline{M}_{g,kv_1}^{\beta}(Y/D)$. Here 
$W$ is a log point 
\[ (\Spec \kk, \overline{M}_W \oplus \kk^{\times})\]
defined by some monoid $\overline{M}_W$.
Taking the tropicalization, we obtain a diagram of cone complexes
\begin{center}
\begin{tikzcd}
\Sigma(C) \arrow{r}{\Sigma(f)} \arrow{d}{\Sigma(\pi)}
& \Sigma(Y)\\
\Sigma(W) & 
\end{tikzcd}
\end{center}

As cone complexes, we have $\Sigma(Y) \simeq (B,\Sigma)$. 
On the other hand, 
$\Sigma(\pi) \colon \Sigma(C) \rightarrow \Sigma(W)$ 
is a family of tropical curves parametrized by the cone 
$\Sigma(W) = \Hom(\overline{M}_W,\R_{\geq 0})$. 
We pick a point $w$ in the interior of $\Sigma(W)$ 
and we denote by 
$\Gamma$ the fiber $\Sigma(\pi)^{-1}(w)$. 
The graph underlying the tropical curve $\Gamma$ 
is the dual graph of $C$: $\Gamma$ has a single unbounded edge, 
corresponding to the marked point on $C$, 
vertices of $\Gamma$ are in one-to-one correspondence with 
irreducible components of $C$, 
and bounded edges of $\Gamma$ are in one-to-one 
correspondence with nodes of $C$. 
We denote by $h \colon \Gamma \rightarrow B$
the restriction of $\Sigma(f)$ to $\Gamma = \Sigma(\pi)^{-1}(w)$.
The image $h(E)$ of every edge $E$ of $\Gamma$ 
is a line segment of rational slope in a cone of 
$\Sigma$. In addition, if $h(E)$ is not a point, 
the line segment $h(E)$
is decorated by a weight $w(E) \in \Z_{>0}$. 
For example, denoting by 
$E_\infty$ the unique unbounded edge of $\Gamma$, $h(E_\infty)$ 
is a half-line of direction $v_1$ and weight $k$.

For every $j \in \{1,2,3\}$, the formal completion of $D_j$ in $Y$ 
is isomorphic to the formal completion of a toric divisor 
in a toric surface, and
the formal completion of $(D_j \cap D_{j+1})$ in $Y$ 
is isomorphic to the formal completion of a 
$0$-dimensional toric stratum in a toric surface. 
Furthermore, the integral affine structure on $B_0$ 
has been defined based on these toric descriptions. 
Therefore, it follows from the general balancing 
condition for stable log maps given in \cite[Proposition 1.5]{MR3011419} 
that the toric balancing condition holds on $B_0$: 
for every vertex $V$ of $\Gamma$ with $h(V) \in B_0$, 
denoting $E_\ell$ the edges of $\Gamma$ adjacent to $V$ 
and not contracted to a point by $\Gamma$, and $u_{V,E_\ell}$ 
the primitive integral direction of $h(E_\ell)$ 
pointing outwards $h(V)$, we have $\sum_\ell w(E_\ell) u_{V,E_\ell}=0$. 
We do not have such simple balancing condition at $0 \in B$: 
the integral affine structure is singular at $0$ 
due to the fact that the surface $Y$ is not toric.

If $f(C) \cap D_2 \neq \emptyset$, then either $C$ 
has a component dominating $D_2$ and then $\Gamma$ 
has a vertex $V$ with $h(V) \in \Int(\rho_2)$, or 
$C$ has a component non-dominating but intersecting $D_2$ 
and then it follows from the general balancing condition of  
\cite[Proposition 1.5]{MR3011419} that $\Gamma$ 
has an edge $E$ intersecting 
$\Int(\sigma_{2,3} \cup \rho_2 \cup \sigma_{1,2})$.
Similarly, if $f(C) \cap D_3 \neq \emptyset$, 
then either $C$ has a component dominating $D_3$ 
and then $\Gamma$ has a vertex $V$ with $h(V) \in \Int(\rho_3)$, 
or 
$C$ has a component non-dominating but intersecting $D_3$ 
and then 
$\Gamma$ has an edge $E$ intersecting 
$\Int(\sigma_{3,1} \cup \rho_3 \cup \sigma_{2,3})$.
Therefore, in order to prove Lemma \ref{lem_tropical_balancing_1}, 
it is enough to show that $h(V)$ belongs to the ray $\rho_1$ 
for every $V$ vertex of $\Gamma$. 
It will automatically imply that $h(E) \subset \rho_1$ 
for every edge $E$ of $\Gamma$.

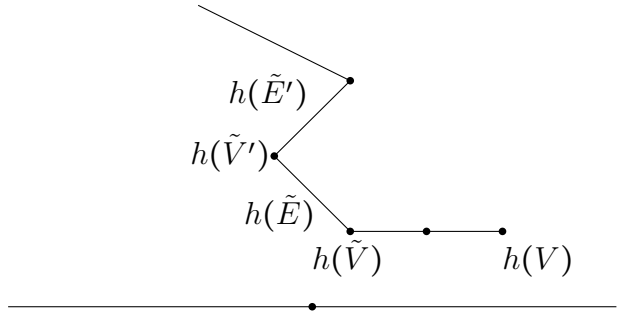
\begin{figure}[h]
\centering
\setlength{\unitlength}{1cm}
\begin{picture}(6,5)
\put(3,0.5){\circle*{0.1}}
\put(3,0.5){\line(1,0){4}}
\put(3,0.5){\line(-1,0){4}}
\put(5.5,1.5){\circle*{0.1}}
\put(4.5,1.5){\circle*{0.1}}
\put(3.5,1.5){\circle*{0.1}}
\put(3.5,1.5){\line(1,0){2}}
\put(5.5,1){$h(V)$}
\put(3,1){$h(\tilde{V})$}
\put(2.1,1.6){$h(\tilde{E})$}
\put(3.5,1.5){\line(-1,1){1}}
\put(2.5,2.5){\circle*{0.1}}
\put(1.4,2.4){$h(\tilde{V}')$}
\put(1.9,3.2){$h(\tilde{E}')$}
\put(2.5,2.5){\line(1,1){1}}
\put(3.5,3.5){\circle*{0.1}}
\put(3.5,3.5){\line(-2,1){2}}
\end{picture}
\caption{Illustration of the proof of Lemma 
\ref{lem_tropical_balancing_1}}
\label{figure: illustration}
\end{figure}

We recall that we use the upper half-plane description of $B$ 
given by Figure \ref{figure: tropicalization}. 
In particular, 
we will refer to this description when using notions of horizontal, 
vertical, left and right.
We argue by contradiction by assuming that there exists a vertex $V$ 
of $\Gamma$ with $h(V) \notin \rho_1$. 
In particular, we have $h(V) \in B_0$ and 
so the toric balancing condition holds at $h(V)$. 

We claim that there exists a vertex $\tilde{V}$ 
of $\Gamma$ such that $h(\tilde{V}) \in B_0$
and a edge $\tilde{E}$ adjacent to $\tilde{V}$ 
such that $u_{\tilde{V},\tilde{E}}$ has positive vertical component. 
Indeed, if it were not the case, 
then $h(\Gamma)$ would be entirely contained in an horizontal line in $B_0$. 
As $\Gamma$ has a unique 
unbounded edge, the toric balancing condition 
cannot hold at both the most left 
and most right vertices of $h(\Gamma)$ and we obtain a contradiction.

The unique unbounded edge of $\Gamma$ being horizontal, 
the edge $\tilde{E}$
is bounded. Let $\tilde{V}'$ be the other vertex of $\tilde{E}$. 
As $u_{\tilde{V},\tilde{E}}$ has positive vertical component, 
the vertical coordinate of $h(\tilde{V}')$ is strictly bigger 
than the one of $h(\tilde{V})$. 
In particular, we have $h(\tilde{V}') \in B_0$, 
the toric balancing condition holds at $h(\tilde{V}')$, 
and so there exists an edge $\tilde{E}'$ adjacent to 
$\tilde{V}'$ such that $u_{\tilde{V}',\tilde{E}'}$ 
has positive vertical component. 
Therefore, we can iterate 
the argument by replacing $(\tilde{V},\tilde{E})$
by $(\tilde{V}',\tilde{E}')$. 
Successive iterations produce infinitely many vertices 
of $\Gamma$, in contradiction with 
the finiteness of the set of vertices of $\Gamma$.
\end{proof}

\begin{lem} \label{lem_tropical_balancing_2}
Let $g \in \Z_{\geq 0}$, $k \in \Z_{\geq 1}$, 
and $\beta \in NE(Y)$.
Let $f \colon C/W \rightarrow Y$ be a stable log map defining a point of 
$\overline{M}_{g,kv_1}^{\beta}(Y/D)$, 
and let $p \in C$ be the corresponding marked point. 
Then, $f(C)$ intersects $D_1$ in a single point, i.e.\ 
$f(C) \cap D_1=\{f(p)\}$, 
and the set $f^{-1}(p)$ of points of $C$ mapping to 
$f(p)$ is connected.
\end{lem}

\begin{proof}
As in the proof of Lemma \ref{lem_tropical_balancing_1}, 
we attach to
$f \colon C/W \rightarrow Y$ a tropical curve $h \colon \Gamma \rightarrow B$.
By Lemma \ref{lem_tropical_balancing_1}, 
we have $h(\Gamma) \subset \rho_1$.
Let $C_\infty$ be the irreducible component 
of $C$ containing the marked point
$p$ and let $V_\infty$ be the corresponding vertex of $\Gamma$. 
The unique unbounded edge $E_\infty$ of $\Gamma$, 
corresponding to the marked point, is attached to $V_\infty$.
Let $C_0$ be an irreducible component of $C$ with 
$f(C_0) \cap D_1 \neq \emptyset$, and let $V_0$ be the corresponding vertex 
of $\Gamma$.
We have to show that $f(C_0) \cap D_1=f(p)$.
By Lemma \ref{lem_tropical_balancing_1}, 
no component of $C$ dominates 
$D_1$. In particular, either $f(C_0)$ is generically contained in 
$Y \backslash D_1$, or $f(C_0)$ is a point on $D_1$.

Let us first assume that $f(C_0)$ is generically contained in 
$Y \backslash D_1$, that is $h(V_0)=\{0\}$. 
Let $x \in f^{-1}(D_1) \cap C_0$. 
By the balancing condition of \cite[Proposition 1.5]{MR3011419}, 
$x$ defines an edge $E$ of $\Gamma$ with 
$h(\Int(E)) \subset \Int(\rho_1)$. If $E=E_\infty$, then $x=p$.
If $E \neq E_\infty$, then $E$ is bounded. 
In this case, let 
$V_1$ be the other vertex of $E$. We have $h(V_1) \in \Int(\rho_1)$.
 As $E_\infty$ is the unique unbounded edge of $\Gamma$, 
 it follows from the toric balancing condition 
 that there exists a path $\gamma$ in $\Gamma$, 
 connecting $V_1$ to $V_\infty$, and such that $h(\gamma) \subset \Int(\rho_1)$. 
Let $C_\gamma \subset C$ be the union of irreducible components of $C$ 
corresponding to the vertices of $\Gamma$ contained in 
$\gamma$. As no component of $C$ dominates $D_1$, 
the connected curve $C_\gamma$ is entirely contracted to a point by $f$. 
Therefore,
we have $f(x)=f(C_\gamma)=f(p)$.

If $f(C_0)$ is a point on $D_1$, we make a similar argument. 
We have 
$h(V_0) \in \Int(\rho_1)$. 
 As $E_\infty$ is the unique unbounded edge of $\Gamma$, 
 it follows from the toric balancing condition 
 that there exists a path $\gamma$ in $\Gamma$, 
 connecting $V_0$ to $V_\infty$, and such that $h(\gamma) \subset \Int(\rho_1)$. 
 Let $C_\gamma \subset C$ be the union 
 of irreducible components of $C$ corresponding 
 to the vertices of $\Gamma$ contained in 
$\gamma$. 
As no component of $C$ dominates $D_1$, 
the connected curve $C_\gamma$ is entirely 
contracted to a point by $f$. Therefore,
we have $f(C_0)=f(C_\gamma)=f(p)$.

The two previous paragraphs also show that every point  
$x \in C$ such that $f(x) \in D_1$ is connected to $C_\infty$
by a chain of irreducible components of $C$ all contracted to a point in $D_1$. 
In particular, the set $f^{-1}(D_1)$ 
of points of $C$ mapped to $D_1$ is connected.

\end{proof}

\begin{lem} \label{lem_lambda_vanishing}
Let $g \in \Z_{\geq 0}$, $k \in \Z_{\geq 1}$ 
and $\beta \in NE(Y)$ be such that, 
for every stable log map $(f \colon C \rightarrow Y) \in \overline{M}_{g,kv_1}^{\beta}(Y/D)$, 
the dual intersection graph of $C$ has positive genus. 
Then, we have $N_{g,kv_1}^{\beta}=0$. 
\end{lem}

\begin{proof}
Recall that, by definition, we have 
\[ N_{g,kv_1}^{\beta} \coloneqq 
\int_{[\overline{M}_{g,kv_1}^\beta(Y/D)]^{\virt}} (-1)^g \lambda_g \,.\]
The class $\lambda_g$ vanishes for families of stable curves with dual graph 
of positive genus. This classical fact is for example reviewed in 
\cite[Section 3]{MR3904449}.
\end{proof}

\begin{lem} \label{lem_homology_classification_2}
Let $g \in \Z_{\geq 0}$, $k \in \Z_{\geq 1}$ and $\beta \in NE(Y)$.
Let $(f \colon C \rightarrow Y)$ be a stable log map 
defining a point of $\overline{M}_{g,kv_1}^{\beta}(Y/D)$ 
and such that the dual intersection graph of $C$ has genus $0$. 
Then, for every irreducible component 
$C_0$ of $C$ on which $f$ is not constant, $f^{-1}(D_1) \cap C_0$
consists of a single point.
\end{lem}

\begin{proof}
By Lemma \ref{lem_tropical_balancing_1}, we have 
$f(C_0) \cap D_2=f(C_0) \cap D_3=\emptyset$. 
As $f$ is non-constant on $C_0$ and $-K_Y=D_1 \cup D_2 \cup D_3$ is ample, 
we deduce that $f^{-1}(D_1) \cap C_0$ is non-empty.

On the other hand, by Lemma \ref{lem_tropical_balancing_2}, 
the set 
$f^{-1}(p)$ of points of $C$ mapping to $f(p)$ is connected. 
As the dual graph of $C$ is of genus $0$, 
we obtain that $C_0$ intersects $f^{-1}(p)$
in at most one point.
\end{proof}

\begin{lem} \label{lem_homology_classification_3}
Let $k \in \Z_{\geq 1}$ and 
$\beta \in NE(Y)$ such that there exists 
$g \in \Z_{\geq 0}$ with $N_{g,kv_1}^{\beta} \neq 0$. 
Then, we have either $\beta=k L_{1j}$ for some
$1 \leq j \leq 8$, or $k$ is even and 
$\beta=\frac{k}{2}(D_2+D_3)$.
\end{lem}

\begin{proof}
As $N_{g,kv_1}^\beta \neq 0$, there exists
by Lemma \ref{lem_lambda_vanishing}
a stable log map $(f \colon C \rightarrow Y) \in \overline{M}_{g,kv_1}^\beta (Y/D)$
such that the dual intersection graph of 
$C$ has genus $0$.
We denote by $p \in C$ the marked point.

Let $C_1,\dots, C_n$ the irreducible components of $f(C)$
equipped with the reduced scheme structure. 
For every $1 \leq \ell \leq n$, $C_j$ is an integral curve in $Y$. 
Denoting $\beta_\ell \coloneqq [C_\ell]$, we have 
$\beta=\sum_{\ell=1}^n m_\ell [C_\ell]$, where $m_\ell \in \Z_{\geq 1}$
is the multiplicity of $C_\ell$ in the cycle 
$[f(C)]$.
By Lemma \ref{lem_tropical_balancing_1}, we have 
$f(C_\ell) \cap D_2=f(C_\ell) \cap D_3= \emptyset$.
By Lemma \ref{lem_homology_classification_2}, we have $C_\ell 
\cap D_1 = f(p)$ and $C_\ell$ is unibranch at the point $f(p)$.
Therefore the normalization $\tilde{C}_\ell$ of $C_\ell$ 
defines a stable log map 
$(f_\ell \colon \tilde{C}_\ell \rightarrow Y) 
\in \overline{M}_{g_\ell,k_\ell v_1}(Y/D)$, 
where $g_\ell$
is genus of $C_\ell$ and $k_\ell \coloneqq \beta_\ell \cdot D_1$.
As $\tilde{C}_\ell$ is irreducible and $f_\ell$ 
is generically injective, 
we can apply Lemma \ref{lem_homology_classification_1} and so $C_\ell$
is either of the $8$ lines 
$L_{1m}$ or one of the two conics 
$\cC_{1k}$. It is shown in the proof of 
\cite[Proposition 2.4]{gross2019cubic} that for general $Y$, 
the 10 curves of 
$L_{1m}$ for 
$1 \leq m \leq 8$ and 
$\cC_{1k}$ for 
$1 \leq k \leq 2$ intersect $D_1$ in different points.
By deformation invariance of log Gromov-Witten invariants, 
we can assume that $Y$ is general. 
Therefore, we have in fact $n=1$ and 
$f \colon C \rightarrow Y$ is a multiple cover 
of one of the 10 curves $L_{1m}$ for 
$1 \leq m \leq 8$ and 
$\cC_{1k}$ for 
$1 \leq k \leq 2$.
\end{proof}

We can now end the proof of Proposition 
\ref{prop_ray_rho_j_contribution}.
From Equation \eqref{eq:rays_can} 
and Lemma \ref{lem_homology_classification_3}, 
we have 
\[ f_{\fd_1}
= \prod_{j=1}^8 \exp \left( \sum_{k \geq 1}\sum_{g \geq 0}
2 \sin \left( 
\frac{k \hbar}{2} \right) N_{g,kv_1}^{kL_{1j}} \hbar^{2g-1}
t^{kL_{1j}} z^{-kv_1} \right) 
\]
\[ \times \exp \left( \sum_{k \geq 1}\sum_{g \geq 0}
2 \sin \left(k \hbar \right) N_{g,2kv_1}^{k(D_2+D_3)} \hbar^{2g-1}
t^{k(D_2+D_3)} z^{-2kv_1} \right) \,.
\]
By Lemma \ref{lem:gw_lines}, we have 
\[ \exp \left( \sum_{k \geq 1}\sum_{g \geq 0}
2 \sin \left( 
\frac{k \hbar}{2} \right) N_{g,kv_1}^{kL_{1j}} \hbar^{2g-1}
t^{kL_{1j}} z^{-kv_1} \right)  = \exp \left( \sum_{k \geq 1}\frac{(-1)^{k-1}}{k}t^{kL_{1j}} z^{-kv_1} \right) = 1+t^{L_{1j}}z^{-v_1} \,,\]
thus producing the numerator of 
Equation \eqref{eq:ray_rho_j_contribution}.
On the other hand, by Lemma \ref{lem:gw_conics} and setting 
$q=e^{i\hbar}$, we have 
\[\exp \left( \sum_{k \geq 1}\sum_{g \geq 0}
2 \sin \left( 
k \hbar\right) N_{g,2kv_1}^{k(D_2+D_3)} \hbar^{2g-1}
t^{k(D_2+D_3)} z^{-2kv_1} \right) \]
\[ = \exp \left( \sum_{k \geq 1} \frac{(q^k-q^{k})(q^{\frac{k}{2}}
+q^{-\frac{k}{2}})}{k(q^{\frac{k}{2}}-q^{-\frac{k}{2}})}
t^{k(D_2+D_3)}z^{-2kv_1} \right) 
= \exp \left( \sum_{k \geq 1} \frac{q^k+2+q^{-k}}{k}
t^{k(D_2+D_3)}z^{-2kv_1} \right)\]
\[=\frac{1}{(1-q^{-1}t^{D_2+D_3}z^{-2v_1})
(1-t^{D_2+D_3}z^{-2v_1})^2(1-qt^{D_2+D_3}z^{-2v_1})}\,,\]
thus producing the denominator of Equation \eqref{eq:ray_rho_j_contribution}.
This concludes the proof of Proposition 
\ref{prop_ray_rho_j_contribution}.

\subsection{Contribution of general rays: $PSL_2(\Z)$ symmetry}
\label{section_contribution_general_rays}

Following Gross, Hacking, Keel and Siebert \cite{gross2019cubic}  
treating the classical case,
we describe the general quantum rays $\fd_{m,n}$ 
of the canonical quantum scattering diagram $\fD_{\can}$ in terms 
of the quantum rays $\fd_j$ computed in Proposition 
\ref{prop_ray_rho_j_contribution} and of a $PSL_2(\Z)$ symmetry.

\begin{lem} \label{lem_tropical_balancing_3}
Let $g \in \Z_{\geq 0}$, $v \in B_0(\Z)$ and $\beta \in NE(Y)$. 
Let $f \colon C/W \rightarrow Y$ be a stable log map defining a point in 
$\overline{M}_{g,v}^{\beta}(Y/D)$, 
and let $p \in C$ be the corresponding marked point. 
Then, 
$f(C)$ intersects $D$ in a single point, i.e.\ $f(C) \cap D = \{f(p)\}$.
\end{lem}

\begin{proof}
We proved this result in Lemmas 
\ref{lem_tropical_balancing_1} and 
\ref{lem_tropical_balancing_2} by a tropical argument
when $v$ is a multiple of $v_1$. 
Exactly the same tropical argument can be applied in general: 
up to rotating the chart that we are using to describe $B$, 
we can assume that $\R_{\geq 0}v$ is the horizontal direction.
\end{proof}

First, the linear action of $SL_2(\Z)$ on $\Z^2$ induces an action of 
$PSL_2(\Z)$ on $B(\Z)=\Z^2/\langle -\id \rangle$.
Then, we
define an action of $PSL_2(\Z)$ on the set
\begin{equation}
\Gamma \coloneqq \{ \beta \in NE(Y) \,|\, 
N_{g,v}^{\beta} \neq 0 \,\,\text{for some} \,\, g\in \Z_{\geq 0}
\,\, \text{and} \,\, v\in B_0(\Z) \}
\end{equation}
Note that $A_1(Y) \simeq \Z^7$ has for basis
$H$, $E_{11}$, $E_{12}$, 
$E_{21}$, $E_{22}$, $E_{31}$, $E_{32}$, 
and that $PSL_2(\Z)$ is generated by 
\begin{equation} \label{eq:ST}
 S= \begin{pmatrix}
0 & -1 \\
1 & 1
\end{pmatrix}
\,\,\,\, \text{and} \,\,\,\,
T= \begin{pmatrix}
1 & 1 \\
0 & 1
\end{pmatrix}\,.
\end{equation}
We define an action $S^{*}$ of $S$ on $A_1(Y)$ is defined by 
\begin{equation} \label{eq:S_action}
S^{*}(H)=H \,\,\,\, \text{and}
\,\,\,\, 
S^{*}(E_{jk})=E_{j+1,k}\,,
\end{equation}

\begin{lem} \label{lem_S_preserving}
The transformation 
$S^{*}$ of $A_1(Y)$ preserves $\Gamma$. 
Moreover, for every $g \in \Z_{\geq 0}$, $v \in B(\Z)$ 
and $\beta \in \Gamma$, we have 
\begin{equation}
N_{g,v}^{S^{*} \beta}=N_{g,v}^\beta \,.
\end{equation}
\end{lem}

\begin{proof}
The transformation $S^{*}$ is induced by the obvious
$\Z/3\Z$-cyclic symmetry of $(Y,D)$ permuting the components 
$D_1$, $D_2$, $D_3$ of $D$, and so the result is clear.
\end{proof}

Let $T^{*}$ be the transformation of $A_1(Y)$ defined by 
\begin{equation} \label{eq:T_action}
T^{*}(H)=2H-E_{31}-E_{32}\,,\,\, 
T^{*}(E_{1j})=E_{1j}\,,\,\, 
T^{*}(E_{2j})=H-E_{3j}\,,\,\, 
T^{*}(E_{3j})=E_{2j}\,.
\end{equation}
Note that $T^{*}$ does not define an action of $T$ on $A_1(Y)$ 
because $T^{*}$ is not bijective. 
Nevertheless, we have the following result.

\begin{lem} \label{lem_T_bijective}
The transformation $T^{*}$ of $A_1(Y)$
preserves $\Gamma$ and the restriction of $T^{*}$ to $\Gamma$ is bijective.
Moreover, for every $g \in \Z_{\geq 0}$, $v \in B(\Z)$ 
and $\beta \in \Gamma$, we have 
\begin{equation}
N_{g,v}^{T^{*} \beta}=N_{g,v}^\beta \,.
\end{equation}
\end{lem}

\begin{proof}
It is shown
in \cite{gross2019cubic}
that the transformation $T^{*}$ of $A_1(Y)$ 
is induced by a log birational modification of the pair $(Y,D)$: 
given $(Y,D)$, one can blow-up the point $D_1 \cap D_2$ 
and contract $D_3$ to obtain a new pair $(Y',D')$.
By Lemma \ref{lem_tropical_balancing_3}, 
a class $\beta \in \Gamma$ is represented by a curve in $Y$ 
whose all components are generically 
contained in the complement of $D$ in $Y$. 
The result then follows from the invariance of log Gromov-Witten invariants 
under log birational modification proved by Abramovich and Wise 
\cite{MR3778185}.
\end{proof}

By Lemmas \ref{lem_S_preserving} and 
\ref{lem_T_bijective}, we have an action of $S$ and $T$ on the set $\Gamma$, 
which generates an action of 
$PSL_2(\Z)$ on $\Gamma$.
Given a power series $f$ with coefficients polynomial in 
$t^{\beta}$ for $\beta \in \Gamma$, 
and $M \in PSL_2(\Z)$, we define $M^{*}(f)$ by 
$M^{*}(t^{\beta})\coloneqq t^{M^{*}(\beta)}$
for $\beta \in \Gamma$, and extending by linearity.
For a quantum ray 
$\fd=(p_{\fd}, f_{\fd})$ and 
$M \in SL_2(\Z)$, we define
\[ M(\fd) \coloneqq (M(p_{\fd}), 
M^{*}(f_{\fd}) )\,,\]
where $M(p_{\fd})$ is the image of 
$p_{\fd}$ by the action of $M$ on
$B(\Z)$.

\begin{prop} \label{prop_symmetry}
For every $(m,n) \in B(\Z)$ with $m$ and $n$ coprime, 
and $M \in SL_2(\Z)$, we have the following relation 
between the quantum rays
$\fd_{m,n}$ and $\fd_{M((m,n))}$ 
of the canonical quantum scattering diagram $\fD_{\can}$:
\[ \fd_{M((m,n))}= M(\fd_{m,n}) \,.\]
\end{prop}

\begin{proof}
It is shown in \cite{gross2019cubic} 
that the action of $PSL_2(\Z)$ on $B(\Z)$ is compatible 
with the action of $PSL_2(\Z)$ on curve classes. 
Thus, the result follows from 
Lemma \ref{lem_S_preserving}
and \ref{lem_T_bijective}.
\end{proof}

As $PSL_2(\Z)$ acts transitively on the set of 
$(m,n) \in B(\Z)$ with $m$ and $n$ coprime, 
one can use Proposition 
\ref{prop_symmetry} to compute all the rays 
$\fd_{m,n}$ in terms of the ray 
$\fd_1=\fd_{1,0}$ given by Proposition 
\ref{prop_ray_rho_j_contribution}.

\begin{cor} \label{cor_integrality}
For every $(m,n) \in B(\Z)$ with $m$ and $n$ coprime, we have 
\[ f_{\fd_{m,n}} \in \Z[q^{\pm}][NE(Y)][\![z^{-(m,n)}]\!]\,,\]
where $q=e^{i \hbar}$.
\end{cor}

\begin{proof}
It is a corollary of Proposition 
\ref{prop_symmetry} and of the fact that 
$f_{\fd_1} \in \Z[q^{\pm}][NE(Y)][\![z^{-(m,n)}]\!]$ 
by Equation \eqref{eq:ray_rho_j_contribution}.
\end{proof}

By Corollary \ref{cor_integrality}, 
we can view $\fD_{\can}$ as a quantum scattering diagram over the ring 
$\Z[q^{\pm}][NE(Y)]$ rather than over the ring $\Q[\![\hbar]\!][NE(Y)]$.

\begin{cor} \label{cor_ray_11}
The ray $\fd_{1,1}$ of the canonical quantum scattering diagram 
$\fD_{\can}$ is given by $p_{\fd_{1,1}}=(1,1)=v_1+v_2$
and 
\begin{equation} \label{eq:ray_11}
f_{\fd_{1,1}}
=\frac{\prod_{m=1}^8(1+t^{D_3+L_{3m}}z^{-v_1-v_2})}{(1-q^{-1} t^{D_1+D_2+2D_3}z^{-2v_1-2v_2})
(1- t^{D_1+D_2+2D_3}z^{-2v_1-2v_2})^2
(1-q t^{D_1+D_2+2D_3}z^{-2v_1-2v_2})} \,,
\end{equation}
where $q=e^{i\hbar}$.
\end{cor}

\begin{proof}
We have $(1,1)=T(0,1)$, so 
$\fd_{1,1}=T(\fd_{0,1})=T(\fd_2)$.
Therefore, it is enough to check that 
$T^{*}(L_{2m})=D_3+L_{3m}$ for 
$1 \leq m \leq 8$, 
which is clear from the birational description of $T^{*}$, and 
$T^{*}(D_1+D_3)=D_1+D_2+2D_3$, 
which can be checked using Equation
\eqref{eq:T_action}:
\[ T^{*}(D_1+D_3)=T^{*}(2H-E_{11}-E_{12}
-E_{31}-E_{32})\]
\[=4H-2E_{31}-2E_{32}
-E_{11}-E_{12}-E_{21}-E_{22} =D_1+D_2+2D_3 \,.\]
\end{proof}

\section{Derivation of the equations of the quantum mirror}
\label{section_derivation}
In Section \ref{section_canonical_scattering}, 
we defined the canonical quantum scattering diagram $\fD_{\can}$, 
that we can view as a quantum scattering diagram over 
$\Z[q^{\pm}][NE(Y)]$ by 
Corollary \ref{cor_integrality}.
By Theorem \ref{thm_can_consistent}, $\fD_{\can}$ 
is consistent, and so by Section \ref{section_quantum_broken_lines} 
we have a $\Z[q^{\pm}][NE(Y)]$-algebra 
$\cA_{\fD_{\can}}$, coming with a 
$\Z[q^{\pm}][NE(Y)]$-linear basis 
of quantum theta functions 
$\{ \vartheta_p \}_{p\in B(\Z)}$, whose structure constants 
$C_{p_1,p_2}^{\fD_{\can},p}$
can be computed in terms of quantum broken lines 
by Equation \eqref{eq:structure_constant}.
In this section, we give an explicit presentation of $\cA_{\fD_{\can}}$ 
by generators and relations.
The non-commutative algebra 
$\cA_{\fD_{\can}}$ is a deformation quantization of the mirror family of 
$(Y,D)$ considered in 
\cite{gross2019cubic} and our presentation of  $\cA_{\fD_{\can}}$ 
will be a non-commutative deformation of the presentation 
of the mirror family given in \cite{gross2019cubic}.

\subsection{Statement of the presentation of $\cA_{\fD_\can}$ by generators and relations}

\begin{thm} \label{thm_eq_cubic}
The $\Z[q^{\pm}][NE(Y)]$-algebra
$\cA_{\fD_\can}$ admits the following 
presentation by generators and relation:
$\cA_{\fD_\can}$ is generated by 
$\vartheta_{v_1}$, $\vartheta_{v_2}$, $\vartheta_{v_3}$, 
with the relations
\begin{equation} \label{eq:comm_1}
 q^{-\frac{1}{2}} \vartheta_{v_1}
\vartheta_{v_2} 
-q^{\frac{1}{2}} 
\vartheta_{v_2}\vartheta_{v_1} 
=(q^{-1}-q) t^{D_3}\vartheta_{v_3}
-(q^{\frac{1}{2}}-q^{-\frac{1}{2}}) \left( \sum_{j=1}^8 t^{D_3+L_{3j}} \right) \,,
\end{equation}

\begin{equation} \label{eq:comm_2}
q^{-\frac{1}{2}} \vartheta_{v_2}
\vartheta_{v_3} 
-q^{\frac{1}{2}} 
\vartheta_{v_3} \vartheta_{v_2} 
=(q^{-1}-q) t^{D_1}\vartheta_{v_1}
-(q^{\frac{1}{2}}-q^{-\frac{1}{2}}) \left( \sum_{j=1}^8 t^{D_1+L_{1j}} \right)\,,
\end{equation}

\begin{equation} \label{eq:comm_3}
q^{-\frac{1}{2}} \vartheta_{v_3}
\vartheta_{v_1} 
-q^{\frac{1}{2}} 
\vartheta_{v_1} \vartheta_{v_3} 
=(q^{-1}-q) t^{D_2} \vartheta_{v_2}
-(q^{\frac{1}{2}}-q^{-\frac{1}{2}}) \left( \sum_{j=1}^8 t^{D_2+L_{2j}} \right)\,,
\end{equation}

\begin{equation} \label{eq:cubic}
q^{-\frac{1}{2}}
\vartheta_{v_1} 
\vartheta_{v_2}
\vartheta_{v_3}
=q^{-1} t^{D_1} \vartheta_{v_1}^2 
+q t^{D_2} \vartheta_{v_2}^2 
+q^{-1} 
t^{D_3} \vartheta_{v_3}^2
+
q^{-\frac{1}{2}} 
\left( \sum_{j=1}^8 t^{D_1+L_{1j}}
\right) 
\vartheta_{v_1}  
+ q^{\frac{1}{2}} 
\left( 
\sum_{j=1}^8 t^{D_2+L_{2j}} \right)
\vartheta_{v_2}
\end{equation}
\begin{equation*}
+ q^{-\frac{1}{2}} 
\left( \sum_{j=1}^8 
t^{D_3+L_{3j}}
\right) 
\vartheta_{v_3}
+
\sum_{1 \leq j <j' \leq 8} t^{D_1+L_{1j}+L_{1j'}}
-(q^{\frac{1}{2}}-q^{-\frac{1}{2}})^2 t^{D_1+D_2+D_3}\,.
\end{equation*}
\end{thm}

In the classical limit $q^{\frac{1}{2}} \rightarrow 1$, 
Theorem \ref{thm_eq_cubic}
reduces to the main result of 
\cite{gross2019cubic} (Theorem 0.1) describing the 
result of the mirror construction of \cite{MR3415066} 
applied to $(Y,D)$ as the family of cubic surfaces 
given in terms of the classical theta functions 
$\{ \vartheta_p^{\cl}\}_{p \in B(\Z)}$ by the equation
\begin{equation}
    \vartheta_{v_1}^{\cl} 
\vartheta_{v_2}^{\cl}
\vartheta_{v_3}^{\cl}
= t^{D_1} (\vartheta_{v_1}^{\cl})^2 
+t^{D_2} (\vartheta_{v_2}^{\cl})^2 
+ 
t^{D_3} (\vartheta_{v_3}^{\cl})^2 
+
\left( \sum_{j=1}^8 t^{D_1+L_{1j}}
\right) 
\vartheta_{v_1}^{\cl}  
+  
\left( 
\sum_{j=1}^8 t^{D_2+L_{2j}} \right)
\vartheta_{v_2}^{\cl}  
\end{equation}
\begin{equation*}
+
\left( \sum_{j=1}^8 
t^{D_3+L_{3j}}
\right) 
\vartheta_{v_3}^{\cl}  
+
\sum_{1 \leq j <j' \leq 8} t^{D_1+L_{1j}+L_{1j'}}\,.
\end{equation*}

The proof of Theorem \ref{thm_eq_cubic} takes the remainder of 
Section \ref{section_derivation}. 
In Sections \ref{section_products} and 
\ref{section_triple_product}, 
we check that the relations in Theorem \ref{thm_eq_cubic} 
are indeed satisfied in $\cA_{\fD_\can}$.
We use the description of the product of quantum theta functions 
in terms of broken lines given by Equations \eqref{eq:structure_constant}
and \eqref{eq:product_theta}. Recalling that $q=A^4$, we have
\begin{equation} 
\vartheta_{p_1}\vartheta_{p_2}=\sum_{p\in B(\Z)} 
C_{p_1,p_2}^{\fD_{\can},p} \vartheta_p \,, 
\end{equation}
\begin{equation}
C_{p_1,p_2}^{\fD_\can, p} \coloneqq \sum_{(\gamma_1,\gamma_2)}
c(\gamma_1)c(\gamma_2) q^{\frac{1}{2}\langle 
s(\gamma_1), s(\gamma_2) \rangle}\,,
\end{equation}
where the sum is over pairs $(\gamma_1,\gamma_2)$ of quantum broken lines for
$\fD_{\can}$ with charges $p_1$,$p_2$ and common endpoint $Q$ close to $p$, 
such that  writing $c(\gamma_1)z^{s(\gamma_1)}$ and 
$c(\gamma_2)z^{s(\gamma_2)}$ the final monomials, we have 
$s(\gamma_1)+s(\gamma_2)=p$. 

Gross, Hacking, Keel and Siebert \cite{gross2019cubic} 
have done these computations in the classical limit, 
enumerating the possible configurations of broken lines  
and using the function $F$ reviewed in Section 
\ref{section_quantum_broken_lines} to bound the possibilities. 
The arguments
of \cite{gross2019cubic} leading to the 
enumeration of possible configurations 
of broken lines still hold in the quantum case. 
Therefore, we will simply explain how to 
modify in the quantum case the computations of \cite{gross2019cubic}. 
Finally, we end the proof of Theorem 
\ref{thm_eq_cubic}
in Section \ref{section_end_proof_presentation}.

\subsection{Products and commutators of quantum theta functions}
\label{section_products}

\begin{figure}[h]
\centering
\setlength{\unitlength}{1cm}
\begin{picture}(6,3.5)
\put(3,0.5){\circle*{0.1}}
\put(3,0.5){\line(1,0){4}}
\put(3,0.5){\line(-1,0){4}}
\put(3,0.5){\line(0,1){3}}
\put(3,0.5){\line(-1,1){3}}
\put(3.7,1.25){$z$}
\put(6,1.2){$\gamma_1$}
\put(6,1.6){$\gamma_2$}
\color{red}
\put(4,1.5){\circle*{0.1}}
\put(4,1.5){\line(1,0){2}}
\put(4,1.55){\line(1,0){2}}
\end{picture}
\caption{Coefficient of $\vartheta_{2v_1}$
in $\vartheta_{v_1}^2$}
\label{figure_coeff_2v1_v1_v1}
\end{figure}
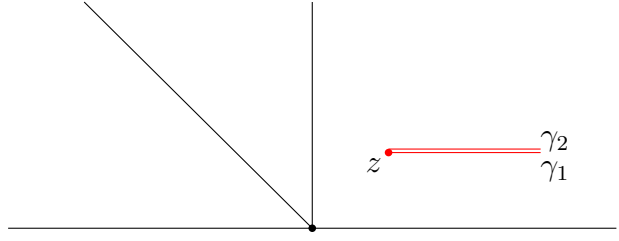

\begin{figure}[h]
\centering
\setlength{\unitlength}{1cm}
\begin{picture}(6,3.5)
\put(3,0.5){\circle*{0.1}}
\put(3,0.5){\line(1,0){4}}
\put(3,0.5){\line(-1,0){4}}
\put(3,0.5){\line(0,1){3}}
\put(3,0.5){\line(-1,1){3}}
\put(3.7,1.25){$z$}
\put(5,1.25){$\gamma_1$}
\put(0,1.25){$\gamma_2$}
\color{red}
\put(4,1.5){\circle*{0.1}}
\put(4,1.5){\line(1,0){2}}
\put(4,1.5){\line(-1,0){5}}
\end{picture}
\caption{Contribution to the coefficient of $\vartheta_0=1$
in $\vartheta_{v_1}^2$}
\label{figure_coeff_1_v1_v1}
\end{figure}

\begin{figure}[h]
\centering
\setlength{\unitlength}{1cm}
\begin{picture}(6,3.5)
\put(3,0.5){\circle*{0.1}}
\put(3,0.5){\line(1,0){4}}
\put(3,0.5){\line(-1,0){4}}
\put(3,0.5){\line(0,1){3}}
\put(3,0.5){\line(-1,1){3}}
\put(3.7,1.25){$z$}
\put(5,1.25){$\gamma_2$}
\put(0,1.25){$\gamma_1$}
\color{red}
\put(4,1.5){\circle*{0.1}}
\put(4,1.5){\line(1,0){2}}
\put(4,1.5){\line(-1,0){5}}
\end{picture}
\caption{Contribution to the coefficient of $\vartheta_0=1$
in $\vartheta_{v_1}^2$}
\label{figure_coeff_1_v1_v1_switch}
\end{figure}

\begin{lem} \label{lem_square_theta}
For every $j,k,\ell$ 
such that $\{j,k,\ell\}
=\{1,2,3\}$, we have 
\begin{equation}
\vartheta_{v_j}^2
=\vartheta_{2v_j}
+2t^{D_k+D_\ell} \,. 
\end{equation}
\end{lem}

\begin{proof}
By the cyclic $\Z/3\Z$-symmetry permuting $\{1,2,3\}$, 
it is enough to treat the case 
$j=1$.
According to the proof of \cite[Lemma 3.6]{gross2019cubic}, 
the only 
configurations of broken lines contributing 
to the product $\vartheta_{v_1}^2$
are given by Figure
\ref{figure_coeff_2v1_v1_v1}, Figure 
\ref{figure_coeff_1_v1_v1}, and Figure 
\ref{figure_coeff_1_v1_v1_switch}
(see \cite[Figure 3.2]{gross2019cubic}).

Figure
\ref{figure_coeff_2v1_v1_v1} gives a term 
$\vartheta_{2v_1}$ in $\vartheta_{v_1}^2$:
we have \[c(\gamma_1)=1\,,\,\, c(\gamma_2)=1\,,\,\,
s(\gamma_1)=(1,0)\,,\,\,
s(\gamma_2)=(1,0)\,,\,\,
 \langle s(\gamma_1),s(\gamma_2) \rangle 
=\langle (1,0), (1,0) \rangle=0\,,\] 
and so $q^{\frac{1}{2} \langle s(\gamma_1),s(\gamma_2) \rangle}=1$.

Figure 
\ref{figure_coeff_1_v1_v1} gives a term $t^{D_2+D_3}$
in $\vartheta_{v_1}^2$: we have 
$c(\gamma_1)=1$,  
$c(\gamma_2)=t^{D_2+D_3}$ because $\gamma_2$
crosses $\R_{\geq 0} v_3$ and $\R_{\geq 0} v_2$ 
without bending, 
\[s(\gamma_1)=(1,0)\,,\,\, s(\gamma_2)=(-1,0)\,,\,\, 
\langle s(\gamma_1),s(\gamma_2) \rangle = \langle (1,0), (-1,0) \rangle =0\,,\] and so 
$q^{\frac{1}{2}\langle s(\gamma_1),s(\gamma_2) \rangle}=1$.
Similarly, Figure \ref{figure_coeff_1_v1_v1_switch} 
gives a term $t^{D_2+D_3}$ in $\vartheta_{v_1}^2$.
\end{proof}

\begin{figure}[h]
\centering
\setlength{\unitlength}{1cm}
\begin{picture}(6,3.5)
\put(3,0.5){\circle*{0.1}}
\put(3,0.5){\line(1,0){4}}
\put(3,0.5){\line(-1,0){4}}
\put(3,0.5){\line(0,1){3}}
\put(3,0.5){\line(-1,1){3}}
\put(3.7,1.25){$z$}
\put(6,1.2){$\gamma_1$}
\put(6,1.6){$\gamma_2$}
\color{red}
\put(4,1.5){\circle*{0.1}}
\put(4,1.5){\line(1,0){2}}
\put(4,1.55){\line(1,0){2}}
\end{picture}
\caption{Coefficient of $\vartheta_{(n+1)v_1}$
in $\vartheta_{v_1}\vartheta_{nv_1}$}
\label{figure_coeff_nv1_v1_v1}
\end{figure}

\begin{figure}[h]
\centering
\setlength{\unitlength}{1cm}
\begin{picture}(6,3.5)
\put(3,0.5){\circle*{0.1}}
\put(3,0.5){\line(1,0){4}}
\put(3,0.5){\line(-1,0){4}}
\put(3,0.5){\line(0,1){3}}
\put(3,0.5){\line(-1,1){3}}
\put(3.7,1.25){$z$}
\put(5,1.25){$\gamma_1$}
\put(0,1.25){$\gamma_2$}
\color{red}
\put(4,1.5){\circle*{0.1}}
\put(4,1.5){\line(1,0){2}}
\put(4,1.5){\line(-1,0){5}}
\end{picture}
\caption{Coefficient of $\vartheta_{(n-1)v_1}$
in $\vartheta_{v_1} \vartheta_{nv_1}$}
\label{figure_coeff_n_v1_v1}
\end{figure}

The following Lemma \ref{lem_power_theta} is not part of 
the proof of Theorem \ref{thm_eq_cubic}. 
It will be used in 
the proof of Theorem \ref{thm:basis_comparison} 
showing that the bracelets basis 
and the quantum theta functions basis agree.

\begin{lem} \label{lem_power_theta}
For every $j,k,\ell$ 
such that $\{j,k,\ell\}
=\{1,2,3\}$, and for every integer $n \geq 1$, we have 
\begin{equation} 
\vartheta_{v_j} \vartheta_{nv_j}
=\vartheta_{(n+1)v_j}+t^{D_k+D_\ell} \vartheta_{(n-1)v_j} \,.
\end{equation}
\end{lem}

\begin{proof}
By the cyclic $\Z/3\Z$-symmetry permuting 
$\{1,2,3\}$, it is enough to treat the case 
$j=1$. We claim that the only configurations 
of broken lines contributing to the product 
$ \vartheta_{v_j} \vartheta_{nv_j}$ are given by Figure 
\ref{figure_coeff_nv1_v1_v1}
and Figure \ref{lem_power_theta}.
As this case is not treated in \cite{gross2019cubic}, 
we give an argument.
The contributing broken lines 
$\gamma_1$ and $\gamma_2$ are horizontal for $t <<0$. 
Assume that one of them bends somewhere and look at the first bending. 
The bending occurs along a quantum ray contained in the strict upper half-plane, 
so the direction of the broken after bending has a positive vertical component. 
Iterating this argument, 
we see that the broken line always remains in the 
strict upper half-plane and its final direction has a positive vertical direction. 
Therefore, if either $\gamma_1$ or $\gamma_2$ bends somewhere, then 
$s(\gamma_1)+s(\gamma_2)$ has a negative vertical component, and so 
$s(\gamma_1)+s(\gamma_2)$ cannot be equal to an element $p \in B(\Z)$ 
and so cannot contribute a term 
in the product $ \vartheta_{v_j} \vartheta_{nv_j}$.
Therefore, $\gamma_1$ and $\gamma_2$ never bend and so Figure 
\ref{figure_coeff_nv1_v1_v1}
and Figure \ref{lem_power_theta}
are the only possibilities.

Figure 
\ref{figure_coeff_nv1_v1_v1} gives a term 
$\vartheta_{v_1}\vartheta_{nv_1}$ in 
$\vartheta_{v_1}\vartheta_{nv_1}$:
we have 
\[ c(\gamma_1)=1\,,\,\,c(\gamma_2)=1\,,\,\,
s(\gamma_1)=(1,0)\,\,\,s(\gamma_2)=(n,0)\,,\,\,\langle s(\gamma_1),s(\gamma_2) \rangle =
\langle (1,0), (n,0) \rangle =0\,\]
and so
$q^{\frac{1}{2}\langle s(\gamma_1),s(\gamma_2) \rangle}=1$.

Figure \ref{figure_coeff_nv1_v1_v1} gives a 
term
$t^{D_2+D_3}\vartheta_{(n-1)v_1}$ in 
$\vartheta_{v_1}\vartheta_{nv_1}$:
we have $c(\gamma_1)=1$, 
$c(\gamma_2)=t^{D_2+D_3}$ because 
$\gamma_2$ crosses $\rho_3$ and $\rho_2$ 
without bending, 
\[ s(\gamma_1)=(1,0)\,,\,\,s(\gamma_2)=(-n,0)\,,\,\, \langle s(\gamma_1),s(\gamma_2) \rangle =
\langle (1,0), (-n,0) \rangle =0\,\]
and so
$q^{\frac{1}{2}\langle s(\gamma_1),s(\gamma_2) \rangle}=1$.
\end{proof}

\begin{figure}[h]
\centering
\setlength{\unitlength}{1cm}
\begin{picture}(6,3.5)
\put(3,0.5){\circle*{0.1}}
\put(3,0.5){\line(1,0){4}}
\put(3,0.5){\line(-1,0){4}}
\put(3,0.5){\line(0,1){3}}
\put(3,0.5){\line(-1,1){3}}
\put(3.7,1.25){$z$}
\put(6,1.25){$\gamma_1$}
\put(3.5,3.25){$\gamma_2$}
\color{red}
\put(4,1.5){\circle*{0.1}}
\put(4,1.5){\line(1,0){2}}
\put(4,1.5){\line(0,1){2}}
\end{picture}
\caption{Coefficient of $\vartheta_{v_1+v_2}$
in $\vartheta_{v_1}\vartheta_{v_2}$}
\label{figure_coeff_v1v2_v1v2}
\end{figure}

\begin{figure}[h]
\centering
\setlength{\unitlength}{1cm}
\begin{picture}(6,3.5)
\put(3,0.5){\circle*{0.1}}
\put(3,0.5){\line(1,0){4}}
\put(3,0.5){\line(-1,0){4}}
\put(3,0.5){\line(0,1){3}}
\put(3,0.5){\line(-1,1){3}}
\put(2.6,1.25){$z$}
\put(0,1.25){$\gamma_1$}
\put(2.55,3.25){$\gamma_2$}
\color{red}
\put(2.5,1.5){\circle*{0.1}}
\put(2.5,1.5){\line(-1,0){3}}
\put(2.5,1.5){\line(0,1){2}}
\end{picture}
\caption{Coefficient of $\vartheta_{v_3}$
in $\vartheta_{v_1}\vartheta_{v_2}$}
\label{figure_coeff_v3_v1v2}
\end{figure}

\begin{figure}[h]
\centering
\setlength{\unitlength}{1cm}
\begin{picture}(6,3.5)
\put(3,0.5){\circle*{0.1}}
\put(3,0.5){\line(1,0){4}}
\put(3,0.5){\line(-1,0){4}}
\put(3,0.5){\line(0,1){3}}
\put(3,0.5){\line(-1,1){3}}
\put(3.8,0.7){$z$}
\put(6,0.7){$\gamma_1$}
\put(3,3.25){$\gamma_2$}
\color{blue}
\put(3,0.5){\line(1,1){3}}
\color{red}
\put(4,1){\circle*{0.1}}
\put(4,1){\line(1,0){3}}
\put(4,1){\line(-1,0){0.5}}
\put(3.5,1){\line(0,1){2.5}}
\end{picture}
\caption{Coefficient of $\vartheta_{0}=1$
in $\vartheta_{v_1}\vartheta_{v_2}$}
\label{figure_coeff_1_v1v2}
\end{figure}

\begin{lem} \label{lem_product_theta_1}
We have 
\begin{equation} \vartheta_{v_1} \vartheta_{v_2} 
=q^{\frac{1}{2}} \vartheta_{v_1+v_2} 
+q^{-\frac{1}{2}} t^{D_3} \vartheta_{v_3}
+ \sum_{j=1}^8 t^{D_3+L_{3j}} \,,
\end{equation}
and 
\begin{equation} \vartheta_{v_2}\vartheta_{v_1} 
=q^{-\frac{1}{2}}\vartheta_{v_1+v_2} 
+q^{\frac{1}{2}} t^{D_3} \vartheta_{v_3}
+ \sum_{j=1}^8 t^{D_3+L_{3j}} \,.
\end{equation}
\end{lem}

\begin{proof}
According to the proof of \cite[Lemma 3.6]{gross2019cubic}, the only 
configurations of broken lines contributing to the product 
$\vartheta_{v_1}\vartheta_{v_2}$
are given by Figure
\ref{figure_coeff_v1v2_v1v2}, Figure 
\ref{figure_coeff_v3_v1v2}, and Figure 
\ref{figure_coeff_1_v1v2}
(see \cite[Figures 3.3-3.4]{gross2019cubic}).

Figure
\ref{figure_coeff_v1v2_v1v2} gives a term 
$q^{\frac{1}{2}} 
\vartheta_{v_1+v_2}$ in 
$\vartheta_{v_1} \vartheta_{v_2}$:
we have 
\[ c(\gamma_1)=1\,,\,\,c(\gamma_2)=1
\,,\
s(\gamma_1)=(1,0)\,,\,\,s(\gamma_2)=(0,1)\,,\ 
\langle s(\gamma_1),s(\gamma_2) \rangle = \langle (1,0),(0,1) \rangle =1\,,\]
and so 
$q^{\frac{1}{2} \langle s(\gamma_1),s(\gamma_2) \rangle }=q^{\frac{1}{2}}$.

Figure \ref{figure_coeff_v3_v1v2} 
gives a term $q^{-\frac{1}{2}} t^{D_3} \vartheta_3$
in $\vartheta_{v_1}\vartheta_{v_2}$:
we have $c(\gamma_1)=t^{D_3}$ 
because $\gamma_1$ crosses $\R_{\geq 0} v_3$ 
without bending, $c(\gamma_2)=1$, 
\[ s(\gamma_1)=(-1,0)\,,\,\, s(\gamma_2)=(0,1)\,,\
\langle s(\gamma_1),s(\gamma_2) \rangle = \langle (-1,0),(0,1) \rangle =-1\,,\] 
and so 
$q^{\frac{1}{2} \langle s(\gamma_1),s(\gamma_2) \rangle }=q^{-\frac{1}{2}}$.

Figure \ref{figure_coeff_1_v1v2} gives a term 
$\sum_{j=1}^8 t^{D_3+L_{3j}}$ in
$\vartheta_{v_1}\vartheta_{v_2}$:
we have 
$c(\gamma_1)=1$, 
$c(\gamma_2)=\sum_{j=1}^8 t^{D_3+L_{3j}}$
because $\gamma_2$ crosses $\R_{\geq 0}(v_1+v_2)$ 
with bending and contribution of the term proportional 
to $z^{-v_1-v_2}$ in Equation
\eqref{eq:ray_11}, 
\[s(\gamma_1)=(1,0)\,,\,\, s(\gamma_2)=(-1,0)\,,\ 
\langle s(\gamma_1),s(\gamma_2) \rangle = \langle (1,0),(-1,0) \rangle =0\,,\] 
and so $q^{\frac{1}{2} \langle s(\gamma_1),s(\gamma_2) \rangle }=1$.
We similarly compute $\vartheta_{v_2} \vartheta_{v_1}$:
as $\langle -,- \rangle$ is skew-symmetric, only the powers of $q$ change of sign.
\end{proof}

\begin{lem} \label{lem_commutators_theta}
We have 
\begin{equation} q^{-\frac{1}{2}}\vartheta_{v_1} 
\vartheta_{v_2}
-q^{\frac{1}{2}}\vartheta_{v_2} 
\vartheta_{v_1}
=(q^{-1}-q) t^{D_3} \vartheta_{v_3} 
-(q^{\frac{1}{2}} -q^{-\frac{1}{2}})
\left( \sum_{j=1}^8 t^{D_3+L_{3j}} \right) \,,
\end{equation}
\begin{equation} q^{-\frac{1}{2}} \vartheta_{v_2}
\vartheta_{v_3} 
-q^{\frac{1}{2}}\vartheta_{v_3}
\vartheta_{v_2}
=(q^{-1}-q) t^{D_1}\vartheta_{v_1} 
-(q^{\frac{1}{2}} -q^{-\frac{1}{2}})
\left( \sum_{j=1}^8 t^{D_1+L_{1j}} \right) \,,
\end{equation}
\begin{equation} q^{-\frac{1}{2}} \vartheta_{v_3}
 \vartheta_{v_1}
-q^{\frac{1}{2}} \vartheta_{v_1} 
\vartheta_{v_3}
=(q^{-1}-q) t^{D_2}\vartheta_{v_2} 
-(q^{\frac{1}{2}} -q^{-\frac{1}{2}})
\left( \sum_{j=1}^8 t^{D_2+L_{2j}} \right) \,.
\end{equation}
\end{lem}

\begin{proof}
By the cyclic $\Z /3$-symmetry permuting 
$\{1,2,3\}$, it is enough to compute 
\[ q^{-\frac{1}{2}} \vartheta_{v_1} 
\vartheta_{v_2}
-q^{\frac{1}{2}} \vartheta_{v_2} 
\vartheta_{v_1}\,.\] 
The result follows immediately from 
Lemma \ref{lem_product_theta_1}.
\end{proof}

\begin{figure}[h]
\centering
\setlength{\unitlength}{1cm}
\begin{picture}(6,3.5)
\put(3,0.5){\circle*{0.1}}
\put(3,0.5){\line(1,0){4}}
\put(3,0.5){\line(-1,0){4}}
\put(3,0.5){\line(0,1){3}}
\put(3,0.5){\line(-1,1){3}}
\put(3,0.5){\line(1,-1){1}}
\put(3.25, 1){$z$}
\put(4.75,3){$\gamma_1$}
\put(4,0){$\gamma_2$}
\color{red}
\put(3.5,1.25){\circle*{0.1}}
\put(3.5,1.25){\line(1,1){2}}
\put(3.5,1.25){\line(1,-1){1.5}}
\end{picture}
\caption{Coefficient of $\vartheta_{2v_1}$
in $\vartheta_{v_1+v_2}
\vartheta_{v_3}$}
\label{figure_coeff_2v1_v1v2v3}
\end{figure}

\begin{figure}[h]
\centering
\setlength{\unitlength}{1cm}
\begin{picture}(6,3.5)
\put(3,0.5){\circle*{0.1}}
\put(3,0.5){\line(1,0){4}}
\put(3,0.5){\line(-1,0){4}}
\put(3,0.5){\line(0,1){3}}
\put(3,0.5){\line(-1,1){3}}
\put(3.5,1){$z$}
\put(4.75,3){$\gamma_1$}
\put(2,3){$\gamma_2$}
\color{red}
\put(3.5,1.25){\circle*{0.1}}
\put(3.5,1.25){\line(1,1){2}}
\put(3.5,1.25){\line(-1,1){2}}
\end{picture}
\caption{Coefficient of 
$\vartheta_{2v_2}$
in $\vartheta_{v_1+v_2}
\vartheta_{v_3}$}
\label{figure_coeff_2v2_v1v2v3}
\end{figure}

\begin{figure}[h]
\centering
\setlength{\unitlength}{1cm}
\begin{picture}(6,3.5)
\put(3,0.5){\circle*{0.1}}
\put(3,0.5){\line(1,0){4}}
\put(3,0.5){\line(-1,0){4}}
\put(3,0.5){\line(0,1){3}}
\put(3,0.5){\line(-1,1){3}}
\put(3,0.5){\line(1,-1){1}}
\put(3.7, 1.2){$z$}
\put(5,3){$\gamma_1$}
\put(5,-0.4){$\gamma_2$}
\color{red}
\put(4,1.5){\circle*{0.1}}
\put(4,1.5){\line(1,1){2}}
\put(4,0.5){\line(1,-1){1}}
\put(4,0.5){\line(0,1){1}}
\end{picture}
\caption{Coefficient of 
$\vartheta_{v_1}$
in $\vartheta_{v_1+v_2}
\vartheta_{v_3}$}
\label{figure_coeff_v1_v1v2v3}
\end{figure}

\begin{figure}[h]
\centering
\setlength{\unitlength}{1cm}
\begin{picture}(6,3.5)
\put(3,0.5){\circle*{0.1}}
\put(3,0.5){\line(1,0){4}}
\put(3,0.5){\line(-1,0){4}}
\put(3,0.5){\line(0,1){3}}
\put(3,0.5){\line(-1,1){3}}
\put(4,1.4){$z$}
\put(4.75,3){$\gamma_1$}
\put(2,3){$\gamma_2$}
\color{red}
\put(4,1.75){\circle*{0.1}}
\put(4,1.75){\line(1,1){2}}
\put(3,1.75){\line(-1,1){2}}
\put(3,1.75){\line(1,0){1}}
\end{picture}
\caption{Coefficient of $\vartheta_{v_2}$
in $\vartheta_{v_1+v_2}
\vartheta_{v_3}$}
\label{figure_coeff_v2_v1v2v3}
\end{figure}

\begin{figure}[h]
\centering
\setlength{\unitlength}{1cm}
\begin{picture}(6,3.5)
\put(3,0.5){\circle*{0.1}}
\put(3,0.5){\line(1,0){4}}
\put(3,0.5){\line(-1,0){4}}
\put(3,0.5){\line(0,1){3}}
\put(3,0.5){\line(-1,1){3}}
\put(3,0.5){\line(1,-1){1}}
\put(4.2, 1.2){$z$}
\put(5.5,2.5){$\gamma_1$}
\put(5,-0.4){$\gamma_2$}
\color{red}
\put(4.5,1){\circle*{0.1}}
\put(4,0.5){\line(1,1){2}}
\put(4,0.5){\line(1,-1){1}}
\end{picture}
\caption{Coefficient of 
$\vartheta_0=1$
in $\vartheta_{v_1+v_2}
\vartheta_{v_3}$}
\label{figure_coeff_1_v1v2v3}
\end{figure}

\begin{lem} \label{lem_product_theta_2}
We have 
\begin{equation} \vartheta_{v_1+v_2} 
\vartheta_{v_3} = q^{-1} t^{D_1} \vartheta_{2v_1}
+q t^{D_2} \vartheta_{2v_2} 
+q^{-\frac{1}{2}} 
\left( \sum_{j=1}^8 t^{D_1+L_{1j}}
\right) \vartheta_{v_1}  
+ q^{\frac{1}{2}} 
\left( 
\sum_{j=1}^8 t^{D_2+L_{2j}} \right)
\vartheta_{v_2}
\end{equation}
\[
+
\sum_{1 \leq j <j' \leq 8} t^{D_1+L_{1j}+L_{1j'}}
+(q^{\frac{1}{2}}
+q^{-\frac{1}{2}})^2 t^{D_1+D_2+D_3} \,.\]
\end{lem}

\begin{proof}
According to the proof of \cite[Lemma 3.6]{gross2019cubic}, the only 
configurations of broken lines contributing to the product 
$\vartheta_{v_1+v_2}\vartheta_{v_3}$
are given by Figure
\ref{figure_coeff_2v1_v1v2v3}, Figure 
\ref{figure_coeff_2v2_v1v2v3}, Figure 
\ref{figure_coeff_v1_v1v2v3}, Figure \ref{figure_coeff_v2_v1v2v3},
and Figure \ref{figure_coeff_1_v1v2v3}
(see \cite[Figures 3.5-3.6]{gross2019cubic}).

Figure \ref{figure_coeff_2v1_v1v2v3} gives a term $q^{-1}t^{D_1} 
\vartheta_{2v_1}$ in 
$\vartheta_{v_1+v_2}\vartheta_{v_3}$:
we have 
$c(\gamma_1)=1$, $c(\gamma_2)=t^{D_1}$ 
because $\gamma_2$
crosses $\R_{\geq 0} v_1$ without bending, 
\[s(\gamma_1)=(1,1)\,, 
s(\gamma_2)=(-1,1)\,,\,\,
\langle s(\gamma_1),s(\gamma_2) \rangle 
=\langle (1,1),(1,-1) \rangle =-2\,,\] 
and so 
$q^{\frac{1}{2} \langle s(\gamma_1),s(\gamma_2) \rangle}
=q^{-1}$.

Figure \ref{figure_coeff_2v2_v1v2v3} gives a term 
$q t^{D_2} \vartheta_{2v_2}$ in 
$\vartheta_{v_1+v_2}\vartheta_{v_3}$:
we have 
$c(\gamma_1)=1$, $ c(\gamma_2)=t^{D_2}$
as 
$\gamma_2$ crosses $\R_{\geq 0}v_2$ without bending, 
\[s(\gamma_1)=(1,1)\,,\,\, s(\gamma_2)=(-1,1)\,,\,\,
\langle s(\gamma_1),s(\gamma_2) \rangle 
=\langle (1,1),(-1,1) \rangle = 2\,,\] and so 
$q^{\frac{1}{2} \langle s(\gamma_1),s(\gamma_2) \rangle}=q$.

Figure \ref{figure_coeff_v1_v1v2v3} gives a term 
$q^{-\frac{1}{2}} \left( \sum_{j=1}^8 t^{D_1+L_{1j} }\right) \vartheta_{v_1}$ 
in $\vartheta_{v_1+v_2}\vartheta_{v_3}$:
we have 
$c(\gamma_1)=1$, $ c(\gamma_2)=\sum_{j=1}^8 t^{D_1+L_{1j}}$
because 
$\gamma_2$ crosses $\R_{\geq 0} v_1$ with bending 
and contribution of the term proportional to $z^{-v_1}$ in Equation
\eqref{eq:ray_rho_j_contribution}, 
\[s(\gamma_1)=(1,1)\,,\,\, s(\gamma_2)=(0,1)\,,\,\,
\langle s(\gamma_1),s(\gamma_2) \rangle = \langle (1,1),(0,-1) \rangle =-1\,,\] 
and so 
$q^{\frac{1}{2} \langle s(\gamma_1), s(\gamma_2) \rangle} =q^{-\frac{1}{2}}$.

Figure \ref{figure_coeff_v2_v1v2v3} gives a term 
$q^{\frac{1}{2}} \left( \sum_{j=1}^8 t^{D_2+L_{2j}} \right)\vartheta_{v_2}$ 
in $\vartheta_{v_1+v_2}\vartheta_{v_3}$:
we have 
$c(\gamma_1)=1$, 
$c(\gamma_2)=\sum_{j=1}^8 t^{D_2+L_{2j}}$
because $\gamma_2$ crosses $\R_{\geq 0}v_2$ with bending 
and contribution of the term proportional to $z^{-v_2}$ in Equation
\eqref{eq:ray_rho_j_contribution},
\[s(\gamma_1)=(1,1)\,,\,\, s(\gamma_2)=(-1,0)\,,\,\,
\langle s(\gamma_1), s(\gamma_2) \rangle 
=\langle (1,1), (-1,0) \rangle =1\,,\] and so
$q^{\frac{1}{2} \langle s(\gamma_1),s(\gamma_2) \rangle}
=q^{\frac{1}{2}}$.

Figure \ref{figure_coeff_1_v1v2v3} gives terms 
\[ \sum_{1 \leq j <j' \leq 8} t^{D_1+L_{1j}+L_{1j'}}
+(q^{\frac{1}{2}}
+q^{-\frac{1}{2}})^2 t^{D_1+D_2+D_3} \,\]
in $\vartheta_{v_1+v_2} \vartheta_{v_3}$.
Indeed, we have 
\[c(\gamma_1)=1\,,\,\, c(\gamma_2)=\sum_{1 \leq j <j' \leq 8} t^{D_1+L_{1j}+L_{1j'}}
+(q^{\frac{1}{2}}
+q^{-\frac{1}{2}})^2 t^{D_1+D_2+D_3} \] 
because 
$\gamma_2$ crosses $\R_{\geq 0} v_1$ with bending and contribution of the term 
proportional to $z^{-2v_1}$ in 
Equation \eqref{eq:ray_rho_j_contribution}, 
\[s(\gamma_1)=(1,1)\,,\,\, s(\gamma_2)=(-1,-1)\,,\,\,
\langle s(\gamma_1),s(\gamma_2) \rangle 
=\langle (1,1), (-1,-1) \rangle =0\,,\] 
and so 
$q^{\frac{1}{2} \langle s(\gamma_1),s(\gamma_2) \rangle }=1$.
\end{proof}

\subsection{Triple product of quantum theta functions}
\label{section_triple_product}

\begin{lem} \label{lem_cubic_equation_theta}
We have 
\begin{equation}
q^{-\frac{1}{2}}
\vartheta_{v_1} 
\vartheta_{v_2}
\vartheta_{v_3}
=q^{-1} t^{D_1}\vartheta_{v_1}^2 
+q t^{D_2} \vartheta_{v_2}^2 
+q^{-1} 
t^{D_3}\vartheta_{v_3}^2
+
q^{-\frac{1}{2}} 
\left( \sum_{j=1}^8 t^{D_1+L_{1j}}
\right) \vartheta_{v_1}  
+ q^{\frac{1}{2}} 
\left( 
\sum_{j=1}^8 t^{D_2+L_{2j}} \right)
\vartheta_{v_2}
\end{equation}
\[ 
+ q^{-\frac{1}{2}} 
\left( \sum_{j=1}^8 
z^{D_3+L_{3j}}
\right) 
\vartheta_{v_3}
+
\sum_{1 \leq j <j' \leq 8} t^{D_1+L_{1j}+L_{1j'}}
-(q^{\frac{1}{2}}-q^{-\frac{1}{2}})^2 t^{D_1+D_2+D_3}
\,.\]
\end{lem}

\begin{proof}
By Lemma \ref{lem_product_theta_1}, we have 
\[ \vartheta_{v_1} 
\vartheta_{v_2}
\vartheta_{v_3}
=\left( 
q^{\frac{1}{2}} 
\vartheta_{v_1+v_2} +q^{-\frac{1}{2}} t^{D_3} \vartheta_{v_3} 
+
\sum_{j=1}^8 t^{D_3+L_{3j}} \right)
\vartheta_{v_3}
=q^{\frac{1}{2}} 
\vartheta_{v_1+v_2}\vartheta_{v_3} +q^{-\frac{1}{2}} t^{D_3}\vartheta_{v_3}^2 +
\left(
\sum_{j=1}^8 t^{D_3+L_{3j}} \right)
\vartheta_{v_3}\,,\]
and  so 
\[ q^{-\frac{1}{2}} 
\vartheta_{v_1}
\vartheta
_{v_2}
\vartheta_{v_3} 
=\vartheta_{v_1+v_2} 
\vartheta_{v_3} 
+q^{-1}t^{D_3} \vartheta_{v_3}^2 
+ q^{-\frac{1}{2}} 
\left( \sum_{j=1}^8 
t^{D_3+L_{3j}}
\right) 
\vartheta_{v_3} \,.\]
Using Lemma \ref{lem_product_theta_2}, we obtain
\[ 
q^{-\frac{1}{2}}
\vartheta_{v_1} 
\vartheta_{v_2}
\vartheta_{v_3}
=q^{-1} t^{D_1} \vartheta_{2v_1} 
+q t^{D_2} \vartheta_{2v_2} 
+q^{-1} 
t^{D_3} \vartheta_{v_3}^2
+
q^{-\frac{1}{2}} 
\left( \sum_{j=1}^8 z^{D_1+L_{1j}}
\right) \vartheta_{v_1}  
+ q^{\frac{1}{2}} 
\left( 
\sum_{j=1}^8 t^{D_2+L_{2j}} \right)
\vartheta_{v_2}
\]
\[ 
+ q^{-\frac{1}{2}} 
\left( \sum_{j=1}^8 
t^{D_3+L_{3j}}
\right) 
\vartheta_{v_3}
+
\sum_{1 \leq j <j' \leq 8} t^{D_1+L_{1j}+L_{1j'}}
+(q^{\frac{1}{2}}
+q^{-\frac{1}{2}})^2 t^{D_1+D_2+D_3}\,.\]
By Lemma \ref{lem_square_theta}, 
we have 
\[ q^{-1}t^{D_1} \vartheta_{2v_1} 
=
q^{-1} t^{D_1} \vartheta_{v_1}^2
-2 q^{-1} t^{D_1+D_2+D_3}\] 
and
\[q t^{D_2} \vartheta_{2v_2} 
=q t^{D_2} \vartheta_{v_2}^2
-2 q t^{D_1+D_2+D_3}\,.\]
Using that
\[ (q^{\frac{1}{2}}+q^{-\frac{1}{2}})^2
-2q-2q^{-1}=-(q^{\frac{1}{2}}-q^{-\frac{1}{2}})^2 \,,\]
we finally obtain Lemma \ref{lem_cubic_equation_theta}.

\end{proof}

\subsection{End of the proof of the presentation of $\cA_{\fD_\can}$}
\label{section_end_proof_presentation}

In this section, we end the proof of Theorem \ref{thm_eq_cubic}.

Recall that we defined in 
Section \ref{section_quantum_broken_lines} the monomials 
$m[p] \in \cA_{\fD_\can}$ as 
follows:
if $p=av_j+bv_{j+1}$ with 
$a \geq 0$ and $b \geq 0$, then 
$m[p]\coloneqq \vartheta_{v_j}^a 
\vartheta_{v_{j+1}}^b$.
We proved in Lemma \ref{lem_m_linear_basis} 
that $\{m[p]\}_{p \in B(\Z)}$ is a 
$\Z[q^{\pm}][NE(Y)]$-linear basis of $\cA_{\fD_\can}$.

Let $\cB$ be the $\Z[q^{\pm}][NE(Y)]$-algebra with generators 
$\vartheta_1$, $\vartheta_2$,
$\vartheta_3$
and relations 
\begin{equation} \label{eq_commutator_1}
q^{-\frac{1}{2}} \vartheta_1
\vartheta_2 
-q^{\frac{1}{2}} 
\vartheta_2 \vartheta_1 
=(q^{-1}-q) t^{D_3}\vartheta_3
-(q^{\frac{1}{2}}-q^{-\frac{1}{2}}) \left( \sum_{j=1}^8 t^{D_3+L_{3j}} \right) \,,
\end{equation}

\begin{equation} \label{eq_commutator_2}
q^{-\frac{1}{2}}\vartheta_2
\vartheta_3 
-q^{\frac{1}{2}} 
\vartheta_3 \vartheta_2 
=(q^{-1}-q) t^{D_1}\vartheta_1
-(q^{\frac{1}{2}}-q^{-\frac{1}{2}}) \left( \sum_{j=1}^8 t^{D_1+L_{1j}} \right)\,,
\end{equation}

\begin{equation} \label{eq_commutator_3}
q^{-\frac{1}{2}} \vartheta_3
\vartheta_1 
-q^{\frac{1}{2}} 
\vartheta_1 \vartheta_3 
=(q^{-1}-q) t^{D_2} \vartheta_2
-(q^{\frac{1}{2}}-q^{-\frac{1}{2}}) \left( \sum_{j=1}^8 t^{D_2+L_{2j}} \right)\,,
\end{equation}

\begin{multline}  \label{eq_cubic}
q^{-\frac{1}{2}}
\vartheta_1 
\vartheta_2
\vartheta_3
=q^{-1} t^{D_1} \vartheta_1^2 
+q t^{D_2} \vartheta_2^2 
+q^{-1} 
t^{D_3} \vartheta_3^2
+
q^{-\frac{1}{2}} 
\left( \sum_{j=1}^8 t^{D_1+L_{1j}}
\right) 
\vartheta_1  
+ q^{\frac{1}{2}} 
\left( 
\sum_{j=1}^8 t^{D_2+L_{2j}} \right)
\vartheta_2
\\
+ q^{-\frac{1}{2}} 
\left( \sum_{j=1}^8 
t^{D_3+L_{3j}}
\right) 
\vartheta_3
+
\sum_{1 \leq j <j' \leq 8} t^{D_1+L_{1j}+L_{1j'}}
-(q^{\frac{1}{2}}-q^{-\frac{1}{2}})^2 t^{D_1+D_2+D_3}\,.
\end{multline}

For every $p \in B(\Z)$, we define 
$n[p] \in \cB$ as the following monomials in $\vartheta_1$,
$\vartheta_2$, $\vartheta_3$:
if $p=av_j+bv_{j+1}$ with 
$a \geq 0$ and $b \geq 0$, then 
$n[p]\coloneqq \vartheta_j^a 
\vartheta_{j+1}^b$.

\begin{lem} \label{lem_tilde_m_linear_generators}
The monomials $n[p]$ for $p \in B(\Z)$ 
form a $\Z[q^{\pm}][NE(Y)]$-linear generating set of 
$\cB$.
\end{lem}

\begin{proof}
By definition, $\vartheta_1$, 
$\vartheta_2$ and
$\vartheta_3$ generate $\cB$ as
a $\Z[q^{\pm}][NE(Y)]$-algebra. 
From the commutation relations 
\eqref{eq_commutator_1}-\eqref{eq_commutator_2}-\eqref{eq_commutator_3},
we deduce that the monomials 
$\vartheta_1^a \vartheta_2^b \vartheta_3^c$ for
$a,b,c \geq 0$
are linear generators of $\cB$.
We can use \eqref{eq_cubic} to eliminate from 
$\vartheta_1^a \vartheta_2^b \vartheta_3^c$
the theta function with the smallest power. It follows that the 
monomials $n[p]$ are 
$\Z[q^{\pm}][NE(Y)]$-linear generators of 
$\tilde{A}^q$.
\end{proof}

By Lemma \ref{lem_commutators_theta}
and Lemma \ref{lem_cubic_equation_theta}, 
there exists a unique algebra morphism 
\begin{equation} \alpha \colon \cB \longrightarrow \cA_{\fD_\can} \,,
\end{equation}
such that $\alpha(\vartheta_j)= \vartheta_{v_j}$
for every $j \in \{1,2,3\}$.
In order to prove 
Theorem \ref{thm_eq_cubic}, 
it remains to show that $\alpha$ is an isomorphism.

By Lemma \ref{lem_m_linear_basis}, the quantum theta functions 
$\vartheta_{v_1}$, $\vartheta_{v_2}$ and
$\vartheta_{v_3}$ generate 
$\cA_{\fD_\can}$ as 
$\Z[q^{\pm}][NE(Y)]$-algebra, and so 
$\alpha$ is surjective.
It remains to show that $\alpha$ is injective. Let $b \in \cB$
with $\alpha(b)=0$. By Lemma
\ref{lem_tilde_m_linear_generators}, we can write $b$ as a 
$\Z[q^{\pm}][NE(Y)]$-linear combination 
$b= \sum_p b_p n[p]$. As $\alpha(n[p])=m[p]$, we have 
$\alpha(b)= \sum_b b_p m[p]$. 
By Lemma \ref{lem_m_linear_basis}, 
$\{m[p]\}_{p \in B(\Z)}$ is a 
$\Z[q^{\pm}][NE(Y)]$-linear basis of $\cA_{\fD_\can}$, 
and so we deduce from 
$\sum_p b_p m[p]=0$ that $b_p=0$ for all $p$.
This concludes the proof of Theorem \ref{thm_eq_cubic}.

\section{Comparison of $\cA_{\fD_\can}$ and $\Sk_A(\bS_{0,4})$}
\label{section_comparison}

In this section, we end the proof of Theorems
\ref{thm:consistent_0_4} and \ref{thm_ring_isom}. 
In Section \ref{section_change_variables}, 
we collect a number of change of variables and algebraic identities, 
which are then used in Section \ref{sect_can_nu}
to compare the quantum scattering diagrams $\fD_\can$ and $\fD_{0,4}$, and to end the 
proof of Theorem \ref{thm:consistent_0_4}.
In Section \ref{section_end_proof}, we compare the algebras $\cA_{\fD_{0,4}}$
and $\Sk_A(\bS_{0,4})$, and we conclude the proof of Theorem 
\ref{thm_ring_isom}.

\subsection{Change of variables and identities}
\label{section_change_variables}
Let $L$ be the quotient of $A_1(Y)$ 
by the subgroup generated by $D_1, D_2, D_3$, 
and let $\nu \colon NE(Y) \rightarrow L$ be the quotient map.

Following \cite{gross2019cubic}, write 
\begin{equation}
\begin{aligned}
F_1 &\coloneqq H-E_{11}-E_{21}-E_{31} \,,\\
F_2 &\coloneqq H-E_{11}-E_{22}-E_{32}\,,\\
F_3 &\coloneqq H-E_{12}-E_{21}-E_{32}\,,\\
F_4 &\coloneqq H-E_{12}-E_{22}-E_{31} \,.
\end{aligned}
\end{equation}
If we take for $Y$ the (non-general) cubic surface 
obtained by blowing up the $6$ intersection points 
of a general configurations of four lines $L_1$,
$L_2$, $L_3$, $L_4$ in $\PP^2_{\kk}$, then 
$F_j$ is the class of the $(-2)$-curve 
given by the strict transform of $L_j$. 
Note that the $(-2)$-curves $F_1$, $F_2$, $F_3$ and $F_4$ are all disjoint.
For $1 \leq j \leq 4$, write 
\begin{equation} 
G_j \coloneqq \nu(F_j) \in L \,.
\end{equation}

\begin{lem} \label{lem_nu}
The image in $L$ by $\nu$ of the classes of the lines $L_{jk}$ are given as follows:
\begin{equation} 
\{\nu(L_{1m})\}_{1 \leq m \leq 8}
=\Big\{ \frac{1}{2}(\epsilon_1 G_1+ \epsilon_2 G_2)\,|\, \epsilon_1, \epsilon_2 \in \{ \pm 1\} \Big\}
\cup \Big\{ \frac{1}{2}(\epsilon_3 G_3+ \epsilon_4 G_4)\,|\, \epsilon_3, \epsilon_4 \in \{ \pm 1\} \Big\} \,,
\end{equation}
\begin{equation} 
\{\nu(L_{2m})\}_{1 \leq m \leq 8}
=\Big\{ \frac{1}{2}(\epsilon_1 G_1+ \epsilon_3 G_3)\,|\, \epsilon_1, \epsilon_3 \in \{ \pm 1\} \Big\}
\cup \Big\{ \frac{1}{2}(\epsilon_2 G_2+ \epsilon_4 G_4)\,|\, \epsilon_2, \epsilon_4 \in \{ \pm 1\} \Big\} \,,
\end{equation}
\begin{equation} 
\{\nu(L_{3m})\}_{1 \leq m \leq 8}
=\Big\{ \frac{1}{2}(\epsilon_1 G_1+ \epsilon_4 G_4)\,|\, \epsilon_1, \epsilon_4 \in \{ \pm 1\} \Big\}
\cup \Big\{ \frac{1}{2}(\epsilon_2 G_2+ \epsilon_3 G_3)\,|\, \epsilon_2, \epsilon_3 \in \{ \pm 1\} \Big\} \,.
\end{equation}
\end{lem}

\begin{proof}
Recalling that $D_1=H-E_{11}-E_{12}$, 
$D_2=H-E_{21}-E_{22}$ and 
$D_3=2H-E_{31}-E_{32}$, we check that 
\begin{equation}
\label{eq:nu_E}
\begin{aligned}
 \nu(E_{11})=-\frac{1}{2}(G_1+G_2) \,,\,\,
\nu(E_{12})=-\frac{1}{2}(G_3+G_4) \,, \\
\nu(E_{21})=-\frac{1}{2}(G_1+G_3) \,,\,\, \nu(E_{22})=-\frac{1}{2}(G_2+G_4) \,,\\
 \nu(E_{31})=-\frac{1}{2}(G_1+G_4) \,,\,\, 
\nu(E_{32})=-\frac{1}{2}(G_2+G_3) \,. \end{aligned}
\end{equation}
Similarly, as $H=D_1+D_2+D_3-\frac{1}{2}(F_1+F_2+F_3+F_4)$, we have 
\begin{equation} \label{eq:nu_H}
\nu(H)=-\frac{1}{2}(G_1+G_2+G_3+G_4) \,.
\end{equation}
It remains to apply $\nu$ to all the classes $L_{jm}$ 
expressed in terms of the classes $H$ and $E_{k\ell}$ by Equations
\eqref{eq:lines_1}-\eqref{eq:lines_2}-\eqref{eq:lines_3}.
\end{proof}

Define 
\begin{equation}
\label{eq:L_basis}
\begin{aligned}
e_1 &\coloneqq \frac{1}{2}(G_1+G_2) \,,\\
e_2 &\coloneqq \frac{1}{2}(G_1+G_3) \,,
\\
e_3 &\coloneqq \frac{1}{2}(G_1+G_4) \,,
\\
e_4 &\coloneqq \frac{1}{2}(G_1+G_2+G_3+G_4) \,.
\end{aligned}
\end{equation}

\begin{lem} \label{lem:L_basis}
The four elements 
$e_1$, $e_2$, $e_3$ and $e_4$ form a
$\Z$-linear basis of $L$.
\end{lem}

\begin{proof}
One checks that the subgroup generated by $D_1$, $D_2$ and $D_3$ in the
 free abelian group $A_1(Y)$ of rank $7$ is saturated of rank $3$, 
 and so the quotient $L$ is free of rank $4$.
 
Note that by Equation \eqref{eq:nu_E}, 
we have $e_1=\nu(-E_{11})$, $e_2=\nu(-E_{21})$, $e_3=\nu(-E_{31})$, 
and by Equation \eqref{eq:nu_H}, $e_4=\nu(-H)$, 
so we have indeed $e_1,e_2,e_3,e_4 \in L$. 
On the other hand, $H$ and $E_{ij}$ generate $A_1(Y)$, and as $E_{21}=
H-E_{11}-D_1$, $E_{22}=H-E_{21}-D_2$, 
$E_{32}=H-E_{31}-D_3$, we deduce that 
$e_1,e_2,e_3,e_4$ generate $L$. As 
$L$ is free of rank $4$, we obtain that $e_1,e_2,e_3,e_4$ 
indeed form a basis of $L$.
\end{proof}

\begin{remark}
Let $(-,-)$ be the unique symmetric bilinear form on $L$ such that $(G_j,G_j)=2$
for $1 \leq j \leq 4$, and $(G_j,G_k)=0$ for every $j \neq k$. Then 
$(L,(-,-))$ is isomorphic to the $D_4$ weight lattice in such a way that 
$\{\nu(L_{1m})\}_{1 \leq m \leq 8}$
(resp.\ $\{\nu(L_{2m})\}_{1 \leq m \leq 8}$ 
and $\{\nu(L_{3m})\}_{1 \leq m \leq 8}$)
is the set of weights of the irreducible fundamental 
(resp.\ left chiral spinor and right chiral spinor) representation of $\Spin(8)$. 
Physically, $(L,(-,-))$ is the lattice of flavour charges 
for the $\Spin(8)$
flavour symmetry group of the $\cN=2$ $N_f=4$ $SU(2)$ gauge theory.
\end{remark}

We view $L$ as a subgroup of the group 
$\frac{1}{2}L$, and the group algebra 
$\Z[A^{\pm}][L]$ as a subalgebra of the group algebra 
$\Z[A^{\pm}][\frac{1}{2}L]$. We also view 
$\Z[A^{\pm}][a_1,a_2,a_3,a_4]$
as a subalgebra of $\Z[A^{\pm}][\frac{1}{2}L]$ 
via the following identifications:
\begin{equation}
\label{eq:identif}
\begin{aligned}
a_1 = t^{\frac{G_1}{2}}
+t^{-\frac{G_1}{2}} \,,\\
a_2=t^{\frac{G_2}{2}}
+t^{-\frac{G_2}{2}} \,,\\
a_3=t^{\frac{G_3}{2}}
+t^{-\frac{G_3}{2}} \,,\\
a_4=t^{\frac{G_4}{2}}
+t^{-\frac{G_4}{2}}\,.
\end{aligned}
\end{equation}
Finally, recall that
we introduced the elements 
\[ R_{1,0}\,, R_{0,1}\,, R_{1,1}\,,
y \in \Z[A^{\pm}][a_1,a_2,a_3,a_4] \]
in Equations \eqref{eq:R} and \eqref{eq:y}. 
The elements 
$R_{1,0}$, $R_{0,1}$, $R_{1,1}$ and $y$ are algebraically independent over 
$\Z[A^{\pm}]$, and so we have the inclusion
\begin{equation}
\Z[A^{\pm}][R_{1,0}, R_{0,1}, R_{1,1}, y] \subset \Z[A^{\pm}][a_1,a_2,a_3,a_4] \,.
\end{equation}
The algebraic independence of 
$R_{1,0}$, $R_{0,1}$, $R_{1,1}$ follows from the more precise fact, 
proved in Appendix B of \cite{MR2649343},
that the morphism 
\begin{equation}
\begin{aligned}
\A^4 &\rightarrow \A^4 \,, \\
(a_1,a_2,a_3,a_4) &\mapsto 
(a_1a_2+a_3a_4,a_1a_3+a_2a_4,a_1a_4+a_2a_3,
a_1a_2a_3a_4+a_1^2+a_2^2+a_3^2+a_4^2-4) \,. 
\end{aligned}
\end{equation}
is a ramified cover of degree $24$.

\begin{prop}
\label{prop_key_identity}
Using the identifications 
\eqref{eq:identif}, 
the following identity holds between degree $8$ 
polynomials in the variable $x$ and with coefficients in $\Z[A^{\pm}][\frac{1}
{2}L]$: 
\begin{equation}
\label{eq:key_identity}
\begin{aligned}
\prod_{m=1}^8(1+t^{\nu(L_{1m})}x)
=
1+x^8+R_{1,0}(x+x^7)
+(y-A^4-2-A^{-4})(x^2+x^6) \\
+(R_{0,1}R_{1,1}-R_{1,0})(x^3+x^5)
+(R_{0,1}^2R_{1,1}^2-2y+2A^4+2+2A^{-4})x^4\,.
\end{aligned} 
\end{equation}
Similarly, $\prod_{m=1}^8(1+t^{\nu(L_{2m})}x)$ 
and $\prod_{m=1}^8(1+t^{\nu(L_{3m})}x)$
are given by the same expression up to 
cyclic permutation of $R_{1,0}$, $R_{0,1}$ and $R_{1,1}$.
\end{prop}

\begin{proof}
Using the definitions \eqref{eq:R} and 
\eqref{eq:y} of $R_{1,0}$, $R_{0,1}$, $R_{1,1}$ and $y$, 
we expand the left-hand side 
of \eqref{eq:key_identity} in terms of $a_1$, $a_2$, $a_3$ and $a_4$. 
We obtain
\begin{equation}
\begin{aligned}
1+x^8&+(a_1a_2+a_3a_4)(x+x^7)
+(a_1a_2a_3a_4+a_1^2+a_2^2+a_3^2+a_4^2-4)
(x^2+x^6) \\
&+(a_1^2a_3a_4
+a_2^2 a_3a_4
+a_1a_2a_3^2+a_1a_2a_4^2 
-a_1a_2-a_3a_4)(x^3+x^5)\\
&+(2a_1a_2a_3a_4+
a_1^2a_3^2+a_1^2a_4^2+a_2^2a_3^2+a_2^2 a_4^2
-2(a_1^2+a_2^2+a_3^2+a_4^2)+6)x^4 \,,
\end{aligned}
\end{equation}
which rather amazingly factors as 
\begin{equation}
(1+a_1a_2x+(a_1^2+a_2^2-2)x^2
+a_1a_2x^3+x^4)(1+a_3a_4x+(a_3^2+a_4^2-2)x^2
+a_3a_4x^3+x^4) \,.
\end{equation}
Using the identifications \eqref{eq:identif}, 
we check that 
\begin{equation}
1+a_1a_2x+(a_1^2+a_2^2-2)x^2
+a_1a_2x^3+x^4= \prod_{\substack{\epsilon_1 \in \{ \pm 1\}\\ \epsilon_2 \in \{\pm 1\}}}
(1+t^{\frac{1}{2}(\epsilon_1G_1+\epsilon_2 G_2)}x)
\end{equation}
and 
\begin{equation}
1+a_3a_4x+(a_3^2+a_4^2-2)x^2
+a_3a_4x^3+x^4= \prod_{\substack{\epsilon_3 \in \{ \pm 1\}\\ \epsilon_4 \in \{\pm 1\}}}
(1+t^{\frac{1}{2}(\epsilon_3 G_1+\epsilon_4 G_2)}x)\,.
\end{equation}
The result then follows from Lemma \ref{lem_nu}.
\end{proof}

\begin{cor} \label{cor:identif_form}
Using the identifications 
\eqref{eq:identif}, 
the following identities hold in
$\Z[A^{\pm}][\frac{1}{2}L]$: 
\begin{equation} \label{eq:identif_form}
R_{1,0}=\sum_{j=1}^8 
t^{\nu(L_{1j})} \,,\,\,
R_{0,1}=\sum_{j=1}^8 
t^{\nu(L_{2j})} \,,\,\,
R_{1,1}=\sum_{j=1}^8 
t^{\nu(L_{3j})} \,,
\end{equation}
\begin{equation} \label{eq:identif_form_2}
\sum_{1 \leq j<j' \leq 8}
t^{L_{1j}+L_{1j'}}
=
\sum_{1 \leq j<j' \leq 8}
t^{L_{2j}+L_{2j'}}
=\sum_{1 \leq j<j' \leq 8}
t^{L_{3j}+L_{3j'}}
=y-A^4-2-A^{-4} \,.
\end{equation}
\end{cor}

\begin{proof}
Equation \eqref{eq:identif_form}
(resp. \eqref{eq:identif_form_2}) 
follows from comparing the coefficients of $x$
(resp. $x^2$) in the identity 
\eqref{eq:key_identity}
of Proposition \ref{prop_key_identity}.
\end{proof}

\begin{cor} \label{cor_identity_wall}
The following identity holds:
\begin{equation} \label{eq_identity_wall}
\begin{aligned}
\frac{\prod_{m=1}^8(1+t^{L_{1m}}x)}{(1-A^{-4} x^2)(1- x^2)^2
(1-A^4x^2)} 
=
1+& \frac{R_{1,0}x(1+x^2)}{(1-A^{-4}x^2)(1-A^4x^2)}
+\frac{yx^2}{(1-A^{-4}x^2)(1-A^4x^2)} \\
+ &\frac{R_{0,1}R_{1,1}x^3(1+
R_{0,1}R_{1,1}x+x^2)}{(1-A^{-4}x^2)
(1-x^2)^2(1-A^4x^2)}\,.
\end{aligned}
\end{equation}
\end{cor}

\begin{proof}
The identity 
\eqref{eq:key_identity}
of Proposition \ref{prop_key_identity}
expresses the numerator of the right-hand side of \eqref{eq_identity_wall}. 
We expand this expression in powers of $R_{1,0}$, 
$R_{0,1}$, $R_{1,1}$, $y$, 
and we simplify the resulting coefficients with the denominator.
\end{proof}

\begin{remark} \label{rem}
The left-hand side of Equation \ref{eq_identity_wall}
has exactly the form expected from the BPS spectrum of the 
$\cN=2$ $N_f=4$ $SU(2)$ gauge theory at large values of 
$u$ on the Coulomb branch: 
for every $(m,n) \in \Z^2$ with $m$ and $n$ coprime, 
we have one vector multiplet of charge $(2m,2n)$, which corresponds to 
to the denominator of the left-hand of Equation \eqref{eq_identity_wall},
$8$ hypermultiplets of charge $(m,n)$,
which correspond to the numerator 
of the left-hand of Equation \eqref{eq_identity_wall},
and no other states of charge a multiple of $(m,n)$
\cite{MR1306869}. 
The states of 
charge $(2,0)$ and $(1,0)$ can be seen classically 
(as $W$-bosons and elementary quarks respectively), 
and the general states of charge $(2m,2n)$ 
and $(m,n)$ are obtained from them by $SL_2(\Z)$ S-duality.
The states of charge $(2m,2n)$ 
are in the trivial representation of the $\Spin(8)$
flavour symmetry group, whereas the states of charge 
$(m,n)$ are in the $8$-dimensional fundamental 
(resp.\ left chiral spinor and right chiral spinor) 
of the $\Spin(8)$ flavour symmetry group if $(m,n)=(1,0) \mod 2$
(resp.\ $(0,1) \mod 2$ and $(1,1) \mod 2$). 
The $SL_2(\Z)$ $S$-duality group acts
on the flavour representations via the triality action of $PSL_2(\Z/2\Z) \simeq S_3$
permuting the three irreducible $8$-dimensional representations of $\Spin(8)$
(fundamental, left chiral spinor and right chiral spinor).
\end{remark}

\subsection{The quantum scattering diagram $\nu(\fD_\can)$}

\label{sect_can_nu}

In Section \ref{section_canonical_scattering}, 
we introduced and studied the quantum scattering diagram $\fD_{\can}$ over 
$\Z[q^{\pm}][NE(Y)]$. Recall that we use the notation $q=A^4$, 
and from now on we view $\fD_{\can}$ 
as a quantum scattering diagram over $\Z[A^{\pm}][NE(Y)]$.
In Section \ref{section_change_variables}, 
we considered the quotient $L$ of 
$A_1(Y)$ by the subgroup generated by 
$D_1$, $D_2$, $D_3$, and the quotient map
$\nu \colon NE(Y) \rightarrow L$.  
Given $f$ a power series with coefficients in $\Z[A^{\pm}][NE(Y)]$, 
we define the power series $\nu(f)$ with
coefficients in $\Z[A^{\pm}][L]$ by applying $\nu$ to each coefficient. 

\begin{defn}
We denote by $\nu(\fD_\can)$ the quantum scattering diagram over $\Z[A^{\pm}][L]$ 
obtained from $\fD_{\can}$ by applying 
$\nu$ to the quantum rays:
\begin{equation}
\nu(\fD_{\can}) \coloneqq \{ \nu(\fd_{m,n})\,|\, (m,n) \in B(\Z)\,, \gcd(m,n)=1 \} \,,
\end{equation}
where, for every quantum ray $\fd_{m,n}
=((m,n), f_{\fd_{m,n}})$, 
\begin{equation}
\nu(f_{\fd_{m,n}})\coloneqq ((m,n), \nu(f_{\fd_{m,n}})) \,.
\end{equation}
\end{defn}

As in Section \ref{section_change_variables}, we view 
$\Z[A^{\pm}][a_1,a_2,a_3,a_4]$
as a subalgebra of $\Z[A^{\pm}][\frac{1}{2}L]$ via \eqref{eq:identif}, 
and we use the 
the elements 
\[ R_{1,0}\,, R_{0,1}\,, R_{1,1}\,,
y \in \Z[A^{\pm}][a_1,a_2,a_3,a_4] \]
defined by \eqref{eq:R} and \eqref{eq:y}.
Recall that we introduced the rational function $F(r,s,y,x)$ 
in Equation \eqref{eq:def_F}.

The following Proposition
\ref{prop_ray_nu}
computes the quantum ray $\nu(\fd_1)=\nu(\fd_{1,0})$ of 
$\nu(\fD_\can)$.

\begin{prop} \label{prop_ray_nu}
The quantum ray $\nu(\fd_1)=(v_1,\nu(f_{\fd_1}))$
satisfies
\begin{equation}
\nu(f_{\fd_1}) = F(R_{1,0},R_{0,1}R_{1,1}, y, z^{-v_1}) \,.
\end{equation}
\end{prop}

\begin{proof}
Equation \eqref{eq:ray_rho_j_contribution} 
of Proposition \ref{prop_ray_rho_j_contribution} gives a formula for
$f_{\fd_1}$. The result of applying $\nu$ is given 
by the identity \eqref{eq_identity_wall} in Corollary 
\ref{cor_identity_wall}. It remains to compare with the definition of
$F(r,s,y,x)$ in Equation \eqref{eq:def_F} to conclude.
\end{proof}

In the following Theorem
\ref{thm_nu_all_rays}, 
we compute all the quantum rays of $\nu(\fD_\can)$.

\begin{thm} \label{thm_nu_all_rays}
The quantum rays $\nu(\fd_{m,n})$ 
of the quantum scattering diagram $\nu(\fD_\can)$
are given as follows. For every 
$(m,n) \in B_0(\Z)$ with $m$ and $n$ coprime, 
\begin{enumerate}
\item if $(m,n)=(1,0) \mod 2$, then
$ \nu(f_{\fd_{m,n}})= F(R_{1,0}, R_{0,1}R_{1,1},y,z^{-(m,n)})$,
\item if $(m,n)=(0,1) \mod 2$, 
then
$ \nu(f_{\fd_{m,n}})= F(R_{0,1}, R_{1,0}R_{1,1},y,z^{-(m,n)})$,
\item if $(m,n)=(1,1) \mod 2$,
then  
$ \nu(f_{\fd_{m,n}})
\coloneqq F(R_{1,1}, R_{1,0}R_{0,1},y,z^{-(m,n)})$.
\end{enumerate}
\end{thm}

\begin{proof}
In
Section \ref{section_contribution_general_rays}, 
we expressed a general quantum
ray $\fd_{m,n}$ of $\fD_\can$ in terms 
of the quantum ray $\fd_{1,0}$ and
a $PSL_2(\Z)$-symmetry acting on curve classes. 
We will show that after applying the quotient map 
$\nu$, the $PSL_2(\Z)$-symmetry simplifies dramatically.

The transformation $S^{*}$ of $A_1(Y)$
is given by Equation \eqref{eq:S_action}.
We have $S^{*}(D_1)=S^{*}(D_2)$,
$S^{*}(D_2)=S^{*}(D_3)$
and $S^{*}(D_3)=S^{*}(D_1)$. Therefore, 
$S^{*}$ preserves the subgroup
of $A_1(Y)$ generated by $D_1$, $D_2$, 
$D_3$, and so defines a transformation
of the quotient $L$, that we still denote by $S^{*}$.
Computing the action of $S^{*}$ 
on the basis $e_1,e_2,e_3,e_4$ of $L$ 
given by Equation \eqref{eq:L_basis} and Lemma \ref{lem:L_basis}, we find 
\begin{equation}
S^{*}(e_1)=e_2 \,,S^{*}(e_2)=e_3\,,
S^{*}(e_3)=e_1\,,S^{*}(e_4)=e_4 \,.
\end{equation}
In particular, $S^{*} \colon L \rightarrow L$ is a bijection.

The transformation $T^{*}$ of 
$A_1(Y)$ is given by Equation \eqref{eq:T_action}. We have 
$T^{*}(D_1)=D_1+D_3$, 
$T^{*}(D_2)=0$ and 
$T^{*}(D_3)=D_2+D_3$.
Therefore, $T^{*}$ preserves the 
subgroup of $A_1(Y)$ generated by 
$D_1$, $D_2$, $D_3$, and so defines 
a transformation of the quotient $L$, that we still denote by 
$T^{*}$.
Computing the action of $T^{*}$ on the basis $e_1,e_2,e_3,e_4$ 
of $L$ given by Equation \eqref{eq:L_basis} and Lemma \ref{lem:L_basis}, we find 
\begin{equation}
T^{*}(e_1)=e_1\,,
T^{*}(e_2)=e_4-e_3\,,
T^{*}(e_3)=e_2\,,
T^{*}(e_4)=e_4\,.
\end{equation}
In particular, $T^{*} \colon 
L \rightarrow L$ is a bijection.

Therefore, $S^{*}$ and $T^{*}$ on $L$ 
defines an action of $PSL_2(\Z)$ on $L$ and so on $\Z[A^{\pm}][L]$
and 
$\Z[A^{\pm}][\frac{1}{2} L]$. Computing the action of $S^{*}$ on 
$G_1, G_2, G_3, G_4$, we find
\begin{equation}
S^{*}(G_1)=G_1\,, S^{*}(G_2)
=G_3 \,,
S^{*}(G_3)=G_4 \,,
S^{*}(G_4)=G_2 \,,
\end{equation}
and so \begin{equation}
S^{*}(a_1)=a_1 \,, S^{*}(a_2)=a_3 \,,
S^{*}(a_3)=a_4 \,,S^{*}(a_4)=a_2\,,
\end{equation}
\begin{equation} \label{S_R_action}
S^{*}(R_{1,0})=R_{0,1} \,, 
S^{*}(R_{0,1})=R_{1,1} \,,
S^{*}(R_{1,1})=R_{1,0}\,,
S^{*}(y)=y\,.
\end{equation}
Computing the action of $T^{*}$ on 
$G_1, G_2, G_3, G_4$, we find
\begin{equation}
\begin{aligned}
T^{*}(G_1)&=\frac{1}{2}(G_1+G_2+G_3-G_4) \,,\\
T^{*}(G_2)&=\frac{1}{2}
(G_1+G_2-G_3+G_4) \,,\\
T^{*}(G_3)&=\frac{1}{2}
(-G_1+G_2+G_3+G_4) \,,\\
T^{*}(G_4)&=\frac{1}{2}
(G_1-G_2+G_3+G_4) \,.
\end{aligned}
\end{equation}
and then 
\begin{equation} \label{eq:T_R_action}
T^{*}(R_{1,0})=R_{1,0} \,,
T^{*}(R_{0,1})=R_{1,1}\,,
T^{*}(R_{1,1})=R_{0,1} \,,
T^{*}(y)=y \,.
\end{equation}
From Equations \eqref{eq:S_action} and 
\eqref{eq:T_R_action}, we see that $PSL_2(\Z)$ 
acts trivially on $y$, and acts on 
$R_{1,0}$, $R_{0,1}$ and $R_{1,1}$ through its finite quotient 
$PSL_2(\Z/2\Z)$ acting on indices $m,n$ of $R_{m,n}$ 
viewed as integers modulo $2$. 
Recalling that $PSL_2(\Z/2 \Z)$ is isomorphic to the symmetric group $S_3$ 
of permutations of a set with three elements, $S^{*}$ acts on $\{R_{1,0},R_{0,1}R_{1,1}\}$ 
as a cyclic permutation, whereas $T^{*}$ acts as a transposition.

We can now end the proof of Theorem \ref{thm_nu_all_rays}. 
For every 
$(m,n) \in B_0(\Z)$ with $m$ and $n$, coprime, there exists $M \in SL_2(\Z)$ 
such that $M(m,n)=(1,0)$. By 
Proposition \ref{prop_symmetry}, we have $\fd_{m,n}=M(\fd_{1,0})$ and so
$\nu(\fd_{m,n})=M(\nu(\fd_{1,0}))$. 
The result then follows from Proposition 
\ref{prop_ray_nu} computing $\nu(\fd_{1,0})$, 
and from the above description of the action 
of $PSL_2(\Z)$ on $R_{1,0}$, $R_{0,1}$, $R_{1,1}$ and $y$ 
through the finite quotient  $PSL_2(\Z/2\Z)$.
\end{proof}

In Section \ref{section_scattering_04}, 
we defined the quantum scattering diagram $\fD_{0,4}$ over 
$\Z[A^{\pm}][R_{1,0},R_{0,1},R_{1,1},y]$.
By Theorem \ref{thm_nu_all_rays}, 
the quantum scattering diagram
$\nu(\fD_\can)$, which is a priori defined over 
$\Z[A^{\pm}][L]$, can be
viewed as a quantum scattering diagram over  $\Z[A^{\pm}][R_{1,0},R_{0,1},R_{1,1},y]$.

\begin{cor} \label{cor_eq_scat}
We have the equality 
$\fD_{0,4}=\nu(\fD_\can)$ of quantum scattering diagrams over 
$\Z[A^{\pm}][R_{1,0},R_{0,1},R_{1,1},y]$.
\end{cor}

\begin{proof}
This follows from comparing the description of $\nu(\fD_\can)$
given by Theorem \ref{thm_nu_all_rays}
with the Definition 
\ref{defn_scat_04}
of $\fD_{0,4}$.
\end{proof}

We can now end the proof of Theorem
\ref{thm:consistent_0_4}. By Theorem 
\ref{thm_can_consistent}, the quantum scattering diagram $\fD_\can$ is consistent, 
and so in particular, applying the quotient map $\nu$, 
the quantum scattering diagram $\nu(\fD_\can)$ is also consistent. Therefore, 
$\fD_{0,4}$ is consistent by Corollary 
\ref{cor_eq_scat}.

\subsection{End of the proof of positivity for $\Sk_A(\bS_{0,4})$}
\label{section_end_proof}

In the previous Section \ref{sect_can_nu}, we proved that 
$\fD_{0,4}=\nu(\fD_\can)$ and so in particular that $\fD_{0,4}$. 
Let $\cA_{\fD_{0,4}}$ be the corresponding 
$\Z[A^{\pm}][R_{1,0}, R_{0,1}, R_{1,1}, y]$-algebra given by Definition
\ref{def:algebra}, with its basis 
$\{ \vartheta_p \}_{p \in B(\Z)}$ of quantum theta functions.
Recall from Section \ref{section_stronger_04} that the isotopy classes 
$\{ \gamma_p\}_{p \in B(\Z)}$ of 
multicurves without peripheral components on $\bS_{0,4}$ 
form a basis of $\Sk_A(\bS_{0,4})$ as
\linebreak
$\Z[A^{\pm}][a_1,a_2,a_3,a_4]$-module, and that the bracelets basis is 
$\{ \mathbf{T}(\gamma_p)\}_{p \in B(\Z)}$.

In the present section, we prove Theorem \ref{thm_ring_isom}, 
that is, we will construct a morphism 
$\varphi \colon \cA_{\fD_{0,4}} 
\rightarrow \Sk_A(\bS_{0,4})$
of $\Z[A^{\pm}][R_{1,0},R_{0,1}, R_{1,1},y]$-algebras
such that $ \varphi(\vartheta_p)=T(\gamma_p)$ 
for every $p \in B(\Z)$, and which becomes an isomorphism of 
$\Z[A^{\pm}][a_1,a_2,a_3,a_4]$-algebras  
after extension of scalars for  $\cA_{\fD_{0,4}} $ from 
$\Z[A^{\pm}][R_{1,0},R_{0,1}, R_{1,1},y]$ to $\Z[A^{\pm}][a_1,a_2,a_3,a_4]$.

Bullock and Przytycki gave in \cite[Theorem 3.1]{MR1625701} 
the following presentation of $\Sk_A(\bS_{0,4})$.

\begin{thm}(\cite[Theorem 3.1]{MR1625701}) \label{thm:sk_cubic}
The $\Z[A^{\pm}][a_1,a_2,a_3,a_4]$-algebra
$\Sk_A(\bS_{0,4})$ admits the following presentation by generators and relation:
$\Sk_A(\bS_{0,4})$ is generated by 
$\gamma_{v_1}$, $\gamma_{v_2}$, $\gamma_{v_3}$, with the relations
\begin{equation} \label{eq:sk_comm_1}
 A^{-2} \gamma_{v_1}
\gamma_{v_2} 
-A^2 
\gamma_{v_2}\gamma_{v_1} 
=(A^{-4}-A^4)\gamma_{v_3}
-(A^2-A^{-2}) R_{1,1} \,,
\end{equation}

\begin{equation} \label{eq:sk_comm_2}
 A^{-2} \gamma_{v_2}
\gamma_{v_3} 
-A^2 
\gamma_{v_3}\gamma_{v_2} 
=(A^{-4}-A^4)\gamma_{v_1}
-(A^2-A^{-2}) R_{1,0}\,,
\end{equation}

\begin{equation} \label{eq:sk_comm_3}
 A^{-2} \gamma_{v_3}
\gamma_{v_1} 
-A^2 
\gamma_{v_1}\gamma_{v_3} 
=(A^{-4}-A^4)\gamma_{v_2}
-(A^2-A^{-2}) R_{0,1}\,,
\end{equation}

\begin{equation} \label{eq:sk_cubic}
A^{-2}
\gamma_{v_1} 
\gamma_{v_2}
\gamma_{v_3}
=A^{-4} \gamma_{v_1}^2 
+A^4 \gamma_{v_2}^2 
+A^{-4} \gamma_{v_3}^2
+
A^{-2} R_{1,0}  \gamma_{v_1}
+ A^2 R_{0,1}\gamma_{v_2}
+A^{-2} R_{1,1} \gamma_{v_3}+y-2(A^4+A^{-4})\,.
\end{equation}
\end{thm}

Note that in Theorem \ref{thm:sk_cubic}, 
we use the generators $\gamma_{v_1}$, $\gamma_{v_2}$, $\gamma_{v_3}$, 
whereas the generators $\gamma_{v_1}$, $\gamma_{v_2}$, $\gamma_{v_1+v_2}$ 
are used in \cite[Theorem 3.1]{MR1625701}.
Using $\gamma_{v_3}$ rather than 
$\gamma_{v_1+v_2}$ has for unique effect on the equations to replace $A$ by $A^{-1}$. 

On the other hand, applying the quotient map 
$\nu$ to the presentation of $\fD_{\can}$ given by Theorem \ref{thm_eq_cubic}, 
and using the identities
\eqref{eq:identif_form} and \eqref{eq:identif_form_2}
given by Corollary 
\ref{cor:identif_form},
we obtain the following presentation of $\cA_{\fD_{0,4}}$.

\begin{thm} \label{thm_eq_cubic_nu}
The $\Z[A^{\pm}][R_{1,0},R_{0,1},R_{1,1},y]$-algebra
$\cA_{\fD_{0,4}}$ admits the following presentation by generators and relation:
$\cA_{\fD_\can}$ is generated by 
$\vartheta_{v_1}$, $\vartheta_{v_2}$, $\vartheta_{v_3}$, with the relations
\begin{equation} \label{eq:comm_1_nu}
 A^{-2} \vartheta_{v_1}
\vartheta_{v_2} 
-A^2
\vartheta_{v_2}\vartheta_{v_1} 
=(A^{-4}-A^4)\vartheta_{v_3}
-(A^2-A^{-2})R_{1,1} \,,
\end{equation}
\begin{equation} \label{eq:comm_2_nu}
A^{-2} \vartheta_{v_2}
\vartheta_{v_3} 
-A^2 
\vartheta_{v_3} \vartheta_{v_2} 
=(A^{-4}-A^4)\vartheta_{v_1}
-(A^2-A^{-2})R_{1,0}\,,
\end{equation}

\begin{equation} \label{eq:comm_nu}
A^{-2} \vartheta_{v_3}
\vartheta_{v_1} 
-A^2 
\vartheta_{v_1} \vartheta_{v_3} 
=(A^{-4}-A^4)\vartheta_{v_2}
-(A^2-A^{-2}) R_{0,1}\,,
\end{equation}

\begin{equation} \label{eq:cubic_nu}
A^{-2}
\vartheta_{v_1} 
\vartheta_{v_2}
\vartheta_{v_3}
=A^{-4}  \vartheta_{v_1}^2 
+A^4\vartheta_{v_2}^2 
+A^{-4} \vartheta_{v_3}^2
+
A^{-2}R_{1,0} 
\vartheta_{v_1}  
+ A^2R_{1,0}
\vartheta_{v_2}
+ A^{-2} R_{1,1}
\vartheta_{v_3}
+
y-2(A^4+A^{-4})\,.
\end{equation}
\end{thm}

Comparing Theorems \ref{thm:sk_cubic} and \ref{eq:cubic_nu}, 
we obtain that there exists a unique morphism 
\begin{equation}\varphi \colon \cA_{\fD_{0,4}} 
\longrightarrow \Sk_A(\bS_{0,4})
\end{equation}
of $\Z[A^{\pm}][R_{1,0},R_{0,1}, R_{1,1},y]$-algebras
such that $ \varphi(\vartheta_{v_j})=\gamma_{v_j}$ for $j \in \{1,2,3\}$, 
and moreover that $\varphi$ becomes an isomorphism of 
$\Z[A^{\pm}][a_1,a_2,a_3,a_4]$-algebras  
after extension of scalars for  $\cA_{\fD_{0,4}} $ from 
$\Z[A^{\pm}][R_{1,0},R_{0,1}, R_{1,1},y]$ to $\Z[A^{\pm}][a_1,a_2,a_3,a_4]$.
Therefore, to conclude the proof of Theorem \ref{thm_ring_isom}, 
it remains to show the following result.

\begin{thm} \label{thm:basis_comparison}
For every $p \in B(\Z)$, we have 
\begin{equation}
\varphi(\vartheta_p) = \mathbf{T}(\gamma_p)\,.
\end{equation}
\end{thm}

\begin{proof}
We first prove that, for every $k \geq 0$, we have 
\begin{equation} \label{eq:v1}
\varphi(\vartheta_{kv_1})
=\mathbf{T}(\gamma_{kp_1})\,.
\end{equation} 
The isotopy class $\gamma_{kp_1}$ is the class of $k$ 
disjoint curves isotopic to $\gamma_1$, 
and so $\mathbf{T}(\gamma_{kp_1})=T_k(\gamma_{kp_1})$.
Recall that the Chebyshev polynomials
$T_k(x)$ are defined by $T_0(x)=1$, $T_1(x)=x$, $T_2(x)=x^2-2$, 
and for every $k\geq 2$, 
$T_{k+1}(x)=xT_k(x)-T_{k-1}(x)$. 

We prove that $\varphi(\vartheta_{kv_1})
=\mathbf{T}(\gamma_{kp_1})$ for every $k \geq 0$ by induction on $k$. 
The result holds trivially for $k=0$ as $\vartheta_0=1$ and $\mathbf{T}(\gamma_0)=1)$. 
It holds for $k=1$ by construction of $\varphi$: 
$\varphi(\vartheta_{v_1})=\gamma_{v_1}
=T_1(\gamma_{v_1})$.
It also holds for $k=2$: using Lemma
\ref{lem_square_theta}, we have 
\begin{equation}
\varphi(\vartheta_{2v_1})
=\varphi(\vartheta_{v_1}^2-2)
=\varphi(\vartheta_{v_1})^2-2
=\gamma_{v_1}^2-2=T_2(\gamma_{v_1}) \,.
\end{equation}
Let $k \geq 2$ and assume that the result holds for all $k' \leq k$. 
Then, using 
Lemma \ref{lem_power_theta}, we have
\begin{equation}
\begin{aligned} \varphi( \vartheta_{(k+1)v_1})
&=\varphi(\vartheta_{v_1}\vartheta_{kv_1}-\vartheta_{(k-1)v_1})
= \varphi(\vartheta_{v_1})\varphi(\vartheta_{kv_1})-\varphi(\vartheta_{(k-1)v_1}) \\
&=\gamma_{v_1}T_k(\gamma_{v_1})
-T_{k-1}(\gamma_{v_1})=T_{k+1}(\gamma_{v_1}) \,,
\end{aligned}
\end{equation}
and so the result holds for $k+1$.

We now explain how to deduce the result for general $p \in B(\Z)$ 
from the result for $p=kv_1$ using $PSL_2(\Z)$-symmetry. 
In order to simplify the notation, we write $R$ for
$\Z[A^{\pm}][R_{1,0},R_{0,1},R_{1,1},y]$. 
Recall from the proof of Theorem \ref{thm_nu_all_rays} that 
$PSL_2(\Z)$ acts through its finite quotient $PSL_2(\Z/2\Z)$ on
$R$
by $\Z[A^{\pm}]$-algebra automorphisms permuting 
$R_{1,0}$, $R_{0,1}$, $R_{1,1}$, and fixing $y$. 
We define  below actions of $PSL_2(\Z)$ on 
$\cA_{\fD_{0,4}}$ and $\Sk_A(\bS_{0,4})$ lifting the action on $R$.

For every $M \in SL_2(\Z)$, we define a lift $\Psi_M$ to $\cA_{\fD_{0,4}}$ 
of the action of $M$ on $R$ by 
\begin{equation}\label{eq:psi}
\Psi_M(\vartheta_p) \coloneqq \vartheta_{Mp} \,,
\end{equation}
where $p \mapsto Mp$ is the action of $PSL_2(\Z)$ on $B(\Z)$. 
We claim that $\Psi_M$ is an automorphism of $\cA_{\fD_{0,4}}$ as $\Z[A^{\pm}]$-algebra.
Indeed, the Definition \ref{defn_scat_04} of $\fD_{0,4}$ 
has the following manifest $PSL_2(\Z)$-symmetry: 
for every $M \in PSL_2(\Z)$ and $p=(m,n) \in B(\Z)$ 
with $m$ and $n$ coprime, 
the function attached to the quantum ray  $\rho_{Mp}$ is obtained
by applying the action of $M \in PSL_2(\Z)$ on $R$ 
to the coefficients of the function attached to the quantum ray 
$\rho_p$. The compatibility of $\Psi_M$ with the product structure of $\cA_{\fD_{0,4}}$ 
then follows from the Definition \ref{def:algebra} of the product of $\cA_{\fD_{0,4}}$
in terms of quantum broken lines for 
$\fD_{0,4}$. Thus, $M \mapsto \Psi_M$ defines an action of $PSL_2(\Z)$ on 
$\cA_{\fD_{0,4}}$ by automorphisms of $\Z[A^{\pm}]$-algebras lifting the action on $R$. 

On the other hand, given the geometric definition of the skein algebra, 
there is a natural action of the mapping class group $MCG(\bS_{0,4})$ 
of $\bS_{0,4}$ on $\Sk_A(\bS_{0,4})$ by 
automorphisms of $\Z[A^{\pm}]$-algebras. 
Recall that the mapping class group is the group of isotopy classes of 
orientation-preserving diffeomorphisms. 
The mapping class group $MCG(\bS_{0,4})$ 
contains a natural subgroup isomorphic to $PSL_2(\Z)$,  
which is coming from the description of $\bS_{0,4}$ as a quotient of a $4$-punctured torus
by an involution, and from the fact that 
the mapping class group of the torus is $SL_2(\Z)$. 
In fact  $MCG(\bS_{0,4})$ is a semi-direct product of 
$PSL_2(\Z)$ with $\Z/2\Z \times \Z/2\Z$ (see e.g.\ Section
2.2.5 of \cite{MR2850125}). 
The action of $PSL_2(\Z)$ on 
$\Sk_A(\bS_{0,4})$ is reviewed at the beginning of 
Section 4 of \cite{bakshi2018multiplying}: this action
$M \mapsto \Phi_M$ lifts the action of $PSL_2(\Z)$ on $R$ and satisfies 
\begin{equation} \label{eq:phi}
\Phi_M(\gamma_p) = \gamma_{Mp} \,,
\end{equation} 
for every $M \in PSL_2(\Z)$ and $p \in B(\Z)$. 

We claim that $\varphi \colon \cA_{\fD_{0,4}} \rightarrow \Sk_A(\bS_{0,4})$ 
intertwines between the actions $\Psi$ and $\Phi$ of $PSL_2(\Z)$ on 
$\cA_{\fD_{0,4}}$ and $\Sk_A(\bS_{0,4})$, that is 
\begin{equation}\label{eq:intertw}
\varphi \circ \Psi_M= \Phi_M \circ \varphi
\end{equation} for every $M \in PSL_2(\Z)$. 
It is enough to check it for the generators $S$ and $T$ of $PSL_2(\Z)$ given in \eqref{eq:ST}. 
The result is clear for $S$: we have $Sv_1=v_2$, $Sv_2=v_3$, $Sv_3=v_1$, 
and so $\varphi \circ \Psi_S(\vartheta_{v_j})=\Phi_S(\gamma_{v_j})$ 
for $j \in \{1,2,3\}$ follows by combining Equations 
\eqref{eq:psi} and \eqref{eq:phi}. 
Similarly, we have $T(v_1)=v_1$, $T_{v_2}=v_1+v_2$, $Tv_3=v_2$, 
so $\varphi \circ \Psi_T(\vartheta_{v_j})=\Phi_T(\gamma_{v_j})$ for
$j \in \{1,3\}$ follows by combining Equations 
\eqref{eq:psi} and \eqref{eq:phi}.
But we need an extra argument for $j=2$: 
one needs to show that $\varphi(\vartheta_{v_1+v_2})
=\gamma_{v_1+v_2}$. This follows from the fact that 
\begin{equation}
A^2 \vartheta_{v_1+v_2} 
=\vartheta_{v_1} \vartheta_{v_2} 
- A^{-2} \vartheta_{v_3}
- R_{1,1} \,,
\end{equation}
in $\cA_{\fD_{0,4}}$ by Lemma \ref{lem_product_theta_1}
and 
\begin{equation}
A^2 \gamma_{v_1+v_2} 
=\gamma_{v_1} \gamma_{v_2} 
-A^{-2} \gamma_{v_3}
- R_{1,1} \,,
\end{equation}
in $\Sk_A(\bS_{0,4})$ by the formula above Equation (2.5) in \cite{bakshi2018multiplying}.

We can now end the proof of Theorem \ref{thm:basis_comparison}. 
Let $p \in B(\Z)$. There exists $M \in PSL_2(\Z)$ and $k \in \Z_{\geq 0}$ 
such that $p=M(kv_1)$. Then, 
\begin{equation} \varphi(\vartheta_p)
=\varphi(\vartheta_{M(kv_1)})
=\varphi(\Psi_M(\vartheta_{kv_1}))
=\Phi_M(\varphi(\vartheta_{kv_1}))
\end{equation}
\[=\Phi_M(\mathbf{T}(\gamma_{kv_1}))
=\mathbf{T}(\Phi_M(\gamma_{kv_1}))
=\mathbf{T}(\gamma_{M(kv_1)})
=\mathbf{T}(\gamma_p) \,,\]
where we use successively \eqref{eq:psi},
\eqref{eq:intertw}, \eqref{eq:v1}, 
the fact that $\Phi_M$ is an algebra automorphism,
and \eqref{eq:phi}.
\end{proof}

\newcommand{\etalchar}[1]{$^{#1}$}

\vspace{+8 pt}
\noindent
Institute for Theoretical Studies \\
ETH Zurich \\
8092 Zurich, Switzerland \\
pboussea@ethz.ch

\end{document}